\documentclass[a4paper]{article} %

\usepackage[margin=3cm]{geometry}

\usepackage{auslander-packages}

\title{Simplicial structures in\\ higher Auslander--Reiten theory}

\author{Tobias Dyckerhoff\footnote{Universit\"{a}t Hamburg, Fachbereich
    Mathematik, Bundesstra{\ss}e 55, 20146 Hamburg, Germany, email: {\tt
      tobias.dyckerhoff@uni-hamburg.de} }, Gustavo Jasso\footnote{Rheinische
    Friedrich-Wilhelms-Universit\"{a}t Bonn, Mathematisches Institut, Endenicher
    Allee 60, 53115 Bonn, Germany, email: {\tt gjasso@math.uni-bonn.de}} and
  Tashi Walde\footnote{Rheinische Friedrich-Wilhelms-Universit\"{a}t Bonn,
    Mathematisches Institut, Endenicher Allee 60, 53115 Bonn, Germany, email:
    {\tt twalde@math.uni-bonn.de} }}

\begin{document}

\maketitle
\vspace{-2em}
\begin{center}
  \emph{Dedicated to the memory of Thomas Poguntke}
\end{center}
\vspace{2em}

\begin{abstract}
  We develop a novel combinatorial perspective on the higher Auslander algebras
  of type $\AA$, a family of algebras arising in the context of Iyama's higher
  Auslander--Reiten theory. This approach reveals interesting simplicial
  structures hidden within the representation theory of these algebras and
  establishes direct connections to Eilenberg--MacLane spaces and
  higher-dimensional versions of Waldhausen's $\eS_\bullet$-construction in
  algebraic $K$-theory. As an application of our techniques we provide a
  generalisation of the higher reflection functors of Iyama and Oppermann to
  representations with values in stable $\infty$-categories. The resulting
  combinatorial framework of slice mutation can be regarded as a
  higher-dimensional variant of the abstract representation theory of type $\AA$
  quivers developed by Groth and \v{S}\v{t}ov\'\i\v{c}ek. Our simplicial point
  of view then naturally leads to an interplay between slice mutation, horn
  filling conditions, and the higher Segal conditions of Dyckerhoff and
  Kapranov. In this context, we provide a classification of higher Segal objects
  with values in any abelian category or stable $\infty$-category.
\end{abstract}

\setcounter{tocdepth}{2}
\tableofcontents

\section*{Introduction}

In the context of higher Auslander--Reiten theory \cite{Iya07}, Iyama introduced a
remarkable family of finite-dimensional algebras $A_\ell^{(m)}$, parameterised by natural numbers
$\ell$ and $m$, called \emph{higher Auslander algebras of type $\AA$} \cite{Iya11}. For
a fixed positive integer $\ell$, these algebras are obtained by a recursive construction that starts
with the hereditary algebra $A_\ell^{(1)}$ given by the path algebra of the quiver
\[
	1 \to 2 \to \cdots \to \ell
\]
and proceeds inductively by setting
\[
  A_\ell^{(m+1)}\coloneqq \End_{A_\ell^{(m)}}(M)^{\op}
\]
where $M$ is a so-called $m$-cluster tilting $A_\ell^{(m)}$-module. The module $M$ is the direct sum
of a skeleton of indecomposable objects in the $m$-cluster tilting subcategory
\[
	\M_{\ell}^{(m)} \subseteq \mmod A_{\ell}^{(m)},
\]
which is an $m$-abelian category in the sense of \cite{Jas16}.

The following results, proven in \S \ref{sec:1}, are the starting point of this work:
\begin{enumerate}
	\item Let $B$ be an abelian group. Then, for every $n\geq m\geq 1$, there is a canonical isomorphism
		\begin{equation}\label{eq:EMI}
				\Hom(K_0(\M_{n-m+1}^{(m)}), B) \cong \K(B,m)_n
		\end{equation}
		where $\K(B,m)_{\bullet}$ denotes an Eilenberg-MacLane space given in the form of its standard
		model as a simplicial abelian group. The relevance of the simplicial object
		$\K(B,m)_{\bullet}$ in topology is, of course, its role as a classifying space for
		cohomology: For every CW complex $Y$, there is a bijection
		\[
			[Y, |\K(B,m)_{\bullet}|] \cong H^m(Y,B).
		\]
		between the set of homotopy classes of maps from $Y$ to $|\K(B,m)_{\bullet}|$ and the
		$m$th cohomology group of $Y$ with coefficients in $B$.
	\item Let $\B$ be an abelian category. Then, for every $n\geq m\geq 1$, there is a canonical
		equivalence of categories
		\begin{equation}\label{eq:SMI}
				\Fun^{\ex}(\M_{n-m+1}^{(m)}, \B) \simeq \Sm_{n}(\B)
		\end{equation}
		where $\Sm_{\bullet}(\B)$ denotes the $m$-dimensional Waldhausen
		$\eS_{\bullet}$-construction of $\B$ as studied for $m=1$ in \cite{Wal85}, for
		$m=2$ in \cite{HM15}, and for all $m \ge 0$ in \cite{Pog17}. The relevance of these
		simplicial objects in algebraic $K$-theory is their role in describing the
		deloopings of the $K$-theory spectrum of $\B$: We have homotopy equivalences
		\[
			\Omega^m |\Sm_{\bullet}(\B)^{\simeq}| \simeq K(\B)
		\]
		exhibiting the $K$-theory space $K(\B)$ of $\B$ as an $m$-fold loop space, and,
		varying $m \ge 1$, as an $\Omega$-spectrum (see \cite{Gre07} for a survey on
    spectra).
\end{enumerate}

The identifications \eqref{eq:EMI} and \eqref{eq:SMI} suggest that it is possible and natural to
organise the various categories $\M_{\bullet-m+1}$ into a simplicial object. The goal of this
article is to introduce a suitable combinatorial framework to achieve this and to study the resulting
interplay between representation theoretic concepts from higher Auslander--Reiten
theory and combinatorial concepts from the theory of simplicial structures.

We provide an overview of the contents of this work:

\subsubsection*{\ref{sec:1}\quad\nameref{sec:1}}

In this section, we develop a combinatorial approach to the higher Auslander algebras of
type $\AA$ in terms of simplicial combinatorics: We provide a description of these algebras in terms
of the posets $\Delta(m,n)$ of monotone functions between the standard ordinals $[m]$ and $[n]$.
This approach allows us to give a precise relationship between these algebras and fundamental
objects in algebraic topology and algebraic $K$-theory: Eilenberg--MacLane spaces and
higher versions of the Waldhausen $\eS_\bullet$-construction, respectively.

\subsubsection*{\ref{sec:reflection}\quad\nameref{sec:reflection}}

In this section, we develop a higher categorical approach to higher reflection
functors based on the combinatorial description of the higher Auslander algebras
of type $\AA$ from \th\ref{prop:comaus}. This generalises the more traditional
approach via tilting modules. A profitable side effect of our theory is that it
allows us to study representations with values in arbitrary stable
$\infty$-categories, not necessarily linear over a field, much in the spirit of
the abstract representation theory of Groth and \v{S}\v{t}ov\'\i\v{c}ek. Indeed,
these results can be seen as higher-dimensional generalisations of results in
\cite{GS16} and \cite{GS16b}.

Our perspective on higher reflection functors brings a certain geometric appeal: Iyama's
homological condition of ``regularity in dimension $m+1$'' is reflected geometrically by the
requirement that representations of the poset $\Delta(m,n)$ map all rectilinear $(m+1)$-dimensional
cubes to biCartesian cubes. The reflection functors are then obtained by puncturing these cubes at
their initial and final vertices, respectively.

We provide another instance of this phenomenon: We show that the derived category of coherent
sheaves on projective $m$-space admits a combinatorial description in terms of representations of
the Beilinson category $B_{\ZZ}^{(m)}$, with relations formulated again in terms of
$(m+1)$-dimensional biCartesian cubes.

These results seem to be precursors of an axiomatic framework for investigating $m$-cluster tilting
subcategories based on higher-dimensional biCartesian cubes in $\infty$-categories.

In the related framework of stable derivators, further examples of tilting
equivalences described in terms of posets of paracyclic simplices were recently
obtained by Beckert in his doctoral thesis \cite{Bec18}.
The abstract representation theory of cubical diagrams with various exactness
conditions has been thoroughly investigated by Beckert and Groth in
\cite{BG18} also in the related framework of stable derivators.

\subsubsection*{\ref{sec:horn}\quad\nameref{sec:horn}}

In this section, we explain how to interpret higher reflection functors
of type $\AA$ in terms of horn filling conditions on simplicial objects. As a
by-product of this observation we relate the notions of
\begin{enumerate}[label=(\roman*)]
\item slice mutation,
\item outer horn filling conditions, and
\item higher Segal conditions.
\end{enumerate}
The main ingredient in the above comparison is an approach to objects of membranes in the sense of
\cite{DK12} by means of a Cartesian fibration
\[
  \bigcup\colon\CovD\lra\Delta
\]
whose fibres parameterise the simplicial subsets of the standard simplices.

\subsubsection*{\ref{sec:dk}\quad\nameref{sec:dk}}

In this section, we provide a simple characterisation of so-called outer $m$-Kan complexes with
values in either an abelian category or a stable $\infty$-category. Our characterisation is given in
terms of appropriate versions of the Dold--Kan correspondence. As a consequence, we show that the
classes of outer $m$-Kan complexes and $2m$-Segal objects coincide in this context.

\subsubsection*{\ref{sec:n-cubes}\quad\nameref{sec:n-cubes}}

In this appendix we collect elementary results concerning $n$-cubes in stable
$\infty$-categories which are needed at various stages throughout the article.

\subsection*{Conventions}

Throughout the article we use freely the language of $\infty$-categories as
developed in \cite{Lur09,Lur17}. When there is no risk of confusion we identify
an ordinary category with its nerve. In particular, if $D$ is a small
category and $\C$ is an $\infty$-category, we denote the $\infty$-category of
functors $\NNN(D)\to\C$ by $\Fun(D,\C)$.

\subsection*{Acknowledgements}
The authors would like to thank Thomas Poguntke, Mikhail Kapranov and Julian
K\"ulshammer for interesting discussions surrounding the various topics of this
article. The authors further thank Julian K\"ulshammer for comments on a first
draft of the article. T.D.\ acknowledges the support by the VolkswagenStiftung
for his Lichtenberg Professorship at the University of Hamburg and by the
Hausdorff Center for Mathematics for his Bonn Junior Fellowship during which
parts of this work were carried out.

\section{The higher Auslander algebras of type $\AA$}
\label{sec:1}

\subsection{Higher Auslander--Reiten theory}
\label{ss:higher_ar_theory}
Let $\KK$ be a field. An {\em Auslander algebra} is a finite-dimensional
$\KK$-algebra $\Gamma$ that satisfies the homological constraints
\begin{equation}\label{eq:hom}
  \on{gl.dim}(\Gamma) \le 2 \le \on{dom.dim}(\Gamma),
\end{equation}
where $\on{gl.dim}(\Gamma)$ denotes the global dimension of $\Gamma$ and
$\on{dom.dim}(\Gamma)$ denotes its dominant dimension: the largest natural
number $d$ such that the terms $I^0$, $I^1$, \dots, $I^{d-1}$ in every minimal
injective co-resolution of $\Gamma$ are projective.

Let $A$ be a finite-dimensional $\KK$-algebra of finite representation type and
denote by $M$ the direct sum of a complete set of representatives of the
indecomposable $A$-modules. The endomorphism algebra $\Gamma=\End_A(M)^{\op}$ is
an Auslander algebra and, up to Morita equivalence, every Auslander algebra is
of this form \cite{Aus71}. The functorial approach to Auslander--Reiten theory
consists of relating the representation theory of $A$ to that of $\Gamma$ by
means of the fully faithful left exact functor
\begin{equation}\label{eq:auslander}
  G\colon	\mod A \lra \mod\Gamma,\quad X \mapsto \Hom_A(M,X).
\end{equation}
Notably, under the functor $G$, the almost split exact sequences of $A$-modules
correspond precisely to the minimal projective resolutions of simple
$\Gamma$-modules of projective dimension $2$. This is the seminal observation in
Auslander--Reiten theory, see \cite{ARS97} for further details.

More recently, Iyama proposed a higher-dimensional generalisation of
Auslander--Reiten theory based on the notion of an $(m+1)$-dimensional Auslander
algebra, that is a finite-dimensional $\KK$-algebra $\Gamma$ satisfying the
homological constraint
\begin{equation}\label{eq:homn}
  \on{gl.dim}(\Gamma) \le m+1 \le \on{dom.dim}(\Gamma).
\end{equation}
For $m>1$, given a finite-dimensional $\KK$-algebra $A$, the existence of a
corresponding $(m+1)$-dimensional Auslander algebra hinges on the existence of
an $(m+1)$-cluster tilting module---an $A$-module $M$ satisfying the following
conditions:
\begin{itemize}
\item For every $0<i<m$ we have $\Ext^i_{A}(M,M) = 0$.
\item Every indecomposable $A$-module $X$ such that, for every $0<i<m$, we have
  $\Ext^i_{A}(X,M)=0$, is isomorphic to a direct summand of $M$. In particular
  each indecomposable projective $A$-module is a direct summand of $M$.
\item Every indecomposable $A$-module $Y$ such that, for every $0<i<m$, we have
  $\Ext^i_{A}(M,Y) = 0$, is isomorphic to a direct summand of $M$. In particular
  each indecomposable injective $A$-module is a direct summand of $M$.
\end{itemize}
The algebra $\Gamma = \End_{A}(M)^{\op}$, where $M$ is an $m$-cluster tilting
$A$-module, is an $(m+1)$-dimensional Auslander algebra and a central result of
Iyama says that, up to Morita equivalence, every such algebra is of this form
\cite{Iya07a}. In analogy with the classical situation, a higher-dimensional
variant of Auslander--Reiten theory, based on the obvious analogue of the
functor \eqref{eq:auslander}, can be formulated in this context \cite{Iya07}.

It is a nontrivial task to construct finite-dimensional $\KK$-algebras where
higher Aus\-lan\-der--Reiten theory can be applied. In \cite{Iya11}
Iyama provides a remarkable recursive construction of a family $\{A^{(m)}\}_{m
  \ge 1}$ of (relative) higher Auslander algebras, where $A^{(1)}$
is an arbitrary finite-dimensional representation-finite $\KK$-algebra of global
dimension $1$. The recursion is defined by the formula
\[
	A^{(m+1)} \coloneqq \End_{A^{(m)}}(M)^{\op},
\]
where $M$ is a canonical (relative) $m$-cluster tilting $A^{(m)}$-module, the
existence of which is guaranteed by one of the main results in \cite{Iya11}.

Of particular interest for this work is the case when $A^{(1)}=A_\ell^{(1)}$ is
the $\KK$-linear path algebra of the quiver $1 \to 2 \to \dots \to \ell$ of
Dynkin type $\AA_\ell$. In this case the various higher Auslander
algebras are related by fully faithful left exact functors
\begin{equation}\label{eq:typea}
  G\colon \mod {A_\ell^{(m)}} \lra \mod {A_\ell^{(m+1)}},\quad X \mapsto
  \Hom_{A_\ell^{(m)}}(M,X).
\end{equation}
We refer to this family of algebras as the \emph{higher Auslander
  algebras of type $\AA$}.

We illustrate the first steps of Iyama's construction for the quiver $1 \to 2
\to 3$. The Auslander--Reiten quiver of the module category $\mmod A_3^{(1)}$ is
given by
\begin{equation}
  \label{eq:A3_AR_quiver}
  \includegraphics[width=0.4\textwidth]{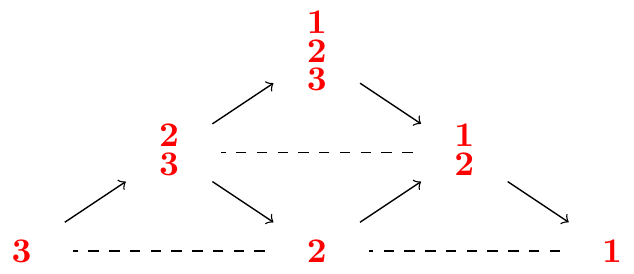}
\end{equation}
where, following standard procedure, we label the various indecomposable modules
by the simple modules appearing in their composition series. The Auslander
algebra $A_3^{(2)}$ is the quotient of the path algebra associated to this
quiver modulo the relations given by the bottom dashed lines, which signify zero
relations, and the middle dashed line, which expresses the commutativity of the
square. As a matter of convenience, we relabel the vertices of this quiver as
\begin{equation*}
  \includegraphics[width=0.4\textwidth]{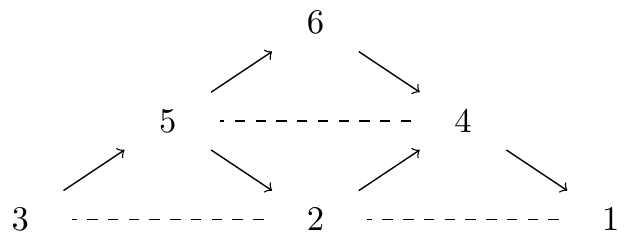}
\end{equation*}
We use this new labelling to describe the Auslander--Reiten quiver of $\mmod
A_3^{(2)}$ which is given by
\begin{equation}
  \label{eq:A3-2_AR_quiver}
  \includegraphics[width=0.85\textwidth]{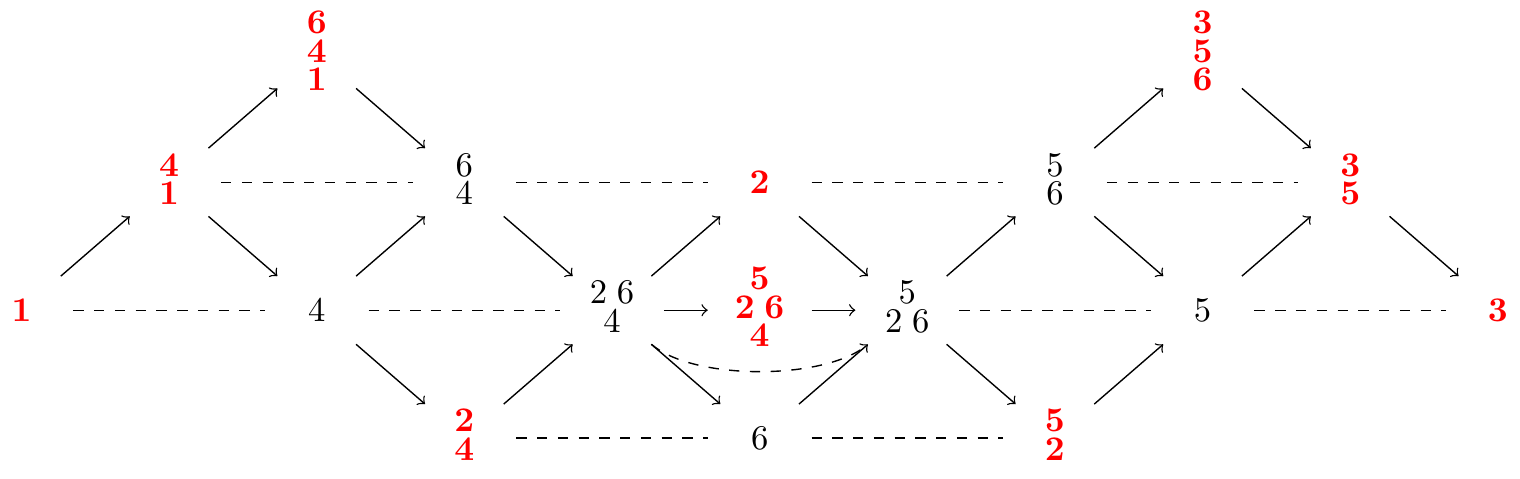}
\end{equation}
In order to define the $3$-dimensional Auslander algebra $A_3^{(3)}$ we need to
find a $2$-cluster tilting $A_3^{(2)}$-module $M$ and set
$A_3^{(3)}=\End_{A_3^{(2)}}(M)$. Iyama shows that, up to multiplicity of its
direct summands, there is a unique such $A_3^{(2)}$-module $M$ and that it is
given by the direct sum of the $10$ indecomposable $A_3^{(2)}$-modules
highlighted in the Auslander--Reiten quiver of $\mmod A_3^{(2)}$. The procedure
continues \emph{ad infinitum} by finding a $3$-cluster tilting
$A_3^{(3)}$-module\ldots

\subsection{The simplicial combinatorics of higher Auslander algebras of type $\AA$}
\label{sec:simplicial}

The starting point for the perspective we develop in this work is the
observation that the structure of the higher Auslander algebras of
type $\AA$ is governed by the combinatorics of simplices in a way we now
explain.

Given a poset $P$, we may construct a category whose objects are the elements of
$P$ with a unique morphism from $i$ to $j$ whenever $i \le j$. Below, we leave
the distinction between a poset and its associated category implicit.

\begin{definition}\label{defi:delta}
  Let $n$ be a natural number. Recall that the \emph{$n$-th standard ordinal} is
  the poset $[n]\coloneqq \set{0<1<\cdots<n}$. Given two natural numbers $m$ and $n$, we
  denote the set of monotone maps $f\colon[m] \to [n]$ by $\Delta(m,n)$. We endow
  $\Delta(m,n)$ with the structure of a poset by declaring $f \le f'$ if, for
  all $i \in [m]$, we have $f(i) \le f'(i)$. We further introduce the partition
	\begin{equation}\label{eq:dpartition}
    \Delta(m,n) = \Delta(m,n)^{\sharp} \cup \Delta(m,n)^{\flat}
	\end{equation}
	into the subset $\Delta(m,n)^{\sharp}$ of strictly monotone maps and its
  complement $\Delta(m,n)^{\flat}$.
\end{definition}

\begin{remark}
  The elements of the poset $\Delta(m,n)$ can be interpreted as $m$-simplices in
  the $n$-simplex $\Delta^n$. The partition \eqref{eq:dpartition} then
  corresponds to the partition into non-degenerate and degenerate $m$-simplices.
  Further note that $\Delta(m,n)$ is a subposet of $\ZZ^{m+1}$ with respect to
  the product order.
\end{remark}

For natural numbers $m$ and $n$, let $\KK \Delta(m,n)$ denote the $\KK$-category
obtained from the category associated to the poset $\Delta(m,n)$ by passing to
$\KK$-linear envelopes of the morphism sets. We further form the $\KK$-category
$\underline{\KK \Delta}(m,n)$ by replacing all morphism spaces in $\KK
\Delta(m,n)$ by the quotient modulo the subspace of morphisms that factor
through an object in $\Delta(m,n)^{\flat}$. Finally, for each $\KK$-category $\A$
with finitely many objects, we introduce the $\KK$-algebra
\[
	\End(\A) = \bigoplus_{(x,y) \in (\ob \A)^2} \A(x,y)
\]
with multiplication obtained from the composition law of $\A$.

Let us illustrate how the simplicial combinatorics relate to the
higher Auslander algebras of type $\AA$ in the special case
$\ell=3$. The $\KK$-category $\KK\Delta(0,2)$ is the free $\KK$-category on the
quiver $0 \to 1 \to 2$. In other words there are algebra isomorphisms
\begin{equation*}
  A_3^{(1)}\cong\End(\KK\Delta(0,2))\cong\End(\underline{\KK\Delta}(0,2)),
\end{equation*}
where the rightmost isomorphism is a consequence of the fact that there are no
degenerate elements in $\Delta(0,2)$.

We now wish to describe the Auslander algebra $A_3^{(2)}$ in a similar way. For
this, consider the poset $\Delta(1,3)$ whose Hasse quiver is given by
\begin{center}
  \includegraphics[width=0.33\textwidth]{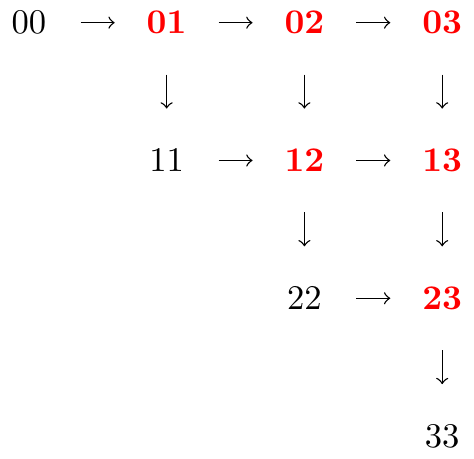}
\end{center}
The elements of the set $\Delta(1,3)^\sharp$ of non-degenerate simplices are
highlighted. Observe that the Hasse quiver of $\Delta(1,3)^\sharp$ agrees with
the Auslander--Reiten quiver of $\mmod A_3^{(1)}$, see \eqref{eq:A3_AR_quiver}.
Note, however, that neither of the algebras
\begin{equation*}
  \End(\KK\Delta(1,3))\qquad\text{and}\qquad\End(\KK\Delta(1,3)^\sharp)
\end{equation*}
is isomorphic to $A_3^{(2)}$: the algebra $\End(\KK\Delta(1,3))$ has too many
simple modules while $\End(\KK\Delta(1,3)^\sharp)$ is missing the zero
relations. As it turns out, the appropriate way to account for the presence of
non-degenerate simplices is to pass to the $\KK$-category
$\underline{\KK\Delta}(1,3)$. Indeed, there is an isomorphism of algebras
\begin{equation*}
  A_3^{(2)}\cong\End(\underline{\KK\Delta}(1,3)).
\end{equation*}

As a final example, let us describe the algebra $A_3^{(3)}$ in terms of posets
of simplices. Consider the poset $\Delta(2,4)$ whose Hasse quiver is given by
\begin{center}
  \includegraphics[width=0.4\textwidth]{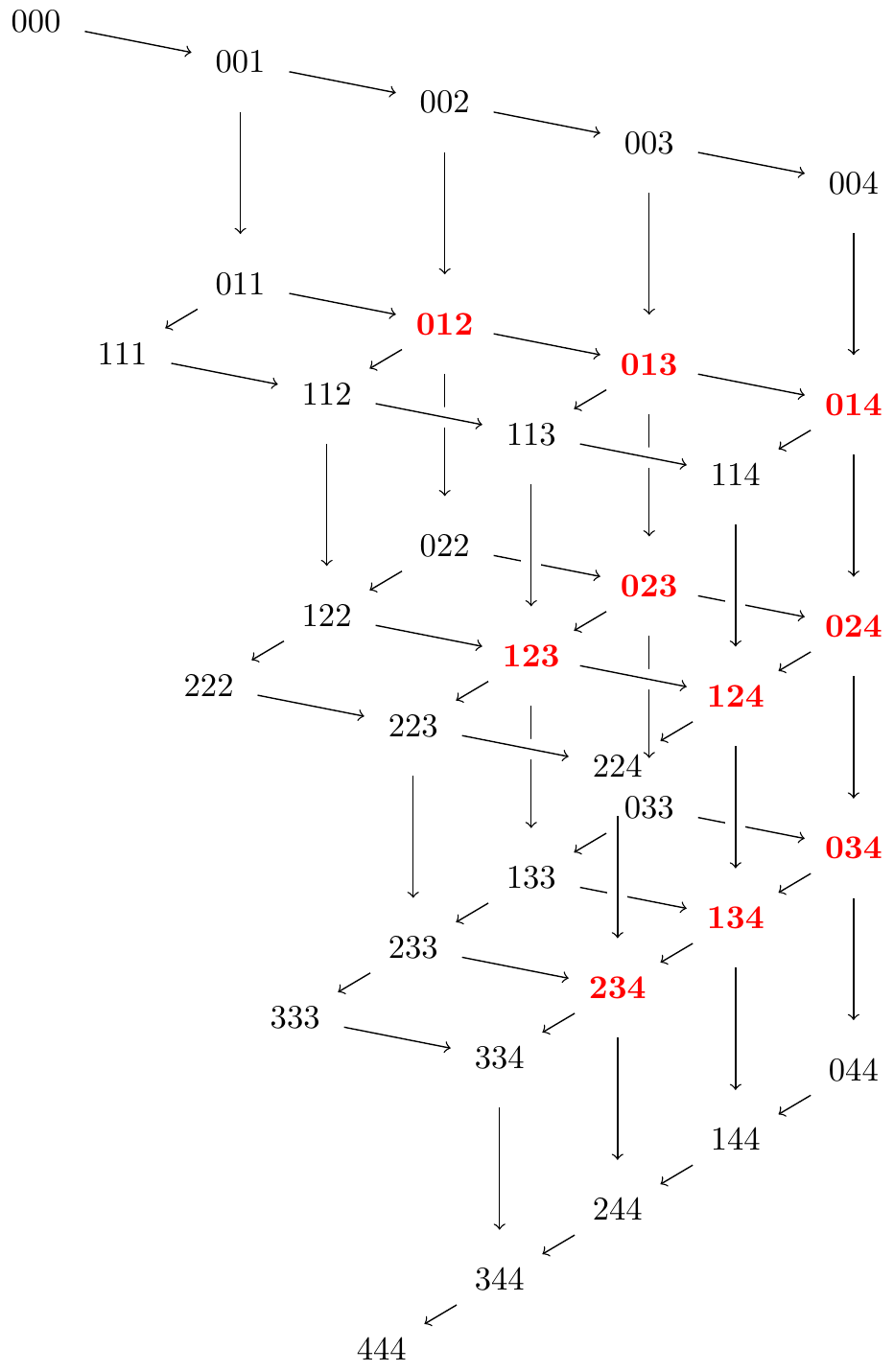}
\end{center}
Again, the elements of the set $\Delta(2,4)^\sharp$ of non-degenerate simplices
are highlighted. Note that the Hasse quiver of $\Delta(2,4)^\sharp$ agrees with
the full subquiver of the Auslander--Reiten quiver of $\mmod A_3^{(2)}$ spanned
by the summands of the unique $2$-cluster tilting $A_3^{(2)}$-module, see
\eqref{eq:A3-2_AR_quiver}. In this case there is an isomorphism of algebras
\begin{equation*}
  A_3^{(3)}\cong\End(\underline{\KK\Delta}(2,4)).
\end{equation*}
More generally, we have the following description of the higher
Auslander algebras of type $\AA$ in terms of posets of simplices.

\begin{proposition}\th\label{prop:comaus}
	Let $n\geq m\geq 1$ and set $\ell=n-m+1$. There is an isomorphism of
  $\KK$-algebras
	\[
		A_\ell^{(m+1)}\cong\End(\underline{\KK \Delta}(m,n))
	\]
	where $A_\ell^{(m+1)}$ denotes the $(m+1)$--dimensional Auslander algebra of
  type $\AA_\ell$.
\end{proposition}
\begin{proof}
	Up to a minor modification of the object labels, a proof of this fact is
  explained in \cite{JK16} based on the description given in \cite{OT12}.
\end{proof}

Let $\ell$ be positive integer. Recall that $A_\ell^{(1)}$ denotes the $\KK$-linear
path algebra of the quiver $1\to2\to\cdots\to\ell$. As explained in
\S\ref{ss:higher_ar_theory}, the higher Auslander algebras of type
$\AA_\ell$ were originally constructed by Iyama by means of the recursion
\begin{equation}
  \label{eq:recursion_A}
  A_\ell^{(m+1)}=\End_{A_\ell^{(m)}}(M)
\end{equation}
where $M$ is, up to multiplicity of its direct summands, the unique
$m$-cluster-tilting $A_\ell^{(m)}$-module. We now explain how this distinguished
$A_\ell^{(m)}$-module can be described in terms of the simplicial combinatorics.

\begin{notation}
  For linearly ordered sets
  \[
    I=\set{i_0<i_1<\dots<i_m}\qquad\text{and}\qquad J=\set{j_0<j_1<\cdots<j_n}
  \]
  let
  \begin{equation*}
    I*J\coloneqq \set{i_0<i_1<\dots<i_m<j_0<j_1<\cdots<j_n}.
  \end{equation*}
  In particular for each $n\geq0$ there is are (unique) isomorphisms of posets
  $[0]*[n]\cong[n+1]$ and $[n]*[0]\cong[n+1]$. These identifications allow us to
  consider the faithful endofunctors
  \begin{equation*}
    ([0]*-)\colon\Delta\lra\Delta\qquad\text{and}\qquad(-*[0])\colon\Delta\lra\Delta.
  \end{equation*}
  Concretely, for a simplex $\sigma\colon[m]\to[n]$ we have
  \begin{align*}
    ([0]*\sigma)_i&=\begin{cases}
      0&i=0\\
      \sigma_{i-1}+1&1\leq i\leq m+1;
    \end{cases}
                      \intertext{and}
                      (\sigma*[0])_i&=\begin{cases}
                        \sigma_i&0\leq i\leq m\\
                        n+1&i=m+1.
                      \end{cases}
  \end{align*}
\end{notation}

\begin{remark}
  The endofunctor $-*[0]$ is the basic ingredient in the
  definition of the \emph{d\'ecalage} of a simplicial set.
\end{remark}

The algebra isomorphism from \th\ref{prop:comaus} implies the existence
of an equivalence of categories
\begin{equation}
  \label{eq:comaus_mit_der_sendung}
  \mod A_\ell^{(m)} \simeq \Fun_{\KK}(\underline{\KK \Delta}(m-1,n-1)^{\op}, \mod{\KK})
\end{equation}
between the category of finitely generated right $A_\ell^{(m)}$-modules and the
category of contravariant $\KK$-linear functors $\underline{\KK \Delta}(m-1,n-1)\to\mmod\KK$.
The poset monomorphism
\begin{equation*}
  ([0]*-)\colon\Delta(m-1,n-1)\hookrightarrow\Delta(m,n)
\end{equation*}
preserves the partition of $\Delta(m-1,n-1)$ into non-degenerate and degenerate
simplices whence it induces a fully faithful $\KK$-linear functor
\begin{equation}
  \label{eq:gamma}
	\gamma\colon \underline{\KK \Delta}(m-1,n-1)^{\op}\lra\underline{\KK \Delta}(m,n)^{\op}.
\end{equation}
The pullback functor $\gamma^*$ admits a fully faithful right adjoint
\begin{equation*}
  \gamma_*\colon\Fun_\KK(\underline{\KK \Delta}(m-1,n-1)^{\op},\mmod\KK)\to\Fun_\KK(\underline{\KK \Delta}(m,n)^{\op},\mmod\KK)
\end{equation*}
given by $\mmod\KK$-enriched right Kan extension along $\gamma$. Under the
equivalence \eqref{eq:comaus_mit_der_sendung}, the functor $\gamma_*$ is
identified with the fully faithful left exact functor
\[
  G\colon \mod A_\ell^{(m)} \lra \mod A_\ell^{(m+1)}
\]
of \eqref{eq:typea}. Finally, consider the Yoneda functor
\[
  \iota\colon \Delta(m,n)\lra\Fun_{\KK}(\underline{\KK \Delta}(m,n)^{\op},
  \mod{\KK}),\quad \sigma \mapsto \Hom_{\underline{\KK \Delta}(m,n)}(-,\sigma)
\]
and form the composite functor
\begin{equation}\label{eq:fundamental}
  \varepsilon = \gamma^* \circ \iota: \Delta(m,n)\lra \mod A_\ell^{(m)},
\end{equation}
where the use of equivalence \eqref{eq:comaus_mit_der_sendung} is left implicit.

\begin{proposition}
  Let $n\geq m\geq 1$ and set $\ell=n-m+1$. Up to multiplicity of its direct
  summands, the unique $m$-cluster-tilting $A_\ell^{(m)}$-module is given by
  \[
    M=\bigoplus_{\sigma\in\Delta(m,n)}\varepsilon(\sigma)=\bigoplus_{\sigma\in\Delta(m,n)}\Hom_{\underline{\KK\Delta}(m,n)}(\gamma(-),\sigma).\qedhere
  \]
\end{proposition}
\begin{proof}
  Indeed, the functor
  \[
    G\colon \mod A_\ell^{(m)} \lra \mod A_\ell^{(m+1)},\quad
    X\mapsto\Hom_{A_{\ell}^{(m)}}(M,X)
  \]
  has a left adjoint
  \[
    F\colon \mod A_\ell^{(m+1)} \lra \mod A_\ell^{(m)},\quad Y\mapsto
    Y\otimes_{A_{\ell}^{(m+1)}}M,
  \]
  where we use the recursion \eqref{eq:recursion_A} to view $M$ as a (left)
  module over $A_\ell^{(m+1)}$. Clearly, the functor $F$ identifies the regular
  representation $A_\ell^{(m+1)}$ with the $A_{\ell}^{(m)}$-module $M$. Finally,
  under the equivalence \eqref{eq:comaus_mit_der_sendung} the adjoint pair
  $F\dashv G$ is identified with the adjoint pair $\gamma^*\dashv\gamma_*$. The
  claim follows.
\end{proof}

As we explain below, the functor $\varepsilon$ reveals interesting simplicial
structures hidden within the representation theory of the higher
Auslander algebras of type $\AA$.

\subsection{Relation to Eilenberg--MacLane spaces}
\label{sec:em}

Let $B$ be an abelian group and $m$ a positive integer. An Eilenberg--MacLane
space of type $\K(B,m)$ is a (path connected) topological space characterised up
to weak homotopy equivalence by the property
\begin{equation}\label{eq:EM}
  \pi_k(\K(B,m)) \simeq \begin{cases} B & \text{if $k = m$,}\\ 0 & \text{if $k \ne m$.}
  \end{cases}
\end{equation}
One important aspect of Eilenberg--MacLane spaces is that they are classifying
spaces for singular cohomology with coefficients in $B$: for every CW-complex
$X$, there is a canonical bijection between $H^n(X,B)$ and the set of homotopy
classes of maps $X\to \K(B,m)$, see Corollary III.2.7 in \cite{GJ99}.

An explicit combinatorial model for an Eilenberg-MacLane space of type $\K(B,m)$
can be obtained via the Dold--Kan correspondence \cite{Dol58,Kan58}. Under this
correspondence, the aforementioned model is the simplicial abelian group that
corresponds to the chain complex $B[m]$ consisting only of the abelian group $B$
placed in homological degree $m$. With some abuse of notation, we denote this
simplicial abelian group by $\K(B,m)$. The characterising property \eqref{eq:EM}
is an immediate consequence of the fact that, under the Dold--Kan
correspondence, the cohomology groups of a given chain complex are identified
with the homotopy groups of the corresponding simplicial abelian group, see
Corollary III.2.5 in \cite{GJ99}. We provide an explicit description of the
abelian group of $n$-simplices of $\K(B,m)$.

\begin{definition}\label{defi:em} Let $m,n \ge 0$. Recall that $\Delta(m,n)$ denotes
  the set of $m$-simplices in $\Delta^n$. Then
  \[
    \K(B,m)_n \subset \Map(\Delta(m,n), B)
  \]
  is the subgroup of those (set-theoretic) functions $f\colon \Delta(m,n) \to B$
  which have the following properties:
  \begin{enumerate}
	\item The function $f$ maps all degenerate $m$-simplices to $0$.
	\item For every non-degenerate $(m+1)$-simplex $\sigma\colon [m+1] \to [n]$ the
    equation
		\[
			\sum_{i=0}^{m+1} f( \sigma \circ d^i ) = 0
		\]
    is satisfied, where $d^i\colon[m]\to[m+1]$ is the unique strictly monotone
    map whose image does not contain $i$.
  \end{enumerate}
  The simplicial structure of $\K(B,m)$ is obtained via the functoriality of
  $\Delta(m,n)$ in $[n] \in \Delta$.
\end{definition}

\begin{theorem}\th\label{thm:em} Let $n\geq m\geq 1$ and set $\ell = n-m+1$. We denote
  the Grothendieck group of the abelian category $\mmod A_\ell^{(m)}$ by
  $K_0(A_\ell^{(m)})$. The following statements hold:
	\begin{enumerate}
  \item The map
    \[
      e\colon \Delta(m,n) \lra K_0(A_{\ell}^{(m)})
    \]
    induced by the functor $\varepsilon\colon\Delta(m,n)\to\mmod A_{\ell}^{(m)}$
    described in \eqref{eq:fundamental} defines an $n$-simplex in
    $\K(K_0(A_\ell^{(m)}),m)$.
  \item The $n$-simplex $e$ is universal in the following sense: For every
    abelian group $B$ the map
    \[
      \Hom_{\Ab}(K_0(A_\ell^{(m)}), B) \lra\K(B,m)_n
    \]
    induced by pullback along $\varepsilon$ is an isomorphism.
  \item There is a co-simplicial abelian group
    \[
      \K_0(A_{\bullet-m+1}^{(m)})\colon\Delta\lra\Ab
    \]
    where, by convention, $K_0(A_{n-m+1}^{(m)})=0$ for $n<m$. Moreover, the
    $n$-simplices $e$ from statement (1) exhibit $K_0(A_{\bullet-m+1}^{(m)})$ as
    a quotient
    \[
      e_\bullet\colon\ZZ\Delta(m,-)\twoheadrightarrow K_0(A_{\bullet-m+1}^{(m)})
    \]
    of the free co-simplicial abelian group $\ZZ\Delta(m,-)$.
  \end{enumerate}
  In particular, for every abelian group $B$ the map
  \[
    e_\bullet^*\colon\Hom_{\Ab}(K_0(A_{\bullet-m+1}^{(m)}), B) \lra \K(B,m)
  \]
  is an isomorphism of simplicial abelian groups.
\end{theorem}
\begin{proof}
  In view of the explicit model for $\K(B,m)$ given in Definition \ref{defi:em},
  the statements are equivalent to saying that the map of abelian groups
  \begin{equation}
    \label{eq:the_map}
    \ZZ \Delta(m,n) \lra K_0(A_{\ell}^{(m)})
  \end{equation}
  induced by $e$ is surjective with kernel $R(m,n)$ generated by:
  \begin{itemize}
  \item the elements of $\Delta(m,n)$ corresponding to degenerate $m$-simplices and
  \item the elements of $\Delta(m,n)$ of the form
    \[
      \sum_{i=0}^{m+1} \tau \circ d^i
    \]
    where $\tau\colon [m+1] \to [n]$ is a non-degenerate $(m+1)$-simplex.
  \end{itemize}
  It is straightforward to verify that $R(m,-)$ is a co-simplicial subgroup of
  $\ZZ\Delta(m,-)$. The map \eqref{eq:the_map} is surjective since all
  indecomposable projectives in
  \[
    \mod A_\ell^{(m)}\simeq\Fun_{\KK}(\underline{\KK \Delta}(m-1,n-1)^{\op}, \mod
    \KK)
  \]
  are in the image of $\varepsilon$ and, since the algebra $A_\ell^{(m)}$ has
  finite global dimension, the indecomposable projectives freely generate
  $K_0(A_\ell^{(m)})$. Indeed, an indecomposable projective module represented
  by an $(m-1)$-simplex $\sigma$ is the image of $\gamma(\sigma)$:
  \[
    \Hom_{\underline{\KK\Delta}(m-1,n-1)}(-,\sigma)\cong\Hom_{\underline{\KK\Delta}(m,n)}(\gamma(-),\gamma(\sigma))=\varepsilon(\gamma(\sigma)).
  \]
  It is obvious that all degenerate $m$-simplices map to $0$. We now show that
  the kernel of the map \eqref{eq:the_map} is given by $R(m,n)$. Each
  non-degenerate simplex $\tau\colon[m+1]\to[n]$ induces a sequence
  \[
    \tau\circ d^{m+1}\lra\tau\circ d^{m}\lra\dots\lra\tau\circ d^{1}\lra\tau\circ
    d^{0}
  \]
  in the poset $\Delta(m,n)$. It is straightforward to verify that this sequence
  is mapped, under the functor $\varepsilon$ of \eqref{eq:fundamental}, to an
  exact sequence
  \[
    0\lra\varepsilon(\tau\circ d^{m+1}) \lra\varepsilon(\tau\circ
    d^{m})\lra\dots\lra\varepsilon(\tau\circ d^{1})\lra\varepsilon(\tau\circ
    d^{0})\lra0
  \]
  in $\mod A_\ell^{(m)}$. Indeed, this exact sequence is described in
  Proposition 3.19 in \cite{OT12}, albeit in a different combinatorial
  framework. This implies the claimed relations.
\end{proof}

\begin{remark}
  \th\label{rmk:em}
  The short exact sequence
  \[
    0\lra R(m,n)\lra\ZZ\Delta(m,n)\lra K_0(A_\ell^{(m)})\lra 0
  \]
  constructed in the proof of \th\ref{thm:em} corresponds to the
  presentation of the Grothendieck group associated to the $m$-cluster tilting
  $A_\ell^{(m)}$-module $M$, which can be defined using the ideas of \cite{BT14}.
\end{remark}

\subsection{Relation to the higher Waldhausen
  $\eS_{\bullet}$-constructions}
\label{sec:sm}

Segal's $\Gamma$-space delooping framework \cite{Seg74} provides a beautiful
means to understand the algebraic $K$-theory spectrum of a ring $R$ as the group
completion of the direct sum operation on the classifying space of finitely
generated projective $R$-modules. In \cite{Wal85} Waldhausen introduced the
$\eS_{\bullet}$-construction as a variant of Segal's construction that also
accounts for the presence of possibly non-split short exact sequences. This
construction works in the much larger generality of what are now called
Waldhausen categories, where Waldhausen himself used it to define his algebraic
$K$-theory of spaces.

We recall a higher-dimensional generalisation of the
$\eS_{\bullet}$-construction introduced by Hesselholt and Madsen in dimension
$2$ \cite{HM15} and by Poguntke in arbitrary positive dimensions \cite{Pog17}.

\begin{definition}\label{defi:sm} For $m,n \ge 0$, let $\Delta(m,n)$ be the poset of $m$-simplices in
	$\Delta^n$. For an abelian category $\B$, we define
	\[
		\eS_n^{\langle m\rangle}(\B) \subset \Fun(\Delta(m,n), \B)
	\]
	to be the full subcategory of functors $X\colon \Delta(m,n) \to \B$ satisfying the
  following conditions:
	\begin{enumerate}
  \item The functor $X$ maps degenerate simplices to zero objects in $\B$.
  \item For every non-degenerate $(m+1)$-simplex $\sigma$, the sequence
    \[
      0 \lra X(\sigma \circ d^{m+1}) \lra X(\sigma \circ d^{m}) \lra \dots \lra
      X(\sigma \circ d^{0}) \lra 0
    \]
    is exact.
	\end{enumerate}
	Via the functoriality of $\Delta(m,n)$ in $[n] \in \Delta$ we obtain a
  simplicial category $\eS_{\bullet}^{\langle m\rangle}(\B)$ that we call the {\em
    $m$-dimensional Waldhausen $\eS_{\bullet}$-construction of $\B$}.
\end{definition}

\begin{remark} Note that Definition \ref{defi:sm} is a categorical version of
  Definition \ref{defi:em}. In fact, there is a direct relation between the two:
  Treating $\eS_{\bullet}^{\langle m\rangle}(\B)$ as a simplicial exact
  category, it follows from additivity that we have an isomorphism of simplicial
  abelian groups
	\[
		K_0(\eS_{\bullet}^{\langle m\rangle}(\B)) \cong K(K_0(\B), m).\qedhere
	\]
\end{remark}

\begin{remark} The simplicial category $\eS_{\bullet}^{\langle 1\rangle}(\B)$ is
  Waldhausen's original $\eS_{\bullet}$-construction. The construction
  $\eS_{\bullet}^{\langle 2\rangle}(\B)$ was introduced by Hesselholt and Madsen
  in \cite{HM15} as a model for real algebraic $K$-theory. Finally, the
  higher versions $\eS_{\bullet}^{\langle m\rangle}(\B)$ were
  introduced and studied by Poguntke in \cite{Pog17} from the point of view of
  higher Segal spaces.
\end{remark}

\begin{remark}
  The relevance of these constructions for algebraic $K$-theory is the
  following: Given a category $\C$, we denote by $\C^{\simeq}$ the groupoid
  obtained by discarding all noninvertible morphisms in $\C$. Then, for an
  abelian category $\B$, the higher algebraic $K$-groups of $\B$ are given by
  the formula
  \[
    K_i(\B) = \pi_{i+m}(|S_{\bullet}^{\langle m\rangle}(\B)^{\simeq}|)
  \]
  where $|-|$ denotes geometric realization. Further, the sequence of spaces
  \[
    |S_{\bullet}^{\langle 1\rangle}(\B)^{\simeq}|,\; |S_{\bullet}^{\langle
      2\rangle}(\B)^{\simeq}|,\; |S_{\bullet}^{\langle 3\rangle}(\B)^{\simeq}|,\;
    \dots
  \]
  forms an $\Omega$-spectrum which is a model for the connective algebraic
  $K$-theory spectrum of $\B$. This can be seen, for example, by using
  straightforward higher-dimensional generalisations of the additivity results
  of Hesselholt and Madsen, or a totalization construction introduced by
  Poguntke.
\end{remark}

We now come to the relation between the higher Waldhausen
$\eS_\bullet$-con\-struct\-ions and the higher Auslander algebras of
type $\AA$. This is captured in the following result which can be considered a
categorified version of \th\ref{thm:em}, see also \th\ref{rmk:em}.

\begin{theorem}\th\label{thm:sm} Let $n\geq m\geq 1$ and set $\ell=n-m+1$. Define the
  subcategory
  \[
    \M_\ell^{(m)}\coloneqq \setP{\varepsilon(\sigma)\in\mmod
      A_\ell^{(m)}}{\sigma\in\Delta(m,n)}.
  \]
  The following statements hold:
	\begin{enumerate}
  \item The functor
    \[
      \varepsilon\colon \Delta(m,n) \lra \M_\ell^{(m)}\hookrightarrow\mod
      A_{\ell}^{(m)}
    \]
    from \eqref{eq:fundamental} defines an $n$-simplex in
    $\eS_{\bullet}^{\langle m\rangle}(\mod A_{\ell}^{(m)})$.
  \item The $n$-simplex $\varepsilon$ is universal in the following sense: For
    every $\KK$-linear abelian category $\B$, the functor
    \[
      \Fun_{\KK}^{\on{ex}}(\M_\ell^{(m)}, \B) \lra \eS_{n}^{\langle m\rangle}(\B)
    \]
    induced by pullback along $\varepsilon$ is an equivalence, where the
    leftmost category is the category of \emph{exact} $\KK$-linear functors
    $\M_\ell^{(m)}\to\B$, that is functors which send sequences
    \begin{equation}
      \label{eq:m-exact_sequence}
      0\lra M_{m+1}\lra M_m\lra\cdots\lra M_1\lra M_0\lra0
    \end{equation}
    in $\M_\ell^{(m)}$ which are exact in $\mmod A_\ell^{(m)}$ to exact
    sequences in $\B$.\qedhere
	\end{enumerate}
\end{theorem}
\begin{proof}
  The first statement is established as part of the proof of
  \th\ref{thm:em}. The second claim follows immediately from
  \th\ref{lemma:OT_m-ex} below and Remark 5.4 in \cite{Pog17} which shows that the
  exactness conditions imposed on the $n$-simplices $F\colon\Delta(m,n)\to\B$ in
  $\eS_n^{\langle m\rangle}(\B)$ are equivalent to the induced functor
  $\M_\ell^{(m)}\cong\underline{\KK\Delta}(m,n)\to\B$ being exact.
\end{proof}

\begin{remark}
  \th\label{rmk:m-abelian}
  After passing to the additive envelope, the resulting subcategory $\M_\ell^{(m)}$ of
  \th\ref{thm:sm} is an $m$-abelian category in the sense of \cite{Jas16}.
\end{remark}

For the convenience of the reader, we translate Proposition 3.19 in \cite{OT12}
to our combinatorial framework.

\begin{lemma}
  \th\label{lemma:OT_m-ex}
  In the setting of \th\ref{thm:sm}, the exact sequences in $\M_\ell^{(m)}$
  of the form \eqref{eq:m-exact_sequence} are precisely the sequences of the
  form
  \begin{equation*}
    0\lra \varepsilon(q_{(0,\dots,0)})\lra\bigoplus_{|v|=1}\varepsilon(q_v)\lra\cdots\lra\bigoplus_{|v|=m}\varepsilon(q_v)\lra\varepsilon(q_{(1,\dots,1)})\lra0
  \end{equation*}
  where $q\colon\set{0,1}^{m+1}\to\Delta(m,n)$ is a rectilinear $(m+1)$-cube in
  the sense of Definition \ref{def:rectilinear_cube}, and $|v|=\sum v_i$.
\end{lemma}

\section{Higher reflection functors in type $\AA$}
\label{sec:reflection}

\subsection{The derived representation theory of higher Auslander algebras of type $\AA$}
\label{subsec:derivedauslander}

Let $\KK$ be a field. We begin this section with a description of the derived
representation theory of the higher Auslander algebras of type $\AA$
in terms of the simplicial combinatorics. Given a differential graded (=dg)
$\KK$-algebra $A$ or, more generally, a small dg $\KK$-category, we denote its
(unbounded) derived $\infty$-category by $\D A$, see Section 1.3 in \cite{Lur17}
for details.

\subsubsection{Derived representation theory of the higher Auslander
  algebras of type $\AA$ via the simplicial combinatorics}
\label{subsec:clustertilting}

Let $\KK$ be a field and let $m$ and $\ell$ be positive integers. Consider the
right derived functor
\begin{equation}\label{eq:rgamma}
  \RRR G\colon \D A_\ell^{(m)} \lra \D A_\ell^{(m+1)},\quad X \mapsto \RHom_{A_\ell^{(m)}}(M, X)
\end{equation}
associated to the left exact functor from \eqref{eq:typea}. The $m$-cluster
tilting $A_\ell^{(m)}$-module $M$ is a compact generator of the derived category
$\D A_\ell^{(m)}$. Indeed, since $A_\ell^{(m)}$ has finite global dimension, $M$
is certainly compact; the fact that $M$ is a generator is also clear since each
indecomposable projective $A_\ell^{(m)}$-module is a direct summand of $M$. In
particular, we could have considered the endomorphism dg algebra
\[
	\RRR A_\ell^{(m+1)} = \RHom_{A_\ell^{(m)}}(M,M)^{\op},
\]
instead of its underived variant
\begin{equation*}
  A_\ell^{(m+1)}=H^0(\RRR A_\ell^{(m+1)})=\Hom_{A_\ell^{(m)}}(M,M),
\end{equation*}
to obtain an exact equivalence
\begin{equation}\label{eq:exequiv}
	\D A_\ell^{(m)}\xra{\simeq}\D \RRR A_\ell^{(m+1)},\quad X \mapsto
	\RHom_{A_\ell^{(m)}}(M, X).
\end{equation}
The functor $\RRR G$ is then obtained as the composite
\begin{equation}\label{eq:composite}
  \D A_\ell^{(m)} \xra{\simeq} \D \RRR A_\ell^{(m+1)} \xra{\iota^*}\D A_\ell^{(m+1)}
\end{equation}
of the eqiuvalence \eqref{eq:exequiv} and the restriction along the canonical
morphism $\iota\colon A_\ell^{(m+1)} \to \RRR A_\ell^{(m+1)}$. The morphism $\iota^*$
should be regarded as a forgetful functor that forgets the action of the higher
structure present in the dg algebra $\RRR A_\ell^{(m+1)}$.

Using the combinatorial description
\begin{equation}\label{eq:comaus}
  A_\ell^{(m)} \cong \End(\underline{\KK \Delta}(m-1,n-1))
\end{equation}
of the higher Auslander algebra from \th\ref{prop:comaus},
our first goal will be to describe the functor \eqref{eq:rgamma} as well as the
factorization \eqref{eq:composite} in terms of representations of the posets of
simplices $\Delta(m,n)$. To this end we work more generally with representations
in an arbitrary stable $\infty$-category $\C$, recovering the classical
situation as the case $\C = \D\KK$.

\begin{definition} Let $m,n \ge 0$, let $S \subset \Delta(m,n)$, and let $\C$ be a
  stable $\infty$-category. A functor $X\colon S \to \C$ is called {\em reduced} if
  it sends all objects in the subset $S \cap \Delta(m,n)^{\flat}$ of degenerate
  simplices to zero objects in $\C$. We further denote by
	\[
		\Fun_*(S, \C) \subset \Fun(S, \C)
	\]
	the full subcategory spanned by the reduced functors.
\end{definition}

Using this terminology, the following result provides a combinatorial
description of the derived category $\D A_\ell^{(m)}$.

\begin{proposition}\th\label{prop:comausd}
	Let $n\geq m\geq 1$ and set $\ell = n-m+1$. There is a canonical equivalence
	\[
		\D A_\ell^{(m)}\simeq \Fun_*(\Delta(m-1,n-1)^{\op},\D\KK)
	\]
	of stable $\infty$-categories.
\end{proposition}
\begin{proof} The existence of the required equivalence is a consequence of
  \th\ref{prop:comaus} and the strictification result Proposition
  4.2.4.4 from \cite{Lur09} applied to the category of (unbounded) chain
  complexes of $\KK$-modules endowed with the projective model structure.
\end{proof}

\begin{notation}
  \label{not:Delta-Delta-op}
  Let $m$ and $n$ be natural numbers. Note that there is a poset isomorphism
  $(-)^*\colon\Delta(m,n)\to\Delta(m,n)^{\op}$ given by sending an $m$-simplex
  $\sigma$ to the simplex
  \begin{equation*}
    \sigma_i^*=n-\sigma_{m-i},\quad i\in[m].
  \end{equation*}
  To simplify the exposition, in what follows we work with functors on
  $\Delta(m,n)$, rather on $\Delta(m,n)^{\op}$. In view of the above
  identification, this is immaterial as it amounts to a relabelling of the
  objects.
\end{notation}

In what follows we use freely the terminology and results about cubical diagrams
in stable $\infty$-categories from Appendix \ref{sec:n-cubes}. Further, given
vectors $v$ and $w$ in $\ZZ^d$, we define the \emph{Hadamard product}
\[
	v \circ w = (v_1w_1,\dots,v_dw_d)
\]
given by their coordinate-wise product.

\begin{definition}
  \th\label{def:rectilinear_cube} Let $d \ge 1$ be an integer. We say that a
  $d$-cube $q\colon I^d \to \ZZ^d$ is \emph{rectilinear} if all of its edges are
  parallel to the standard coordinate vectors in $\ZZ^d$. This condition can be
  expressed in terms of the Hadamard product by saying that, for each $v\in
  I^d$, the equality
  \begin{equation*}
    q_v=q_{0\cdots0}+v\circ(q_{1\cdots1}-q_{0\cdots0})
  \end{equation*}
  is satisfied.
\end{definition}

\begin{definition} Let $m,n \ge 0$, let $S \subset \Delta(m,n)$, and let $\C$ be a
  stable $\infty$-category. A reduced functor $X\colon S \to \C$ is called {\em
    exact} if the restriction of $X$ along every rectilinear $(m+1)$-cube in $S$
  is biCartesian. We further denote by
	\[
		\Fun_*^{\on{ex}}(S, \C) \subset \Fun_*(S, \C)
	\]
	the full subcategory spanned by the exact functors.
\end{definition}

The following result should be regarded as a combinatorial version of
\eqref{eq:exequiv}, see also Corollary \ref{cor:derivedcomb}. It can also be
regarded as a higher-dimensional version of Theorem 4.6 in \cite{GS16}, which
deals with the case $m=1$.

\begin{proposition} \th\label{prop:sliceres} Let $m$ and $n$ be positive integers and
  \[
    \gamma=([0]*-)\colon\Delta(m-1,n-1)\hookrightarrow\Delta(m,n).
  \]
  For a stable $\infty$-category $\C$, the restriction functor
	\[
		\gamma^*\colon \Fun_*^{\ex}(\Delta(m,n), \C) \lra \Fun_*(\Delta(m-1,n-1), \C)
	\]
	is an equivalence of $\infty$-categories.
\end{proposition}
\begin{proof}
  We factor the morphism $\gamma$ as the composite
  \[
    \Delta(m-1,n-1) \xhra{h} \Delta(m-1,n-1)' \xhra{g} \Delta(m-1,n-1) \cup
    \Delta(m,n)^{\flat} \xhra{f} \Delta(m,n)
  \]
  where $\Delta(m-1,n-1)'$ is the union of $\Delta(m-1,n-1)$ with the degenerate
  $m$-simplices of the form $00\dots$. Let $\C$ be a stable $\infty$-category
  and consider the functor
  \[
    f_!g_*h_!\colon \Fun_*(\Delta(m-1,n-1), \C) \lra \Fun_*(\Delta(m,n), \C)
  \]
  given by
  \begin{enumerate}[label=\arabic*.]
	\item left Kan extension along $h$, followed by
	\item right Kan extension along $g$, followed by
	\item left Kan extension along $f$.
  \end{enumerate}
  Since Kan extensions along fully faithful functors are fully faithful
  (Corollary 4.3.2.16 in \cite{Lur09}), this composite functor is fully
  faithful. We now claim that a diagram
  \[
		X\colon \Delta(m,n) \lra \C
  \]
  is in the essential image of the functor $f_!g_*h_!$ if and only if $X$ is
  exact. This claim is verified by means of the pointwise criterion for the Kan
  extensions: the Kan extensions along $h$ and $g$ have the effect of extending
  a given diagram $Y\colon \Delta(m-1,n-1) \to \C$ by assigning zero objects to all
  degenerate edges. The remaining left Kan extension along $f$ completes the
  various punctured rectilinear cubes in $\Delta(m-1,n-1)\cup
  \Delta(m,n)^{\flat}$ to biCartesian cubes in $\Delta(m,n)$ so that the
  required exactness conditions are satisfied. We leave the straightforward
  combinatorial details to the reader.
\end{proof}

\begin{remark}
  The auxiliary rectilinear cubes introduced in the proof of
  \th\ref{prop:sliceres} should be compared the minimal projective resolutions
  described in Proposition 3.17 in \cite{OT12}.
\end{remark}

\begin{remark}
  The proof of Theorem 5.27 in \cite{IO11} can be used to give a constructive
  proof of \th\ref{prop:sliceres} along the lines of the proof of Theorem 4.6 in
  \cite{GS16}, see also \th\ref{prop:paracyclic} below.
\end{remark}

\begin{corollary}\th\label{cor:derivedcomb} Let $m\geq n\geq 1$ and set $\ell = n-m+1$.
  There is a canonical equivalence
	\begin{equation}
    \label{eq:DK_Funex}
		\D A_\ell^{(m)}\simeq \Fun_*^{\ex}(\Delta(m,n)^{\op},\D\KK)
  \end{equation}
	of stable $\infty$-categories.
\end{corollary}

\begin{remark} Consider equivalence \eqref{eq:DK_Funex}. The condition on
  $M$ to be an $m$-cluster tilting $A_{\ell}^{(m)}$-module includes that the only self-extensions
  of $M$ lie in degree $m$. This condition is nicely reflected by the fact that
  the exactness conditions for a functor $X\colon \Delta(m,n)^{\op}\to \D\KK$ only involve
  cubes of dimension $m+1$.
\end{remark}

In the situation of \th\ref{prop:sliceres}, we may choose a quasi-inverse of the
functor $\gamma^*$ thus obtaining the sequence of functors
\[
	\Fun_*(\Delta(m-1,n-1),\C) \xra{\simeq} \Fun_*^{\ex}(\Delta(m,n),\C) \hra
  \Fun_*(\Delta(m,n),\C).
\]
For $\C = \D\KK$, this yields the promised combinatorial description of the
sequence \eqref{eq:composite}.

\subsubsection{Cluster tilting in the derived category via paracyclic
  combinatorics}
\label{subsec:derivedtilting}

In \S \ref{subsec:clustertilting} we discussed a combinatorial counterpart of
the description of the derived category $\D A_\ell^{(m)}$ in terms of a compact
generator given by the $m$-cluster tilting $A_\ell^{(m)}$-module $M$. In this
section we develop an analogous combinatorial viewpoint on the description of
$\D A_\ell^{(m)}$ in terms of a natural (infinite) subcategory associated to
$M$.

Let $\U$ be the subcategory of $\D A_\ell^{(m)}$ spanned by the complexes of the
form $\varepsilon(\sigma)[mi]$ where $\sigma$ is an $m$-simplex in $\Delta^n$
and $i\in\ZZ$. We view $\U$ as a dg $\kk$-category and set
\[
  \RRR A_{\ZZ\ell}^{(m+1)}=\U^{\op}\qquad\text{and}\qquad
  A_{\ZZ\ell}^{(m+1)}=H^0(\RRR A_{\ell}^{(m+1)}).
\]
Note that there is an equivalence of stable $\infty$-categories
\begin{equation}
  \label{eq:Zl-l}
  \D RA_{\ZZ\ell}^{(m)}\xra{\simeq}\D RA_{\ell}^{(m)}
\end{equation}
given by restricting along the morphism $RA_{\ell}^{(m)}\to RA_{\ZZ\ell}^{(m)}$
induced by the inclusion of $M$ into $\U$ as the complexes concentrated in
degree $0$.
\begin{remark}
  The subcategory $\U$ is an $m$-cluster tilting subcategory of the perfect
  derived category $\perf A_{\ell}^{(m)}$, see Theorem 1.21 in \cite{Iya11}.
  Moreover, $\U$ is clearly invariant under the action of the $m$-fold
  suspension functor of $\perf A_{\ell}^{(m)}$. This makes $H^0(\U)$ into an
  $(m+2)$-angulated category in the sense of \cite{GKO13}, see also
  \th\ref{rmk:m-abelian}.
\end{remark}

Arguing as in \eqref{eq:composite}, we obain a sequence of functors
\begin{equation}\label{eq:dcomposite}
	\D A_{\ell}^{(m)}\xra{\simeq} \D \RRR A_{\ZZ \ell}^{(m+1)}\lra \D A_{\ZZ\ell}^{(m+1)}\lra \D A_{\ell}^{(m+1)},
\end{equation}
where the rightmost functor given by restriction along the morphism
$A_\ell^{(m+1)}\to A_{\ZZ\ell}^{(m+1)}$ As we will demonstrate in this section,
the combinatorial counterpart of the passage from $A_\ell^{(m)}$ to $\RRR
A_{\ZZ\ell}^{(m)}$ is the passage from simplices to paracyclic simplices. In
particular, using this paracyclic perspective we will obtain a combinatorial
description of the sequence \eqref{eq:dcomposite}.

\begin{definition}\label{defi:para}
	We denote by $\Li(m,n)$ the set of those monotone maps $f\colon \ZZ \to \ZZ$ that
  satisfy, for all $i \in \ZZ$, the equivariance condition
  \[
    f(i+m+1) = f(i) + n+1.
  \]
  Again, we interpret $\Li(m,n)$ as a poset by declaring $f \le f'$ if, for all
  $i \in \ZZ$, we have $f(i) \le f(i')$. We also introduce the partition
	\begin{equation}\label{eq:lpartition}
		\Li(m,n) = \Li(m,n)^{\sharp} \cup \Li(m,n)^{\flat}
	\end{equation}
	into the subset $\Li(m,n)^{\sharp}$ of strictly monotone maps and its
  complement $\Li(m,n)^{\flat}$.
\end{definition}

\begin{remark} The underlying set of the poset $\Li(m,n)$ is the set of morphisms
  between objects $\widetilde{m}$ and $\widetilde{n}$ in the {\em paracyclic
    category} $\Li$, an enlargement of the simplex category introduced
  independently in \cite{Nis90}, \cite{FL91} and \cite{GJ93}. The poset
  structure on $\Li(m,n)$ that we consider makes $\Li$ a $2$-category.
\end{remark}

\begin{remark}
  The paracyclic category admits a presentation by generators and relations
  extending that of the simplex category. More precisely, $\Lambda$ is generated
  by coface and codegeneracy morphisms satisfying the usual relations in
  $\Delta$ together with an additional automorphism
  $t=t_{n+1}\colon\widetilde{n}\to\widetilde{n}$ of infinite order subject to
  the relations
  \begin{equation}
    \label{eq:paracyclic_category}
    \begin{split}
      t\circ d^i&=d^{i-1}\circ t\\
      t\circ s^i&=s^{i-1}\circ t
    \end{split}
    \begin{split}
      \qquad 1&\leq i\leq n,\qquad\text{and}\qquad\\
      \qquad 1&\leq i\leq n,\qquad\text{and}\qquad
    \end{split}
    \begin{split}
      t\circ d^0&=d^n,\\
      t\circ s^0&=s^n\circ t^2.
    \end{split}\qedhere
  \end{equation}
\end{remark}

\begin{remark}
  Let $m,n\geq0$. A paracyclic simplex $\sigma \in \Li(m,n)$ is determined by
  its restriction to $[m] \subset \ZZ$ so that we obtain an embedding $\Li(m,n)
  \subset \ZZ^{m+1}$. In particular, we may talk about rectilinear $(m+1)$-cubes
  in $\Li(m,n)$ in the sense of \th\ref{def:rectilinear_cube}.
\end{remark}

The following result is a combinatorial version of the equivalence
\eqref{eq:Zl-l}.

\begin{proposition}\th\label{prop:paracyclic} Let $m,n \ge 0$ and let $\C$ be a stable
  $\infty$-category. We identify $\Delta(m,n)$ with the subset of $\Li(m,n)$
  consisting of those maps $f\colon \ZZ \to \ZZ$ that map $[m] \subset \ZZ$ into $[n]
  \subset \ZZ$. Then, the induced restriction functor
	\[
		\Fun^{\ex}_*(\Li(m,n), \C) \lra \Fun^{\ex}_*(\Delta(m,n), \C)
	\]
	is an equivalence of stable $\infty$-categories.
\end{proposition}
\begin{proof} This follows by exhibiting a quasi-inverse of the above functor by
  means of successive Kan extensions similar to the proof of
  \th\ref{prop:sliceres}. We leave the details to the reader.
\end{proof}

The poset isomorphism $\Delta(m,n)\to\Delta(m,n)^{\op}$ described in Notation
\ref{not:Delta-Delta-op} extends to a poset isomorphism
$\Li(m,n)\to\Li(m,n)^{\op}$. In particular, we have the following corollary of
\th\ref{cor:derivedcomb} and \th\ref{prop:paracyclic}.

\begin{corollary}\label{cor:derivedcombpara} Let $m\geq n\geq 1$ and set $\ell =
  n-m+1$. There is a canonical equivalence
	\[
		\D A_\ell^{(m)}\simeq \Fun_*^{\ex}(\Li(m,n)^{\op},\D\KK)
	\]
	of stable $\infty$-categories.
\end{corollary}

Combining the equivalences from
\th\ref{prop:sliceres} and \th\ref{prop:paracyclic} yields the sequence of
functors
\[
	\Fun_*(\Delta(m-1,n-1), \C) \xra{\simeq} \Fun^{\ex}_*(\Li(m,n), \C)
  \hra \Fun_*(\Li(m,n), \C)\lra\Fun_*(\Delta(m,n),\C)
\]
which, for $\C = \D\KK$, provides the promised combinatorial construction of
\eqref{eq:dcomposite}.

\subsection{Higher reflection functors in type $\AA$ and slice
  mutation}
\label{sec:mutation}

Let $m$ and $n$ be natural numbers and consider the composite inclusion of
posets
\[
	j\colon \Delta(m-1,n-1) \xhra{[0]*-} \Delta(m,n) \hra \Li(m,n).
\]
We denote the image of this inclusion by $\underline{S}_P$ and let
$S_P=\underline{S}_P\cap\Lambda(m,n)^\sharp$. From
\th\ref{prop:paracyclic} and \th\ref{prop:sliceres}, we deduce that, for a
stable $\infty$-category $\C$, the restriction functor
\begin{equation}\label{eq:sliceeq}
  j^*\colon \Fun_*^{\ex}(\Li(m,n), \C) \lra \Fun_*(\underline{S}_P, \C)
\end{equation}
is an equivalence of stable $\infty$-categories. We wish to explain the
following:
\begin{itemize}
\item The subset $S_P\subset \Li(m,n)^\sharp$ is an example of a {\em slice} in
  $\Li(m,n)^\sharp$.
\item An equivalence analogous to \eqref{eq:sliceeq} can be obtained more
  generally for each slice $S \subset \Li(m,n)^\sharp$.
\item We can mutate between different slices via higher-dimensional analogues of
  reflection functors.
\end{itemize}
To this end, we begin with the description of certain symmetries of the poset
paracyclic simplices.

\begin{notation}
  For a natural number $n$, we denote by $t_n\colon\ZZ\to\ZZ$ the monotone map
  given by $i\mapsto i+1$ and view it as an element of $\Li(n,n)$.
\end{notation}

\begin{definition}
  \th\label{def:three_amigos} Let $n\geq m\geq 1$ be integers.
  \begin{enumerate}
  \item The \emph{Heller automorphism of $\Li(m,n)$} is the poset automorphism
    \begin{equation*}
      \Sigma_m\coloneqq -\circ t_m\colon \Li(m,n)\lra \Li(m,n).
    \end{equation*}
    Thus, given $\sigma\in\Lambda(m,n)$ we have
    $\Sigma_m(\sigma)_i=\sigma_{i+1}$ for each $i\in\ZZ$.
  \item The \emph{Coxeter automorphism of $\Li(m,n)$} is the poset automorphism
    \begin{equation*}
      \Phi_m\coloneqq t_n^{-1}\circ -\colon \Li(m,n)\lra \Li(m,n).
    \end{equation*}
    Thus, given $\sigma\in\Lambda(m,n)$ we have $\Phi_m(\sigma)_i=\sigma_i-1$
    for each $i\in\ZZ$.
  \item The \emph{Nakayama automorphism} is the poset automorphism
    \begin{equation*}
      \SS\coloneqq \Phi_m\circ\Sigma_m\colon \Li(m,n)\lra \Li(m,n).
    \end{equation*}
    Thus, given $\sigma\in\Lambda(m,n)$ we have $\SS(\sigma)_i=\sigma_{i+1}-1$
    for each $i\in\ZZ$.\qedhere
  \end{enumerate}
\end{definition}

\begin{remark} Let $\KK$ be a field, let $m,n \ge 0$, and set $\ell = n-m+1$.
  Recall from Corollary \ref{cor:derivedcombpara} that there is an equivalence
  \[
		\D A_\ell{(m)}\simeq \Fun_*^{\ex}(\Li(m,n), \D\KK).
	\]
  Utilising methods similar to those employed in Section 5 in \cite{GS16}, one
  can show that the Heller, Coxeter, and Nakayama automorphisms of the poset
  $\Li(m,n)$ induce, via pullback, familiar autoequivalences on the derived
  category of $A_\ell^{(m)}$. More precisely
  \[
    \Sigma_m^*\cong[m]\qquad\text{and}\qquad\SS^*\cong-\otimes_{A_\ell^{(m)}}^{\mathbb{L}}\Hom_\KK(A_\ell^{(m)},\KK).
  \]
  In particular $\SS^*$ is a Serre functor on $\perf A_{\ell}^{(m)}$. Moreover,
  $\Phi_m^*=(\SS\circ\Sigma_m^{-})^*$ corresponds to the derived
  $m$-Auslander--Reiten translation functor introduced by Iyama in \cite{Iya11}.
  Moreover, by the very definition of the poset $\Li(m,n)$, there is an equality
  \[
    \SS_m^{n+1}=\Sigma_m^{n-m}
  \]
  which reflects the fractional Calabi--Yau dimension of $A_{\ell}^{(m)}$, see
  \cite{HI11a} for details.
\end{remark}

What follows is a translation of Definition 5.20 in \cite{IO11} to our
combinatorial framework. Recall that, for positive integers $m$ and $n$, the
subset $\Lambda_\infty(m,n)^\sharp$ consists of the strictly monotone maps in
$\Lambda_\infty(m,n)$.

\begin{definition}\label{defi:slice}
  Let $n\geq m\geq 1$ and let $S\subset\Li(m,n)^\sharp$ a finite subset. The
  subset $S$ is a \emph{slice} if the following conditions are satisfied:
  \begin{enumerate}
  \item The subset $S$ consists of a complete system of representatives of the
    $\Phi_m$-orbits of non-degenerate paracyclic $m$-simplices.
  \item For every pair of paracyclic $m$-simplices $\sigma$ and
    $\tau$ in $S$, the interval $[\sigma,\tau]\subset\Li(m,n)^{\sharp}$
    is contained in $S$.\qedhere
  \end{enumerate}
\end{definition}

In this context, slices should be thought of higher-dimensional analogues of
Dynkin quivers of type $\AA$ with an arbitrary orientation. More precisely, we
have the following observation.

\begin{remark}
	Let $n\geq 1$ and let $S \subset \Li(1,n)^\sharp$ be a slice. In this case the
  underlying graph of the Hasse quiver of $S$ is the Dynkin diagram $\AA_n$ and
  the Hasse quivers for the various slices exhaust all possible choices of
  orientation.
\end{remark}

\begin{example}
	Let $n\geq m\geq 1$ be integers. Recall the construction of the subset $S_P
  \subset \Li(m,n)$ as the image of the composite inclusion
	\[
		\Delta(m-1,n-1) \xhra{[0]*-} \Delta(m,n) \hra \Li(m,n).
	\]
	Then $S_P$ is a slice, called the {\em projective slice}. Dually, denote by
  $S_I \subset \Li(m,n)$ the image of the composite inclusion
	\[
		\Delta(m-1,n-1) \xhra{-*[0]} \Delta(m,n) \hra \Li(m,n).
	\]
	Then $S_I$ is a slice, called the {\em injective slice}. This terminology
  stems from the fact that the image of $S_P$ (resp.\ $S_I$) under the functor
	\[
		\varepsilon\colon \Delta(m,n) \lra \mod A_\ell^{(m)}
	\]
	from \eqref{eq:fundamental} is a complete set of representatives of the
  indecomposable projective (resp.\ injective) $A_\ell^{(m)}$-modules.
\end{example}

Our interest in the notion of a slice arises from the existence of the following
mutation operations which can be thought of as higher-dimensional versions of
sink-source reflections, \emph{c.f.} Definition 5.25 in \cite{IO11}.

\begin{proposition}
  Let $n\geq m\geq 1$ be integers and let $S \subset \Li(m,n)^\sharp$ be a
  slice.
  \begin{enumerate}
  \item Let $\sigma\in S$ be a minimal element. The subset
    \begin{equation*}
      \mu_\sigma^R(\sigma)= (S\setminus\set{\sigma})\cup\set{\Phi_m^{-1}(\sigma)}
      \subset \Li(m,n)^\sharp
    \end{equation*}
    is a slice, called the \emph{right mutation of $S$ at $\sigma$}, and
    $\Phi_m^{-1}(\sigma)$ is a maximal element therein.
  \item Let $\sigma\in S$ be a maximal element. The subset
    \begin{equation*}
      \mu_\sigma^L(\sigma)=(S\setminus\set{\sigma})\cup\set{\Phi_m(\sigma)} \subset
      \Li(m,n)^\sharp
    \end{equation*} is a slice, called the {\em left mutation of $S$ at $\sigma$}, and
    $\Phi_m(\sigma)$ is a minimal element therein.
  \item If $\sigma\in S$ is a minimal element, then
    $\mu_\sigma^L(\mu_\sigma^R(S))=S$ and, dually, if $\sigma\in S$ is a maximal
    element, then $\mu_\sigma^R(\mu_\sigma^L(S))=S$.\qedhere
  \end{enumerate}
\end{proposition}
\begin{proof}
  See statements (3) and (4) in Proposition 5.26 in \cite{IO11}.
\end{proof}

\begin{remark}
  Let $n\geq m\geq 1$ and let $\ell=n-m+1$. Slices in $\Li(m,n)^\sharp$ classify
  a distinguished class of tilting complexes in the perfect derived category
  $\perf(A_\ell^{(m)})$ which are higher-dimensional analogues of the APR
  tilting complexes of \cite{APR79} (cf. Section 5 in \cite{IO11}). For $m=1$,
  these tilting complexes and their mutations are closely related to the
  classical reflection functors of \cite{BGP73}.
\end{remark}

The following is the main theorem in the context of slice mutation, \emph{c.f.}
Theorem 5.27 in \cite{IO11}.

\begin{theorem}[Iyama and Oppermann]
  \th\label{thm:slice_mutation:transitivity} Let $n\geq m\geq 1$ be integers.
  Iterated slice mutation acts transitively on the set of slices in
  $\Li(m,n)^\sharp$.
\end{theorem}

To apply \th\ref{thm:slice_mutation:transitivity} in our higher categorical
context, we prove a minor refinement of statements (1) and (2) in Proposition
5.26 in \cite{IO11}.

\begin{lemma}
  \th\label{lemma:slice_mutation:cube} Let $n\geq m\geq 1$ be integers. Let
  $S\subset\Li(m,n)^\sharp$ be a slice and let $\sigma\in S$ a minimal element.
  Then all the non-degenerate paracyclic simplices of the $(m+1)$-cube
$\sigma+I^{m+1}$ are contained in $S\cup\set{\Phi_m^{-1}(\sigma)}$.
\end{lemma}
\begin{proof}
  Let $v\in I^{m+1}$ be such that $\sigma+v$ is non-degenerate. We need to prove
  that $\sigma+v$ belongs to $S\cup\set{\Phi_m^{-1}(\sigma)}$. The claim is
  clear if $|v|=0$ or $|v|=m+1$ since, by definition,
  $\Phi_m^{-1}(\sigma)=\sigma+(1,\dots,1)$. Suppose that $0<|v|<m+1$. Given that
  $S$ is a slice in $\Li(m,n)$ there exists an integer $a$ such that
  $\Phi_m^a(\sigma+v)\in S$. The minimality of $\sigma$ in $S$ readily implies
  that $a\leq 0$, as otherwise we would have $\Phi_m^a(\sigma+v)<\sigma$.
  Suppose that $a<0$. If this is the case, then the inequalities
  \begin{equation*}
    \sigma<\Phi_m^{-1}(\sigma)\leq\Phi_m^a(\sigma+v)
  \end{equation*}
  are satisfied. But the convexity of $S$ implies that $\Phi_m^{-1}(\sigma)\in
  S$, which contradicts the fact that $S$ is a slice in $\Li(m,n)$ for $S$
  cannot contain both $\sigma$ and $\Phi_m^{-1}(\sigma)$. Whence $a=0$ and
  $\sigma+v$ belongs to $S$ as required. This finishes the proof.
\end{proof}

\begin{definition}
  \th\label{not:diamond} Let $n\geq m\geq 1$ be integers. Let $S \subset
  \Li(m,n)$ be a slice, let $\sigma\in S$ be a minimal element, and let
  $S'=\mu_\sigma^R(\sigma)$ be the right mutation of $S$ at $\sigma$. We define
  the poset
  \begin{equation*}
    S\diamond S'\coloneqq S\cup\set{\Phi_m(\sigma)}=\set{\sigma}\cup S'.\qedhere
  \end{equation*}
\end{definition}

\begin{remark}
  \th\label{rmk:diamond} Let $n\geq m\geq 1$ be integers. Let $S \subset
  \Li(m,n)$ be a slice, let $\sigma\in S$ be a minimal element, and let
  $S'=\mu_\sigma^R(\sigma)$ be the right mutation of $S$ at $\sigma$. The
  canonical inclusions
  \begin{equation*}
    S\subset S\diamond S'\supset S'
  \end{equation*}
  allow us to describe the passage from $S$ to $S'$ by means of the rectilinear
  $(m+1)$-cube
  \begin{equation*}
    \sigma+I^{m+1}=\Phi_m^{-1}(\sigma)-I^{m+1},
  \end{equation*}
  see \th\ref{lemma:slice_mutation:cube}.
\end{remark}

We introduce the following auxiliary definition which allows us to meaningfully
consider slices in the higher categorical context.

\begin{definition}
  Let $n\geq m\geq 1$ and $J\subset\Lambda_\infty(m,n)^\sharp$. We enlarge $J$
  to the poset
  \[
    \underline{J}\coloneqq \setP{\rho\in\Lambda_\infty(m,n)}{\exists\sigma,\tau\in
      J\colon\sigma\leq\rho\leq\tau}.
  \]
  Note that, if $S$ is a slice, then $\underline{S}\setminus S$ consists only of
  degenerate paracyclic simplices.
\end{definition}

We will now prove the main result of this section, which can be regarded as a
higher-dimensional generalisation of Theorem 4.15 in \cite{GS16}.

\begin{theorem}
  \th\label{thm:slices_Smn} Let $n \ge m \ge 1$, let $\C$ be a stable
  $\infty$-category, and let $S \subset \Li(m,n)^\sharp$ be a slice. The functor
  \begin{equation*}
	  \Fun_*^{\ex}(\Li(m,n),\C) \xra{\simeq} \Fun_*(\underline{S},\C),
  \end{equation*}
  induced by restriction to $\underline{S}$ is an equivalence of stable
  $\infty$-categories.
\end{theorem}

\th\ref{thm:slices_Smn} will be deduced from
\th\ref{thm:slice_mutation:transitivity} and the following result.

\begin{proposition}
  \th\label{prop:slices_Smn}Let $n \ge m \ge 1$, let $\C$ be a stable
  $\infty$-category. Let $S$ be a slice in $\Li(m,n)$ and let
  $S'=\mu_\sigma^R(S)$ the right mutation of $S$ at a minimal element $\sigma\in
  S$. Then, the restriction functors comprising the zig-zag
  \begin{equation*}
	  \Fun_*(\underline{S},\C) \xra{\simeq} \Fun_*^{\Ex}(\underline{S\diamond S'},\C) \xra{\simeq} \Fun_*(\underline{S}',\C).
  \end{equation*}
  are equivalences of $\infty$-categories.
\end{proposition}
\begin{proof}
  The claim follows immediately from \th\ref{lemma:slice_mutation:cube},
  \th\ref{rmk:diamond} and the fact that the space of colimits of a given finite
  shape in $\C$ is contractible, see Proposition 1.2.12.9 and Remark 1.2.13.5 in
  \cite{Lur17}.
\end{proof}

\begin{proof}[Proof of \th\ref{thm:slices_Smn}]
  By \th\ref{thm:slice_mutation:transitivity,prop:slices_Smn}, for each slice $S$
  and its right mutation $S' \subset \Li(m,n)$ at a minimal element, we have a
  diagram
  \[
    \begin{tikzcd}
      &	& \Fun_*(\underline{S},\C)\\
			\Fun_*^{\ex}(\Li(m,n),\C) \ar{r} & \Fun_*^{\ex}(\underline{S\diamond S'},
			\C)\ar{ur}{\simeq}\ar{dr}{\simeq} & \\
      & & \Fun_*(\underline{S}',\C).
    \end{tikzcd}
  \]
  By two-out-of-three, we deduce that the slice restriction
  $\Fun_*^{\ex}(\Li(m,n),\C) \to \Fun_*(\underline{S},\C)$ is an equivalence if
  and only if the slice restriction $\Fun_*^{\ex}(\Li(m,n),\C) \to
  \Fun_*(\underline{S}',\C)$ is an equivalence. Therefore, to prove the claim, by
  the transitivity of slice mutation, it suffices to check that a particular
  slice restriction is an equivalence. To this end, we note that we have
  established, in \eqref{eq:sliceeq}, such an equivalence
  \[
	  \Fun_*^{\ex}(\Li(m,n),\C) \xra{\simeq} \Fun_*(\underline{S}_P,\C)
  \]
  given by restriction to the projective slice.
\end{proof}

\begin{remark}
  A proof of \th\ref{thm:slices_Smn} in the case $m=1$ which utilises a version
  of the `knitting algorithm' is given in Theorem 4.15 in \cite{GS16}, see also
  Section 6 in \cite{GS16b}. In contrast, an important feature of our proof of
  \th\ref{thm:slices_Smn} is the use of slices in the poset of paracyclic
  simplices and the description of their mutations by means of biCartesian
  cubes.
\end{remark}

\subsection{Relation to the higher Waldhausen
  $\eS_{\bullet}$-constructions}
\label{ssec:Sk}

Let $\KK$ be a field and $n\geq m\geq 1$. In Corollary \ref{cor:derivedcomb} and
Corollary \ref{cor:derivedcombpara} we established equivalences
\begin{equation}\label{eq:delta}
  \D A_\ell^{(m)}\simeq \Fun_*^{\ex}(\Delta(m,n), \D\KK)
\end{equation}
and
\begin{equation}\label{eq:para}
  \D A_\ell^{(m)}\simeq \Fun_*^{\ex}(\Li(m,n), \D\KK),
\end{equation}
respectively. In \S \ref{sec:mutation} we have seen that these equivalences
allow us to develop a combinatorial approach to higher reflection
functors in type $\AA$. However, there is more to be learnt: It follows from equivalence
\eqref{eq:delta} that the various derived categories $\D A_{\bullet-m+1}^{(m)}$
can be organised into a simplicial object and equivalence \eqref{eq:para}
implies that this simplicial object carries a canonical paracyclic structure. We
formulate this statement more generally as follows.

\begin{notation}
  Following Section 1.1.4 in \cite{Lur17}, we denote the $\infty$-category of
  (small) stable $\infty$-categories and exact functors between them by
  $\Cat_\infty^{\Ex}$.
\end{notation}

\begin{proposition}\label{prop:Sm_paracyclic}
	Let $\C$ be a (small) stable $\infty$-category and let $m \ge 0$.
	\begin{enumerate}
  \item The association
    \[
      [n] \mapsto \Fun_*^{\ex}(\Delta(m,n), \C)
    \]
    extends to define a simplicial object
    \[
      \eS_{\bullet}^{\langle m\rangle}(\C)\colon \Delta^{\op} \lra
      \Cat^{\Ex}_{\infty},
    \]
    called the {\em $m$-dimensional $\eS_{\bullet}$-construction of $\C$}.
  \item The association
    \[
      \widetilde{n} \mapsto \Fun_*^{\ex}(\Li(m,n), \C)
    \]
    extends to define a paracyclic object
    \[
      \widetilde{\eS}_{\bullet}^{(m)}(\C)\colon \Li^{\op} \lra \Cat^{\Ex}_{\infty},
    \]
    called the {\em $m$-dimensional paracyclic $\eS_{\bullet}$-construction of
      $\C$.}
  \item Restriction along the inclusion $\Delta\hookrightarrow\Li$ induces an
    equivalence
    \[
      \widetilde{\eS}_{\bullet}^{\langle m\rangle}(\C)|_\Delta\simeq
      \eS_{\bullet}^{\langle m\rangle}(\C).\qedhere
    \]
	\end{enumerate}
\end{proposition}
\begin{proof} In view of Proposition \ref{prop:paracyclic}, we only need to
  verify that the values of the simplicial (resp.\ paracyclic) objects are really
  stable $\infty$-categories and that the functoriality is provided by exact
  functors. This can be verified using the results in Appendix
  \ref{sec:n-cubes}.
\end{proof}

\begin{remark}
  Let $\C$ be a stable $\infty$-category and $n\geq m\geq 1$ integers. The
  paracyclic autoequivalence $t_n$ on $\eS_n^{\langle m\rangle}(\C)$ is induced
  by the Coxeter automorphism $\Phi_m$ on $\Li(m,n)$.
\end{remark}

\subsection{Ladders of recollements}

We now wish to express the inductive nature of the higher-dimensional versions
of the Waldhausen $\eS_\bullet$-construction in terms of certain ladders of recollements
in the sense of \cite{BlGS88} and \cite{AHKLY17}.

\begin{proposition}
  \th\label{def:ds_adjunctions} Let $\C$ be a stable $\infty$-category and $m$ a
  natural number. For each $n\geq 0$ the functors
  \begin{equation*}
    \begin{tikzcd}[column sep=large]
      \eS_n^{\langle m\rangle}(\C)\ar[shift
      left=4.5]{r}[description]{s_0}\ar[shift
      right=4.5]{r}[description]{s_n}&\eS_{n+1}^{\langle
        m+1\rangle}(\C)\ar[shift
      right=8]{l}[description]{d_{0}}\lar[description,phantom,shift
      right=1]{\vdots}\ar[shift left=8]{l}[description]{d_{n+1}}
    \end{tikzcd}
  \end{equation*}
  are part of a string of adjunctions $d_0\dashv s_0\dashv d_1\dashv
  s_1\dashv\cdots\dashv d_n\dashv s_n\dashv d_{n+1}$.
\end{proposition}
\begin{proof}
  This is a straightforward consequence of the following observations:
    \begin{itemize}
    \item The poset structure on the sets of morphisms in $\Delta$ make it into
      a $2$-category $\DDelta$. Moreover, there is a chain of adjunctions
      $d_0\dashv s_0\dashv d_1\dashv s_1\dashv\cdots\dashv d_n\dashv s_n\dashv
      d_{n+1}$ in the $1$-cell dual $2$-category $\DDelta^{(\op,-)}$.
    \item The higher $\eS_\bullet$-constructions are defined in terms of the
      representable $2$-functor $\Fun(-,\C)$.
    \end{itemize}
    We leave the details to the reader.
\end{proof}

\begin{lemma}
  \th\label{prop:ker_di} Let $\C$ be a stable $\infty$-category and $m\geq1$ an
  integer. For each $n\geq 0$ and for each $i\in[n+1]$ there is an equivalence
  of $\infty$-categories
  \begin{equation*}
    \eS_n^{\langle m\rangle}(\C)\simeq \ker d_i\subseteq\eS^{\langle m+1\rangle}_{n+1}(\C).
  \end{equation*}
  In particular the stable $\infty$-category $\ker d_i$ is independent of
  $i\in[n]$.
\end{lemma}
\begin{proof}
  The paracyclic identities \eqref{eq:paracyclic_category} imply the existence
  of equivalences of $\infty$-categories $\ker d_{i}\simeq\ker d_{i+1}$ for each
  $i\in[n]$. Hence it is enough to prove the claim for $i=0$. Firstly, recall
  from \eqref{eq:sliceeq} the equivalence
  \[
    \eS_{n+1}^{\langle
      m+1\rangle}(\C)=\Fun_*^{\ex}(\Lambda_\infty(m+1,n+1),\C)\xra{\simeq}\Fun_*(\underline{S}_P,\X),
  \]
  where $S_P\subset\Lambda_\infty(m+1,n+1)$. Secondly, Lemma \ref{lemma:t-cofib}
  and Proposition \ref{prop:tcofib_characterisation} imply that
  \[
    \ker d_0\simeq\Fun_*^{\ex}(\underline{S}_P,\C)\subseteq\Fun_*(\underline{S}_P,\C).
  \]
  Finally, the inclusion $([0]*-)\colon\Delta(m,n)\to\Delta(m+1,n+1)$ identifies
  $\Delta(m,n)$ with $S_P$ yielding equivalences
\[
  \ker
  d_0\simeq\Fun_*^{\ex}(\underline{S}_P,\C)\simeq\Fun_*^{\ex}(\Delta(m,n),\C)=\eS_n^{\langle
    m\rangle}(\C)
\]
of stable $\infty$-categories, which is what we needed to prove.
\end{proof}

As an easy consequence of \th\ref{def:ds_adjunctions,prop:ker_di} we obtain the
following result.

\begin{proposition}
  \th\label{prop:recollements} Let $\C$ be a stable $\infty$-category and
  $m\geq1$ an integer. Then, for each $n\geq0$ there is a ladder of recollements
  of stable $\infty$-categories
  \begin{equation*}
    \begin{tikzcd}[column sep=large]
      \eS_n^{\langle m+1\rangle}(\C)\ar[shift
      left=4.5]{r}[description]{s_0}\ar[shift
      right=4.5]{r}[description]{s_n}&\eS_{n+1}^{\langle m+1
        \rangle}(\C)\ar[shift
      right=8]{l}[description]{d_0}\lar[description,phantom,shift
      right=1]{\vdots}\ar[shift left=8]{l}[description]{d_{n+1}}\rar[shift
      right=4.5]\rar[shift left=4.5]\ar[phantom, shift
      left=0.5]{r}{\vdots}&\eS_n^{\langle m\rangle}(\C).\lar[shift
      right=8]\lar[shift left=8]
    \end{tikzcd}
  \end{equation*}
  in the sense of \cite{AHKLY17}.
\end{proposition}
\begin{proof}
  Let $i\in[n]$. According to Proposition A.8.20 in \cite{Lur17}, the sequence
  of adjunctions $d_i\dashv s_i\dashv d_{i+1}$ is part of a recollement
  \begin{equation*}
    \recollement{\eS^{\langle m+1\rangle}_n(\C)}{\eS^{\langle m+1\rangle}_{n+1}(\C)}{\D}{d_i}{s_i}{d_{i+1}}{j^*}{j}{j_*}
  \end{equation*}
  where the stable $\infty$-category $\D$ can be identified with the full
  subcategory of $\eS^{\langle m+1\rangle}_{n+1}$ spanned by the objects $X$
  such that for every object $Y$ of $\eS_n^{\langle m+1\rangle}(\C)$ the mapping
  space $\operatorname{Map}_{\eS^{\langle m+1\rangle}_{n+1}}(s_i(Y),X)$ is
  contractible. We claim that $\D$ can be equivalently identified with the
  kernel of the right adjoint of $s_i$, that is the kernel of $d_{i+1}$. Indeed,
  the unit transformation of the adjunction $s_i\dashv d_{i+1}$ induces an
  isomorphism
  \[
    \operatorname{Map}_{\eS_{n+1}^{\langle
        m+1\rangle}(\C)}(s_i(Y),X)\xra{\simeq}
    \operatorname{Map}_{\eS_n^{\langle m+1\rangle}(\C)}(Y,d_{i+1}(X))
  \]
  in the homotopy category of spaces, see Definition 5.2.2.7 and Proposition
  5.2.2.8 in \cite{Lur17}, which readily implies that $\D$ contains the kernel
  of $d_{i+1}$. Conversely, suppose that $X$ is an object of $\D$. Then, the
  unit transformation of the adjunction $s_i\dashv d_{i+1}$ induces an
  isomorphism
  \[
      \operatorname{Map}_{\eS_{n+1}^{\langle
          m+1\rangle}(\C)}(s_id_{i+1}(X),X)\xra{\simeq}\operatorname{Map}_{\eS_n^{\langle
          m+1\rangle}(\C)}(d_{i+1}(X),d_{i+1}(X)).
  \]
  in the homotopy category of spaces. Since the leftmost space is contractible
  by assumption, we conclude that $d_{i+1}(X)$ is a zero object of the stable
  $\infty$-category $\eS_n^{\langle m+1 \rangle}(\C)$. This shows that $\D$ is
  contained in the kernel of $d_{i+1}$. The statement of the proposition now
  follows from \th\ref{prop:ker_di}.
\end{proof}

\subsection{Tilting and cluster tilting for projective space}

In this section, we discuss another instance of the phenomenon that the presence
of an $m$-cluster tilting subcategory goes hand in hand with a diagrammatic
description of the derived category in terms of $(m+1)$-dimensional biCartesian
cubes.

Let $\KK$ be a field and $m$ a natural number. Let $\PP^m=\PP_{\KK}^m$ be the
projective $m$-space over $\KK$ and $\D\PP^m$ the derived $\infty$-category of
complexes of quasi-coherent sheaves on $\PP^m$, see Section 1.3 in \cite{Lur17}
for details. A famous theorem of Be\u{\i}linson \cite{Bei78} shows that the
object
\[
	T = \bigoplus_{0 \le i \le m} \O(i)
\]
is a compact generator of $\D\PP^m$; in fact, it is even a tilting bundle. As a
consequence, we obtain an equivalence of stable $\infty$-categories
\begin{equation}\label{eq:pntilting}
  \D\PP^m \simeq \D\End(T)^{\op}.
\end{equation}
Choosing homogeneous coordinates $x_0,x_1,\dots, x_m\in \Gamma(\O(1))$, the
algebra $\End(T)$, commonly referred to as the \emph{Beilinson algebra}, admits
the following combinatorial description. Denote by $B^{(m)}$ the category with
objects $0,1,\dots,m$, and morphisms from $i$ to $j$ given by monomials in
$\{x_0,x_1,\dots,x_m\}$ of degree $j-i$. The composition law in $B^{(m)}$ is
given by multiplication of monomials. It is then immediate that we have an
isomorphism
\[
	\End(T) \cong \End\KK B^{(m)}.
\]
In particular, we may reformulate the tilting equivalence \eqref{eq:pntilting}
as an equivalence
\begin{equation*}
  \label{eq:Beilinson}
	\D\PP^m\simeq \Fun(B^{(m)}, \D\KK).
\end{equation*}

\begin{remark}
  The algebra $\End(\KK B^{(m)})$ is the prototypical example of an
  $m$-representation infinite algebra in the sense of \cite{HIO14}, see Example
  2.15 therein. From the point of view of higher Auslander--Reiten
  theory, these algebras are higher-dimensional analogues of hereditary algebras
  of infinite representation type.
\end{remark}

Now, instead of the tilting bundle $T$, consider the subcategory
$\operatorname{Line}\PP^m$ of $\D\PP^m$ spanned by the line bundles $\O(i)$,
$i\in\ZZ$. The subcategory $\operatorname{Line}\PP^m$ admits a combinatorial
description in terms of the category $B_\ZZ^{(m)}$ with set of objects $\ZZ$ and
morphisms from $i$ to $j$ given by the set of monomials in
$\{x_0,x_1,\dots,x_m\}$ of degree $j-i$.

For every $i \in \ZZ$ and every sequence $a_0,\dots,a_m\geq 1$ of positive
integers, we exhibit a cube
\[
	q\colon I^{m+1} \lra B_\ZZ^{(m)}
\]
uniquely determined by the following properties:
\begin{enumerate}
\item We have $q_{(0,\dots,0)} = i$.
\item For $0 \le j \le m$, let $e_j$ denote the $j$th coordinate vector of
  $\ZZ^{m+1}$. Then every edge of $q$ that is parallel to $e_j$ gets mapped to
  multiplication by $x_j^{a_j}$ in $B_\ZZ^{(m)}$.
\end{enumerate}
We refer to the cubes of this form as the {\em rectilinear cubes} in
$B_\ZZ^{(m)}$. We call those cubes with $a_0=a_1=\dots=a_m=1$ {\em minimal}. For
example, for $m=2$ a typical minimal rectilinear cube in $B^{(m)}$ has the form
\begin{equation*}
  \begin{tikzcd}
    i\ar{rr}{x_0}\drar{x_1}\ar{dd}{x_2}&&i+1\drar{x_1}\ar{dd}[near end]{x_2}\\
    &i+1\ar[crossing over]{rr}[near start]{x_0}&&i+2\ar{dd}{x_2}\\
    i+1\ar{rr}[near end]{x_0}\drar{x_1}&&i+2\drar{x_1}\\
    &i+2\ar[leftarrow, crossing over]{uu}[near end]{x_2}\ar{rr}{x_0}&&i+3
  \end{tikzcd}
\end{equation*}

\begin{definition}
  Given a stable $\infty$-category $\C$, we call a functor $B_\ZZ^{(m)} \to \C$
  \emph{exact}, if it maps all rectilinear cubes in $B_\ZZ^{(m)}$ to biCartesian
  cubes in $\C$. We denote by
  \[
    \Fun^{\ex}(B_\ZZ^{(m)},\C) \subset \Fun(B_\ZZ^{(m)},\C)
  \]
  the full subcategory of exact functors.
\end{definition}

\begin{remark}
  Since every rectilinear cube can be decomposed into minimal rectilinear cubes,
  it is enough to require minimal rectilinear cubes to be sent to biCartesian
  cubes.
\end{remark}

In this context, we have the following analogue of \th\ref{prop:sliceres} and
\th\ref{prop:paracyclic}.

\begin{proposition}\th\label{prop:pncubes}
	Let $m \ge 1$, and let $\C$ be a stable $\infty$-category. The restriction
  functor
	\[
		\Fun^{\ex}(B_\ZZ^{(m)}, \C) \xra{\simeq} \Fun(B^{(m)}, \C)
	\]
	is an equivalence of $\infty$-categories.
\end{proposition}
\begin{proof}
  We shall exhibit a quasi-inverse to the restriction functor
  \[
    \Fun^{\ex}(B_{\ZZ}^{(m)},\C)\lra\Fun(B^{(m)},\C).
  \]
  For a subset $J\subseteq\ZZ$, we define $B_J^{(m)}$ to be the full subcategory
  of $B_\ZZ^{(m)}$ spanned by the objects in $J$. Factorise the inclusion
  $\iota$ into the composite
  \[
    B^{(m)}\xra{f}B_{[0,\infty)}^{(m)}\xra{g} B_{\ZZ}^{(m)}.
  \]
  We claim that the essential image of the fully faithful functor
  \[
    g_*f_!\colon\Fun(B^{(m)},\C)\lra\Fun(B_{\ZZ}^{(m)},\C)
  \]
  given by
  \begin{enumerate}[label=\arabic*.]
  \item left Kan extension along $f$, followed by
  \item right Kan extension along $g$.
  \end{enumerate}
  is precisely $\Fun^{\ex}(B_{\ZZ}^{(m)},\C)$. We make the following
  observations:
	\begin{itemize}
  \item A functor $F\colon B_{[0,\infty)}^{(m)}\to \C$ is a left Kan extension
    of its restriction $\restr {F}{B^{(m)}}\colon B^{(m)}\to\C$ if and only if
    for each $i>0$ the restriction of $F$ to $B_{[0,i+m+1]}^{(m)}$ is a left Kan
    extension of $\restr{F}{B_{[0,i+m]}^{(m)}}$.
  \item Let $i>0$. We identify the punctured cube $I_{|v|<m+1}^{m+1}$ with the
    poset of non-empty subsets of $[m]$. There is a cofinal functor
    $I_{|v|<m+1}^{m+1}\to\overcat{\left(B_{[0,i+m]}^{(m)}\right)}{i+m+1}$ given
    by mapping a subset $v\subsetneq[m]$ to the morphism
    \[
      i+|v|\xra{\prod_{k\notin v}x_k} i+m+1
    \]
    and an inclusion $v\subseteq w$ to the morphism
    \[
      i+|v|\xra{\prod_{k\in w\setminus v}x_k} i+|w|.
    \]
    The composite
    \[
      I_{|v|<m+1}^{m+1}\lra\overcat{\left(B_{[0,i+m]}^{(m)}\right)}{i+m+1}\lra
      B_{\ZZ}^{(m)}
    \]
    is a punctured minimal rectilinear cube in $B_{[0,\infty)}^{(m)}$. All
    punctured minimal rectilinear cubes in $B_{[0,\infty)}^{(m)}$ are of this
    form.
  \end{itemize}
  It follows that the essential image of $f_!$ consists of those functors
  $B_{[0,\infty)}^{(m)}\to\C$ whose restriction to each minimal rectilinear cube
  in $B_{[0,\infty)}^{(m)}$ is biCartesian. A similar argument shows that the
  essential image of $g_*$ consists precisely of those functors whose
  restriction to each minimal rectilinear cube in $B_{(-\infty, m]}$ is
  biCartesian. This finishes the proof because every minimal rectilinear cube in
  $B_\ZZ$ lies either in $B_{[0,\infty)}$ or in $B_{(-\infty,m]}$.
\end{proof}

The connection to higher Auslander--Reiten theory is as follows: As
shown in \cite{HIMO14} (see also Example 6.7 in \cite{Jas16}), the subcategory
$\operatorname{Line}\PP^m$ is an $m$-cluster tilting subcategory of the exact
category $\Vect\PP^m$ of vector bundles on $\PP^m$. Again, in our combinatorial
context the condition of homological purity in degree $m$ is reflected
geometrically by the fact that only $(m+1)$-dimensional cubes appear in the
exactness constraints of \th\ref{prop:pncubes}.

Various features of the discussion from \S \ref{subsec:derivedauslander}
reappear in the current context of projective $m$-space, be it in a somewhat
elementary fashion:
\begin{itemize}
\item In analogy to the combinatorial description of the Coxeter functor for
  higher Auslander algebras of type $\AA$, the autoequivalence of $\D(\PP^m)$
  given by the tensor action of the line bundle $\O(1)$ is described by pulling
  back along the shift autoequivalence $i \mapsto i+1$ of $B_\ZZ^{(m)}$. In
  other words, the action of the Picard group $\on{Pic}(\PP^m) \cong \ZZ$ admits
  a purely combinatorial description in terms of $B_\ZZ^{(m)}$.
\item The analogues of the slices defined in \S \ref{subsec:derivedauslander}
  are the full subcategories of $B_\ZZ^{(m)}$ spanned by the intervals $[j,
  j+m]$ of length $m$. The restriction functors from $\Fun^{\ex}(B_\ZZ^{(m)},
  \C)$ to each of these subcategories are equivalences and the procedure of
  mediating between these different slices can be regarded as a version of slice
  mutation. Note, however, that all slices are isomorphic to $B^{(m)}$.
\end{itemize}

\section{Mutation, horn fillings, and higher Segal conditions}
\label{sec:horn}

Let $n\geq m\ge 1$ and $\C$ be a stable $\infty$-category. In
\S\ref{sec:mutation} we established a combinatorial approach to slice mutation,
based on the equivalences
\[
	\Fun^{\ex}_*(\Delta(m,n), \C) \xra{\simeq} \Fun(\underline{S}, \C)
\]
for the various slices $S \subset \Delta(m,n)$. In \S \ref{ssec:Sk} we then
observed that, for a fixed natural number $m$, the various stable
$\infty$-categories $\Fun^{\ex}_*(\Delta(m,n), \C)$ organise into a simplicial
object $\eS_{\bullet}^{\langle m\rangle}(\C)$, which we called the
$m$-dimensional $\eS_{\bullet}$-construction of the stable $\infty$-category
$\C$.

The goal of this section is to analyse the interplay between our combinatorial
framework for slice mutation and the simplicial object
$\eS_{\bullet}^{\langle m\rangle}(\C)$. Our main insight is that slice mutation is intimately
related to
\begin{itemize}
\item certain $2$-categorical outer horn filling conditions and
\item the higher Segal conditions introduced in \cite{DK12,Pog17}.
\end{itemize}
We illustrate these relations by means of an explicit example. Consider the
slice in $\Delta(1,3)$ given by the subset $S = \{12,13,03\}$. We depict $S$ as
a subquiver of the Auslander--Reiten quiver of $\mmod A^{(1)}_3$ as follows:
\begin{center}
  \includegraphics{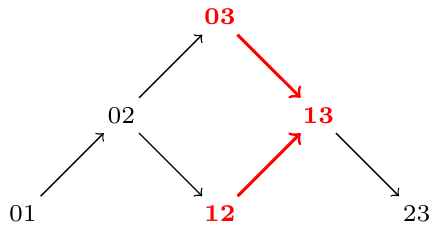}
\end{center}
Alternatively, we may emphasise the role of the elements of $\Delta(1,3)$ as
edges in the $3$-simplex and visualise the slice geometrically as
\begin{center}
  \includegraphics{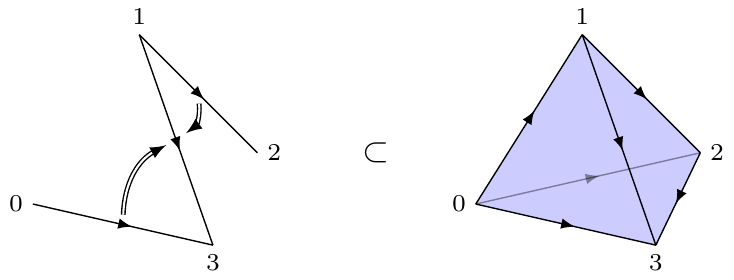}
\end{center}
where the notation for the arrows is chosen to signify their interpretation as
$2$-morphisms in the $2$-category $\DDelta$. In this presentation, it becomes
apparent that we may interpret $S$ as a $1$-dimensional subcomplex of the
$2$-dimensional subcomplex of $\Delta^3$ formed by the union of triangles
contained in the following triangulation of the square:
\begin{equation}
	\label{eq:slicesegal}
  \includegraphics{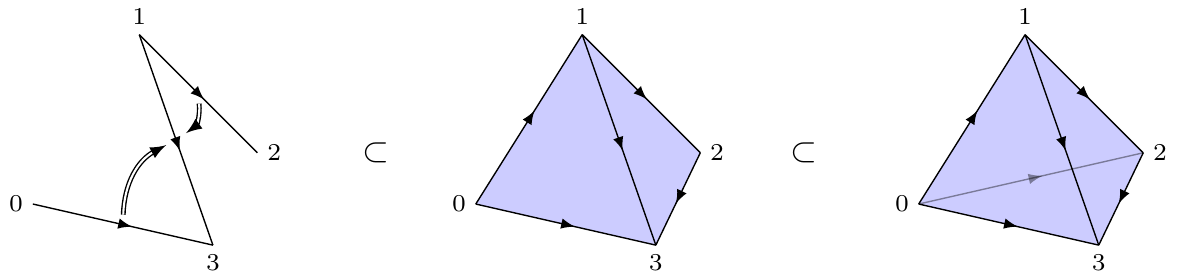}
\end{equation}

Let $\C$ be a stable $\infty$-category and denote the Waldhausen
$\eS_{\bullet}$-construction of $\C$ by $\X_{\bullet}$. The $2$-Segal property
of the $\eS_\bullet$-construction, as established in \cite{DK12}, provides an
equivalence of $\infty$-categories
\begin{equation}\label{eq:segal}
	\X_3 \xra{\simeq} \X_{ \{1,2,3\}} \times_{\X_{\{1,3\}}} \X_{ \{0,1,3\}}
\end{equation}
while restriction to the slice $S$ yields an equivalence
\begin{equation}\label{eq:slice}
  \X_3 \xra{\simeq} \Fun(12 \to 13 \leftarrow 03, \C).
\end{equation}
In comparing the equivalences \eqref{eq:segal} and \eqref{eq:slice}, we make the
following observations:
\begin{enumerate}
\item The description of the stable $\infty$-category $\X_3$ provided by the
  $2$-Segal equivalence \eqref{eq:segal} contains, in comparison with
  \eqref{eq:slice}, a certain amount of redundancy: the datum corresponding to
  each $2$-simplex amounts to biCartesian squares
  \[
    \begin{tikzcd}
      A_{01} \ar{r}\ar{d}\popb& A_{03}\ar{d}{f} \\
      0 \ar{r} & A_{13}
    \end{tikzcd}
    \qquad\text{and}\qquad
    \begin{tikzcd}
      A_{12} \ar{r}{g}\ar{d}\popb& A_{13}\ar{d} \\
      0 \ar{r} & A_{23}
    \end{tikzcd}
  \]
  in the stable $\infty$-category $\C$, while the data retained in the slice
  equivalence \eqref{eq:slice} are only the morphisms $f$ and $g$.
\item According to the pictorial presentation of \eqref{eq:slicesegal}, the
  discrepency between \eqref{eq:slice} and \eqref{eq:segal} can be accounted for
  by unique fillings of {\em outer} horns:
  \begin{equation}
    \label{eq:horns}
    \includegraphics{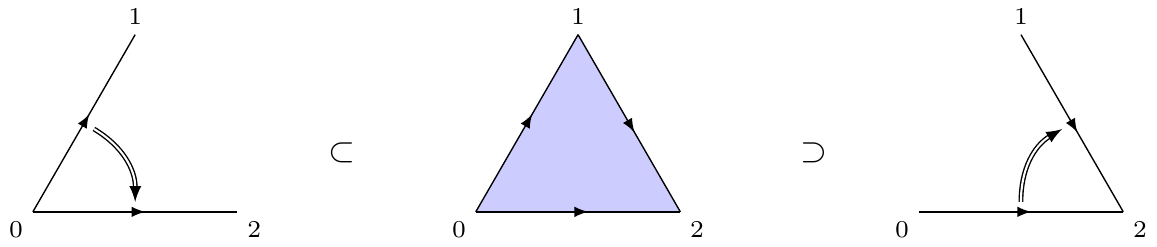}
  \end{equation}
\item\label{it:3} However, and this is a crucial point, the outer horns in
  \eqref{eq:horns} need to be considered in a suitable $2$-categorical
  framework: The data they capture is not just comprised of the two separate
  edges, but also includes a representation of the $2$-morphism that relates the
  edges.
\end{enumerate}

The starting point that enables us to realise the $2$-categorical horn filling
conditions of \ref{it:3} is the observation that, for every $m \ge 0$ and every
stable $\infty$-category $\C$, the simplicial stable $\infty$-category
$\eS_{\bullet}^{\langle m\rangle}(\C)$ admits a natural extension to a $2$-simplicial object
in the $\infty$-bicategory of stable $\infty$-categories (cf. \cite{Dyc17}).

However, in order to formulate the $2$-categorical horn filling conditions and
their relation to the Segal conditions one needs a systematic framework of
limits and Kan extensions in $\infty$-bicategories. Since, to our knowledge,
this theory has not been developed in the literature to the extent needed for
our purposes, we will postpone the realization of this program to a sequel
\cite{DJW} where we establish the necessary foundational results and establish
$2$-categorical counterparts of the results established in this section.

In this article, we allow ourselves to work in the $\infty$-categorical context
by means of the following modification of the above challenge: Instead of
studying $2$-simplicial stable $\infty$-categories $\X_\bullet$ themselves, we
focus on understanding the shadow of the interplay of slice mutation, horn
filling, and the higher Segal conditions after passing to
\begin{itemize}
\item Grothendieck groups $K_0(-)$ so that we obtain simplicial abelian groups,
\item $K$-theory spectra $K(-)$ so that we obtain simplicial objects in spectra.
\end{itemize}
To formulate our results in this context we introduce the following terminology,
{\em c.f.} \cite{DK12,Pog17}.

\begin{definition}
  Let $K$ be a (finite) simplicial set. Recall that the \emph{category of
    simplices of $K$} is the slice category $\overcat{\Delta}{K}$, where we view
  $\Delta$ as a full subcategory of the category of simplicial sets via the
  Yoneda embedding $[n]\mapsto\Delta^n$. Let $\C$ be an $\infty$-category and
  $X$ a simplicial object in $\C$. We obtain a diagram
  \[
    X|_{(\overcat{\Delta}{K})^{\op}}\colon(\overcat{\Delta}{K})^{\op}\lra\Delta^{\op}\xra{X}\C
  \]
  where $\overcat{\Delta}{K}\to\Delta$ is the forgetful functor. Assuming that
  $\C$ has (finite) limits, we further refer to the object
  \begin{equation*}
    X_K \coloneqq  \lim_{(\overcat{\Delta}{K})^{\op}} X|_{(\overcat{\Delta}{K})^{\op}}.
  \end{equation*}
  as the {\em object of $K$-membranes in $X$}.
\end{definition}

We are interested in the following descent conditions on simplicial sets, which
are made more precise in \S\ref{subsec:coverings} below.

\begin{definition}[Sketch] Let $\C$ be an $\infty$-category and let $X$ be a
  simplicial object in $\C$.
	\begin{enumerate}
  \item We say that $X$ {\em satisfies $m$-dimensional slice descent} if, for
    every $n > m$, and every $m$-dimensional slice $S \subset \Delta^n$, the
    restriction map
    \[
      X_n \lra X_S
    \]
    is an equivalence in $\C$.
  \item We say that $X$ is an {\em outer $m$-Kan complex}, or that it has {\em
      unique outer horn fillers above dimension $m$} if, for every $n > m$, the
    restriction maps
    \[
      X_n \lra X_{\Lambda_0^n} \quad \text{and} \quad X_n \lra X_{\Lambda_n^n}
    \]
    are equivalences in $\C$.
  \item We say that $X$ is a {\em $2m$-Segal object} if, for every $n > 2m$ and
    every triangulation $\T \subset \Delta^n$ of the $2m$-dimensional cyclic
    polytope $C([n],2m)$ on the vertices $[n]$, the restriction map
    \[
      X_n \lra X_\T
    \]
    is an equivalence in $\C$.\qedhere
	\end{enumerate}
\end{definition}

The interrelations between these notions are captured by the following summary
of results of this section:

\begin{theorem}\th\label{thm:relations}
	Let $\C$ be an $\infty$-category with finite limits, and let $X$ a simplicial
  object in $\C$. For a fixed $m > 0$, consider the following conditions on $X$:
	\begin{enumerate}[label=(\Roman*)]
  \item\label{it:x1} The simplicial object $X$ satisfies $m$-dimensional slice descent.
  \item\label{it:x2} The simplicial object $X$ satisfies outer horn descent above dimension $m$.
  \item\label{it:x3} The simplicial object $X$ is a $2m$-Segal object.
	\end{enumerate}
	Then the following hold:
	\begin{enumerate}
  \item We always have the implications \ref{it:x1} $\lRa$ \ref{it:x2}
    $\lRa$ \ref{it:x3}.
  \item Assume $m=1$. Then we have \ref{it:x1} $\lLRa$ \ref{it:x2}.\qedhere
  \end{enumerate}
\end{theorem}

In \cite{DJW}, we will show that \th\ref{thm:relations} remains true, {\em mutatis mutandis}, for
$2$-simplicial objects with values in $\infty$-bicategories. In particular, via the categorified
Dold-Kan correspondence of \cite{Dyc17}, we will obtain a classification of $2m$-Segal objects in
the $\infty$-bicategory of stable $\infty$-categories.

\subsection{Membranes, covers and refinements}
\label{subsec:coverings}

We introduce the following framework which allows us to work efficiently with
objects of membranes associated to simplicial subsets of the standard simplices.

\begin{definition}
  A \introduce{cover} of $[n]\in\Delta$ is a collection $\F$ of subsets of $\n$
  such that $\bigcup\F=\n$. We write $\F\covers\n$ to indicate that $\F$ is a
  cover of $[n]$. Given two covers $\F'$ and $\F$ of $n$, we say that $\F'$ is a
  \introduce{refinement} of $\F$ if for every $f'\in\F'$ there is an $f\in\F$
  with $f'\subseteq f$. In this case we write $\F'\refines\F$.
\end{definition}

Every cover $\F\covers \n$ refines the trivial cover $\{\n\}\covers\n$; if
the converse is true (i.e.\ $\n\in \F$) then we call $\F\covers\n$ a
\introduce{degenerate cover}. More generally, we call a refinement
$\F'\refines\F$ \introduce{degenerate}, if $\F$ also refines $\F'$. For each
$\n\in\Delta$, covers of $\n$ and refinements between them are naturally
assembled into a category $\Coverings\n$ which is equivalent to a poset; the
isomorphisms are the degenerate refinements. The categories $\Coverings\n$ can
be assembled into a Cartesian fibration
\begin{equation}
  \label{eq:U}
  \bigcup\colon\CovD\lra\Delta,
\end{equation}
where there is a unique morphism from $\F'\covers\n$ to $\F\covers\m$ over
$\alpha\colon \m\to\n$ provided that for each $f'\in\F'$ there is an $f\in\F$
such that $\alpha(f')\subseteq f$. A cover $\F\covers\n$ gets mapped to
$\n=\bigcup\F\in\Delta$. The assignment which sends $\n$ to the trivial cover
$\{\n\}\covers\n$ gives rise to a fully faithful right adjoint
\[
  \{\}\colon\Delta\hra \CovD
\]
to the functor \eqref{eq:U}. From now on we will identify $\Delta$ with its image in $\CovD$.\\

Given an $\infty$-category $\C$ with finite limits, right Kan extension along
$\{\}^{\op}$ yields a fully faithful embedding
\[\Fun(\Dop,\C)\hra \Fun(\CovDop,\C)\]
given by evaluating a functor $X$ on a cover $X\covers\n$
via the formula
\[\Xcover{\F}\simeq \lim\left(\left(\overcat{\Delta}{\F}\right)^{\op}\to
    \Dop\xra{X}\C\right).\] Every cover $\F\covers[n]$ gives rise to a
simplicial subset
\[
  K_\F\colon\n\mapsto \Hom_{\CovD}([n], \F)
\]
of $\Delta^n$. Two covers of $\n$ give rise to the same simplicial set if and
only if they are isomorphic in $\CovD$. It is immediate from the definition that
$\Xcover{\F}$ agrees with the membrane space $X_{K_{\F}}$ defined above.

Given a cover $\F\covers \n$, we denote by $\minsubsets\F\subset\powerset\n$ the subposet of those non-empty subsets of $\n$ which are of the form $\bigcap F$ for some $\emptyset\neq F\subseteq\F$.

\begin{remark}
The fully faithful inclusion $\minsubsets\F\hra \overcat\Delta\F$ has a left adjoint given on objects by
\[
  (\alpha\colon \m\to\F)\mapsto\bigcap\setP{f\in \F}{\alpha(\m)\subseteq f}.
\]
In particular, by cofinality, we can compute $\Xcover{\F}$ as
\[\Xcover{\F}\simeq\lim\left(\minsubsets\F^{\op}\hra\Delta^{\op}\xra X\C\right).\qedhere\]
\end{remark}

\subsubsection{Descent}

Let $\C$ be an $\infty$-category with all finite limits and
$X\colon\Delta^{\op}\to\C$ a simplicial object.

\begin{definition}
  We say that a cover $\F\covers\n$ is \introduce{$X$-local}, or that $X$
  \introduce{satisfies descent} with respect to $\F$, if $X$, viewed as a
  functor $\CovDop\to\C$ sends the structure map $\F\to\n$ to an equivalence
  \[
    X_n\lra \Xcover{\F}.
  \]
  More generally, we say that a refinement of covers $\F'\refines\F$ is
  \introduce{$X$-local} if it is sent by $X$ to an equivalence
  \[
    \Xcover{\F}\lra \Xcover{\F'}
  \]
  in $\C$.
\end{definition}

All of the conditions we impose on simplicial objects in this article are
descent-conditions with respect to certain specific covers. We will now develop
some basic techniques for comparing different such descent conditions to each
other.

\begin{remark}
  \th\label{rem:deglocals}
  Degenerate refinements $\F'\degrefines\F$ are precisely the isomorphisms
  in the category $\CovD$, hence they are $X$-local for every $X$. In
  particular, every degenerate cover is $X$-local for every $X$.
\end{remark}

\begin{lemma}\th\label{lem:descentrefinement}
  Let $\F'\refines\F\covers\n$ be a refinement of covers. Assume that for
  every $F\subseteq\F$ with $\emptyset\neq I\coloneqq \bigcap F$, the functor $X\colon\Dop\to\C$
  satisfies descent with respect to the cover
  \[\coverpb{F'} I\coloneqq \setP{f'\cap I}{ f'\in\F'}\covers I.\]
  Then $\F'\refines\F$ is $X$-local.
\end{lemma}

\begin{proof}
  Consider the composite functor
  \[g\colon\left(\overcat{\Delta}{\F'}\right)^{\op}\lra\left(\overcat{\Delta}{\F}\right)^{\op}\lra
    \minsubsets\F^{\op}\] where the first map is induced by the refinement
  $\F'\refines\F$ and the second one is (the opposite of) the left adjoint to
  $\minsubsets\F\hra \overcat\Delta\F$.

  We claim that $\restr X{\minsubsets\F^{\op}}$ is a right Kan extension of
  $\restr X{\left(\overcat{\Delta}{\F'}\right)^{\op}}$ along $g$. For each
  $I=\bigcap F\in \minsubsets\F$, we have to prove that the natural map
  \[\restr X{\minsubsets\F^{\op}}(I)\lra g_*\left(\restr
      X{\left(\overcat\Delta{\F'}\right)^{\op}}\right)(I)\] (induced by unit of
  the adjunction $\overcat{\Delta}\F\leftrightarrow \minsubsets\F$) is an
  equivalence. Observe that for each $I\in\minsubsets\F$ we have
\[
  \undercat
  I{\left(\overcat\Delta{\F'}\right)^{\op}}=\left(\overcat\Delta{\coverpb{F'}
      I}\right)^{\op}
\] hence using the pointwise formula for Kan extensions we
  reduce to showing that the natural map
  \[X(I)\lra \lim\left(\left(\overcat\Delta{\coverpb{F'} I}\right)^{\op}
      \hra{\Delta^{\op}}\xra X\C\right)\simeq X(\coverpb{F'} I)\] is an equivalence;
  this is precisely the assumption on $X$.  The result follows because we have
  \[\Xcover{\F'}\simeq\lim(\restr
  X{\left(\overcat\Delta{\F'}\right)^{\op}})\simeq\lim(\restr X{\minsubsets\F})\simeq \Xcover{\F}.\qedhere\]
\end{proof}

\begin{remark}\th\label{rem:descentrefinementenough}
If there is an $f\in F$ which is contained in some $f'\in\F'$ then
$\coverpb{F'}{\left(\bigcap F\right)}\covers \bigcap F$ is a degenerate cover, hence the
assumption of \th\ref{lem:descentrefinement} is automatic in this case.
\end{remark}

\begin{notation}
  Given a cover $\F'\covers\n$ and a subset $f\subset \n$, we introduce a new cover
  \[\coveradj{\F'} f\coloneqq \{f\}\cup \setP{f'\in\F'}{f'\not\subseteq f}\covers\n\]
  by adding the element $f$ and then removing the redundant elements of $\F'$.
  We then have a refinement $\F'\refines\coveradj{\F'} f$.
\end{notation}

\begin{corollary}\th\label{lem:descentsmallrefinement}
 Let $\F'\covers\n$ be a cover and $f\subset\n$ a subset.
 If $\coverpb{\F'} f\covers f$ is an $X$-local cover then
 $\F'\refines\coveradj{\F'}f$ is an $X$-local refinement.
\end{corollary}
\begin{proof}
 Follows by direct application of \th\ref{lem:descentrefinement} since, by
 \th\ref{rem:descentrefinementenough}, we only need to check the case $F=\{f\}$,
 which is true by assumption.
\end{proof}

Let $\localP$ be a set of covers, which we identify with the full subcategory
$\localP\subset \CovD$ which they span. We denote by $\localP_\n$ the set/category of covers
$\F\covers\n$ in $\localP$, i.e.\ the fibre over $\n$ of the composition $\localP\subset\CovD\xra{\bigcup}\Delta$.

\begin{definition}
 We say a cover $\F\covers\n$ is \introduce{$\localP$-local}, if for every
 $X\colon{\Delta^{\op}}\to\C$ we have the implication
 \[{\text{every $\F'\in\localP$ is $X$-local}}\Longrightarrow
   {\text{$\F$ is $X$-local}}\]
 We denote the set of all $\localP$-local covers by $\closure\localP$.
\end{definition}

\begin{definition}\th\label{defi:saturatedcoverclass}
 Let $\localP$ be a set of non-degenerate covers. We say that $\localP$ is
 \introduce{saturated} if it has the following properties:
  \begin{enumerate}
  \item\label{saturated:step} Given a non-degenerate refinement
    $\F'\refines\F$ of covers in $\localP_n$, there is an element $f\in\F$ such
    that
    \begin{itemize}
    \item The cover $\coveradj{\F'}f$ is isomorphic to a cover in $\localP_n$.
    \item The cover $\coverpb{\F'}f\covers f$ is isomorphic to a cover in $\localP_f$.
    \end{itemize}
  \item For each $\n\in\Delta$, the category $\localP_\n$ is connected.\qedhere
  \end{enumerate}
\end{definition}

\begin{lemma}[Refinement Principle]\th\label{lem:refinementprinciple}
  Let $\localP$ be a saturated cover. For each $n\in\BN$ where $P_\n$ is
  non-empty, choose an element $p_n\in\localP_\n$. Then we have
  \[\closure\localP = \closure{\setP{p_n}{n\in\BN, P_\n\neq\emptyset}}.\qedhere\]
\end{lemma}

\begin{proof}
  The inclusion ``$\supseteq$'' is trivial. To prove that the converse inclusion
  holds, assume that $X$ satisfies descent with respect to the covers
  $p_n\covers \n$; we prove by induction that for each $n$ all covers in
  $\localP_n$ are $X$-local.
  Let $\F'\refines\F$ be a refinement in $\localP_n$. By repeated application of
  assumption \ref{saturated:step}, we can write this refinement
  as a composition
  \[\F'=\F'_0\refines(\coveradj{\F'_0}{f_1})\degrefines\F'_1
    \refines(\coveradj{\F'_1}{f_2})\degrefines\F'_2\dots\refines(\coveradj{\F'_{l-1}}{f_l})\degrefines\F\]
  where each $\F'_i$ lies in $\localP$ and where all the covers
  $\coverpb{\F'_{i-1}}{f_{i}}\covers f_{i}$ lie in
  $\localP$. Since $f_i$ is a proper subset of $\n$, we know by induction that
  the covers $\coverpb{\F'_{i-1}}{f_{i}}\covers f_{i}$ are $X$-local.
  It follows from \th\ref{lem:descentsmallrefinement}
  that all the refinements $\F'_{i-1}\refines\coveradj{\F'_{i-1}}{f_i}$ are $X$-local; hence
  $\F'\refines\F$ is $X$-local. Since the category
  $\localP_\n$ is connected and every refinement in $\localP_\n$ is $X$-local,
  we conclude that all covers in $\localP_\n$ are
  $X$-local if any one of them (\emph{e.g.}\ $p_n$) is.
\end{proof}

\subsection{Outer horn descent and slice mutations}

We introduce the following classes of covers.

\begin{definition}
  \th\label{def:slicesNhorns}
  Let $m\geq 0$ and consider the following classes of covers:
  \begin{itemize}
  \item The class of \introduce{left horns above dimension $m$} is
\[\lefthorns m\coloneqq \setP{\horn n 0\covers \n} {n > m}.\]
  \item The class of \introduce{right horns above dimension $m$} is
\[\righthorns m\coloneqq \setP{\horn n n\covers \n} {n > m}.\]
  \item The class of \introduce{projective $m$-slices} is
\[\projslices m\coloneqq \setP{\projslice m n\coloneqq\setP{f\subset\n}{0\in f, |f|=m+1}\covers\n}{n> m}.\]
  \item The class of \introduce{injective $m$-slices} is
\[\injslices m\coloneqq \setP{\injslice m n\coloneqq\setP{f\subset\n}{n\in f, |f|=m+1}\covers\n}{n> m}.\]
  \item The class of all \introduce{$m$-dimensional slices} is
\[\allslices m\coloneqq \{S\covers \n\}\]
 where $S$ ranges over all those slices $S\subset\Li(m,n)^\sharp$
    (as in Definition~\ref{defi:slice}) which are contained in $\Delta(m,n)^\sharp$.
    Here we identify elements of $\Delta(m,n)^\sharp$ with their image (a subset of
    $\n$ of cardinality $m+1$), so that we can interpret $S\subset\Delta(m,n)$ as a cover $S\covers\n$.\qedhere
  \end{itemize}
\end{definition}

\begin{notation}\th\label{}
  Consider the class $\leftKcoverings m$ of those non-degenerate covers
  $\F\covers \n$ satisfying:
  \begin{enumerate}
  \item $0\in \bigcap \F$
  \item For every $I\subset \n$ with $|I|=m$ there is a $f\in\F$ with
    $I\subset f$.
  \end{enumerate}
  We define the class $\rightKcoverings m$ analogously with $n\in\bigcap\F$
  instead of $0\in\bigcap\F$.
\end{notation}

\begin{proposition}
  \th\label{prop:outerhornslices}
  Let $\C$ be an $\infty$-category with finite limits and $X$
    a simplicial object in $\C$.
  \begin{enumerate}
  \item The following statements are equivalent:
    \begin{itemize}
    \item The simplicial object $X$ satisfies left horn descent above dimension $m$.
    \item The simplicial object $X$ satisfies descent with respect to all covers in
      $\leftKcoverings m$.
    \item The simplicial object $X$ satisfies $m$-dimensional projective slice descent.
    \end{itemize}
  \item The following statements are equivalent
    \begin{itemize}
    \item The simplicial object $X$ satisfies right horn descent above dimension $m$.
    \item The simplicial object $X$ satisfies descent with respect to all covers in
      $\rightKcoverings m$.
    \item The simplicial object $X$ satisfies $m$-dimensional injective slice descent.\qedhere
    \end{itemize}
  \end{enumerate}
\end{proposition}

\begin{proof}
  For $n\leq m$, the category $\leftKcoverings m_\n$ is empty. Observe that for $n > m$, the covers $\horn n 0$ and $\projslice m n$ are
  terminal and initial in $\leftKcoverings m_\n$, respectively. It is
  straightforward to check condition \ref{saturated:step}
  of \th\ref{defi:saturatedcoverclass} so that
  $\leftKcoverings m$ is a saturated class. Two applications
  of the refinement principle (\th\ref{lem:refinementprinciple}) yield
  $\closure{\lefthorns m} = \closure{\leftKcoverings m}  =\closure{\projslices m}$, hence
  the first claim. The second claim is analogous using the saturated class $\rightKcoverings m$.
\end{proof}

\begin{lemma}\th\label{lem:1dimslices}
  A cover $\F\covers \n$ is a equal to a $1$-dimensional slice $S\subset\Delta(1,n)$ if
  and only if $\F$ is the trivial cover $\{[1]\}\covers[1]$ or it is of the
  form $\F=\F'\cup\{f\}$ where $f=[n-1]$ or $f=\{1,\dots,n\}$.
\end{lemma}
\begin{proof}
  The ``if'' direction is clear. To prove the converse, let $S\subset\Delta(m,n)^\sharp$ be a slice. Since $\{0,n\}$ is the unique element in
  its $\Phi_m^{-1}$-orbit, we must have $\{0,n\}\in S$. If $n=1$, then we are
  done; otherwise exactly one of $\{0,n-1\}$ and $\{1,n\}$ must be in $S$,
  without loss of generality the first case. No element of the form $\{i,n\}$ with
  $i>0$ can then be in $S$ because otherwise by convexity also $\{1,n\}\in S$
  which is false. We conclude that $S\setminus\{\{0,n\}\}$ is a subset of
  $\Delta(m,n-1)^\sharp$ where it still satisfies the axioms to be a slice $S'$.
  The result follows by induction.
\end{proof}

\begin{proposition}
  If a simplicial object $X$ satisfies $m$-dimensional slice
  descent then it also satisfies outer horn descent above dimension $m$.
  The converse ``only if'' is true for $m=1$.
\end{proposition}
\begin{proof}
  The first statement follows by combining the two parts of
  \th\ref{prop:outerhornslices}. To prove the converse, assume that
  $X$ satisfies outer horn descent above dimension $1$ and let
  $S\covers\n$ be a $1$-dimensional slice.
  We write $S=S'\cup \{f\}$ as in \th\ref{lem:1dimslices} and observe that
  by induction the cover $\F\cap f\degrefines \F'\refines f$ is
  $X$-local; hence by \th\ref{lem:descentsmallrefinement} we conclude
  that the refinement $\F\refines \F\cup\{f\}\cong\{\{0,n\},f\}$ is
  $X$-local. The result follows because the covers
  $\{\{0,n\},[n-1]\}\covers\n$ and $\{\{0,n\}, \{1,\dots,n\}\}\covers\n$ are
  $\righthorns 1$-local and $\lefthorns 1$-local, respectively, by
  \th\ref{prop:outerhornslices}.
\end{proof}

\subsection{Higher Segal objects}

We recall from \cite{DK12} and \cite{Pog17} the notion of a higher
Segal object in an $\infty$-category with finite limits. For the sake of brevity
we keep our exposition rather terse.

\begin{definition}\th\label{def:Tnk}
  Let $I\subset\n$ be a subset of $\n$. An index $i\in\n\setminus I$ is called
  an \introduce{odd gap of $I$} if the cardinality of the set $\setP{j\in
    I}{j>i}$ is odd and an \introduce{even gap of $I$} if this cardinality is
  even.
  We define the following covers:
  \begin{itemize}
  \item The \introduce{upper $k$-Segal cover} $\upperScovering nk \covers\n$ consists of all
    even subsets $I\subset\n$ of cardinality $k+1$.
  \item The \introduce{lower $k$-Segal cover} $\lowerScovering nk \covers\n$ consists of all
    even subsets $I\subset\n$ of cardinality $k+1$.\qedhere
  \end{itemize}
\end{definition}

\begin{remark}
  Let $n\geq k\geq 1$ be integers. The sets $\upperScovering nk$ and
  $\lowerScovering nk$ can be
  identified, respectively, with the minimal and the maximal element of the set
  of triangulations of a $k$-dimensional cyclic polytope on the vertex set $[n]$
  with respect to an appropriate partial order, see Section 6.1 in \cite{DLRS10}
  for details.
\end{remark}

\begin{definition}\th\label{def:k-Segal}
  Let $\C$ be an $\infty$-category with finite limits, $X$ a
  simplicial object in $\C$ and $k\geq 1$ an integer.
  \begin{enumerate}
  \item The simplicial object $X$ is called \introduce{lower $k$-Segal} if
    it satisfies descent with respect to all lower $k$-Segal covers
    $\upperScovering nk\covers\n$.
  \item The simplicial object $X$ is called \introduce{upper $k$-Segal} if
    it satisfies descent with respect to all upper $k$-Segal covers
    $\upperScovering nk\covers\n$.
  \item The simplicial object $X$ is called \introduce{$k$-Segal} if
    is both lower and upper $k$-Segal.\qedhere
  \end{enumerate}
\end{definition}

Before moving on, we recall from \cite{Pog17} the following explicit characterization of the lower and upper
$2m$-Segal covers.

\begin{remark}\th\label{rem:Segalexplicit}
  Let $n\geq 2m-1$.
  \begin{enumerate}
  \item The lower $(2m-1)$-Segal cover $\lowerScovering n{2m-1}\covers \n$
    consists precisely of subsets $I\subseteq\n$ of the form
    \[I= \bigcup_{j=1}^m\{i_j,i_j+1\}\]
  where the $i_j$ are such that
  this union is disjoint (i.e.\ $0<i_j<i_j+1<i_{j+1}$ for all $j$).
\end{enumerate}
  Let $n\geq 2m$.
\begin{enumerate}
  \item The lower $2m$-Segal cover $\lowerScovering n{2m}\covers [n]$
    consists precisely of subsets $I\subseteq\n$ of the form
    \[I= \{0\}\cup I'\]
    where $I'$ is an element of the lower $(2m-1)$-Segal cover $\lowerScovering {n-1}{2m-1}\covers\{1,\dots,n\}$.
  \item The upper $2m$-Segal cover $\upperScovering n{2m}\covers [n]$
    consists precisely of subsets $I\subseteq\n$ of the form
    \[I= I'\cup\{n\}\]
    where $I'$ is an element of the lower $(2m-1)$-Segal cover $\lowerScovering {n-1}{2m-1}\covers[n-1]$.\qedhere
\end{enumerate}
\end{remark}

\subsubsection{Outer horn descent and even-dimensional Segal objects}

The following result relates outer horn descent conditions to even Segal
conditions in $\infty$-categories with finite limits.

\begin{theorem}
  \th\label{prop:outer_horns} Let $\C$ be an $\infty$-category with finite
  limits, let $m\geq1$ be an integer, and $X$ a simplicial object in $\C$.
  \begin{enumerate}
  \item\label{inthm:leftlowerevenSegal} Suppose that $X$ satisfies left horn
    descent above dimension $m$. Then $X$ is a lower $2m$-Segal object in $\C$.
  \item\label{inthm:rightupperevenSegal} Suppose that $X$ satisfies right horn
    descent above dimension $m$. Then $X$ is an upper $2m$-Segal object in $\C$.
  \end{enumerate}
  In particular, if $X$ satisfies outer horn descent above dimension $m$, then
  $X$ is a $2m$-Segal object in $\C$.
\end{theorem}
\begin{proof}
  Using the explicit description of the Segal covers
  (\th\ref{rem:Segalexplicit}) it is clear that $\lowerScovering n{2m}$
  belongs to the class $\leftKcoverings m$ and $\upperScovering n{2m}$ belongs
  to the class $\rightKcoverings m$. Hence claims \ref{inthm:leftlowerevenSegal}
  and \ref{inthm:rightupperevenSegal} follow from \th\ref{prop:outerhornslices}.
  The last statement is an immediate consequence of the first two.
\end{proof}

\subsubsection{Inner horn descent and odd-dimensional Segal objects}

\begin{notation}
  Let $m\geq 0$. Consider the following classes of covers:
  \begin{itemize}
  \item The class of \introduce{almost right horns above dimension $m$} is
    \[\alrighthorns m\coloneqq \setP{\horn n {n-1}\covers \n} {n > m}.\]
  \item The class of \introduce{almost left horns above dimension $m$} is
    \[\allefthorns m\coloneqq \setP{\horn n {1}\covers \n} {n > m}.\]
  \item The class of \introduce{inner horns above dimension $m$} is
    \[\innerhorns m\coloneqq \setP{\horn n i\covers \n} {0<i<n > m}.\qedhere\]
  \end{itemize}
\end{notation}

\begin{theorem}\th\label{thm:alrighthorns} Let $\C$ be an $\infty$-category with
  finite limits, let $m\geq 1$ be an integer and $X$ a simplicial object in $\C$.
  Suppose that $X$ satisfies almost right
    horn descent above dimension $m$ \emph{or} almost left horn descent above dimension $m$. Then $X$ is lower $(2m-1)$-Segal.

  In particular, if $X$ satisfies inner
    horn descent above dimension $m$ then $X$ is lower $(2m-1)$-Segal.
\end{theorem}

\begin{proof}
  We treat the case of almost right horns; the case of almost left horns is analogous.
  We fix the following notation:
  \begin{itemize}
  \item Let $\localP$ be the class consisting of those non-degenerate covers $\F\covers\n$ such
    that
    \begin{itemize}
    \item $[n-1]=\{0,\dots, n-1\}\in\F$.
    \item for all $[n-1]\neq f\in\F$ we have $n-1,n\in f$.
    \item for all $I\subset \n$ with $|I|=m$ there is an $f\in\F$ with
      $I\subseteq f$.
    \end{itemize}
  \item For each $n>2m-1$, put $U^n\coloneqq \coveradj{\upperScovering
      n{2m-1}}{[n-1]}\covers\n$ and
    $\localE_n\coloneqq \{\upperScovering n{2m-1}\refines U^n\}$
  \end{itemize}
  Observe that for each $n>m$ the almost right horn $\horn n {n-1}$ is terminal in
  $\localP_\n$; it is straightforward to check that $\localP$ is saturated.
  The class $\localE\coloneqq\bigcup_{n>2m-1}\localE_n$ is also saturated
  because the cover $\coverpb{\upperScovering n {2m-1}}{[n-1]}\covers[n-1]$ can be
  identified with $\upperScovering {n-1}{2m-1}\covers[n-1]$ (which is the
  trivial cover for $n=2m$).
  Observe, finally, that for all $n>2m-1$, the cover
  $U^n\covers\n$ lies in $\localP_n$. We
  conclude \[\closure{\alrighthorns m} = \closure{\localP}\supset
    \closure{\setP{U^n}{n>2m-1}}=\closure{\localE} =
    \closure{\setP{\upperScovering n{2m-1}}{n>2m-1}}\] by repeated application
  of the refinement principle.
\end{proof}

\section{Classification via the Dold--Kan correspondence}
\label{sec:dk}

We introduce the following terminology.

\begin{definition}
  \th\label{def:outer_m-Kan_cpx} Let $\C$ be an $\infty$-category with finite
  limits and $m\geq0$. An \emph{outer $m$-Kan complex in $\C$} is a simplicial
  object $X$ which satisfies descent with respect to all covers in $\lefthorns m
  \cup \righthorns m$, see \th\ref{def:slicesNhorns}. In other words, $X$
  satisfies both left and right horn descent above dimension $m$.
\end{definition}

Let $m$ be a positive integer. In this section we show that, if $\C$ is an
abelian category or a stable $\infty$-category, then the classes of outer
$m$-Kan complexes and of $2m$-Segal objects in $\C$ coincide, see
\th\ref{thm:abelianSegalch,thm:stableSegalch}. Our proofs rely on a
characterisation of the outer $m$-Kan complexes expressed in terms of
appropriate versions of the Dold--Kan correspondence, see
\th\ref{thm:outer-Kan-DK-abelian,thm:k-groupoids-k-truncted_complex:infty-cats}.

\subsection{Simplicial objects in abelian categories}

\subsubsection{The Dold--Kan correspondence}

Let $\A$ be an abelian category and $X$ a simplicial object in $\A$. Recall that
the Moore chain complex of $X$ is the connective chain complex $(X,\partial)$
where $\partial\colon X_n\to X_{n-1}$ is given by the formula
\begin{equation*}
  \partial \coloneqq  \sum_{i=0}^{n}(-1)^id_i.
\end{equation*}
The \emph{normalised Moore chain complex of $X$} is the subcomplex
$(\overline{X},\partial)$ of $(X,\partial)$ given by
\begin{equation*}
  \overline{X}_n\coloneqq \bigcap_{i=1}^n\ker d_i.
\end{equation*}
In particular $\partial=d_0$ on $\overline{X}_n$. The passage from a simplicial
object to its normalised Moore chain complex yields a functor
\begin{equation*}
  C\colon\A_\Delta\lra\Ch_{\geq0}(\A)
\end{equation*}
from the category $\A_\Delta$ of simplicial objects in $\A$ to the category
$\Ch_{\geq0}(\A)$ of connective chain complexes in $\A$. The functor
$C$ admits a right adjoint
\begin{equation*}
  N\colon\Ch_{\geq0}(\A)\lra\A_\Delta
\end{equation*}
which associates to a connective chain complex in $\A$ its \emph{Dold--Kan
  nerve}. The Dold--Kan correspondence \cite{Dol58,Kan58} is the following
fundamental equivalence of categories.

\begin{theorem}[Dold--Kan correspondence]
  \th\label{thm:dk-abelian} The functors
  \begin{equation*}
    \begin{tikzcd}
      C\colon\A_\Delta\rar[shift left]&\Ch_{\geq0}(\A)\colon N\lar[shift left]
    \end{tikzcd}
  \end{equation*}
  form a pair of adjoint equivalences of categories.
\end{theorem}

\begin{remark}
  \th\label{rmk:d0dn-abelian} Let $\A$ be an abelian category and $X$ a
  simplicial object in $\A$. Up to natural isomorphism, the normalised Moore
  chain complex of $X$ can be alternatively defined by setting
  \begin{equation*}
    \overline{X}_n\coloneqq \bigcap_{i=0}^{n-1}\ker d_i
  \end{equation*}
  with the differential given by $d_n\colon \overline{X}_n\to
  \overline{X}_{n-1}$.
\end{remark}

\subsubsection{Outer $m$-Kan complexes in abelian categories}

\begin{definition}
  \th\label{def:k-truncated_complex} Let $\A$ be an abelian category and $m$ a
  natural number. A connective chain complex $X$ in $\A$ is \emph{(strictly)
    $m$-truncated} if for each $n>m$ the object $X_m$ is a zero object of $\A$.
\end{definition}

Outer $m$-Kan complexes in $\A$ admit the following simple characterisation.

\begin{theorem}
  \th\label{thm:outer-Kan-DK-abelian} Let $\A$ be an abelian category and $m$
  natural number. The Dold--Kan correspondence
  \begin{equation*}
    \begin{tikzcd}
      C\colon\A_\Delta\rar[shift left]&\Ch_{\geq0}(\A)\colon N\lar[shift left]
    \end{tikzcd}
  \end{equation*}
  restricts to an equivalence of categories between the full subcategory of
  $\A_\Delta$ spanned by the outer $m$-Kan complexes and the full subcategory of
  $\Ch_{\geq0}(\A)$ spanned by the $m$-truncated chain complexes.
\end{theorem}

In fact, \th\ref{thm:outer-Kan-DK-abelian} is an immediate consequence of the
following general statement.

\begin{proposition}
  \th\label{prop:cofibre_sequence:abelian-cats} Let $\A$ be an abelian category
  and $X$ a simplicial object in $\A$.
  \begin{enumerate}
  \item For each $n\geq 1$ there are split short exact sequences
    \begin{equation*}
      0\to\overline{X}_n\to X_n\to X_{\Lambda_0^n}\to0\qquad\text{and}\qquad0\to\overline{X}_n\to X_n\to X_{\Lambda_n^n}\to0,
    \end{equation*}
    where $(\overline{X},\partial)$ is the normalised Moore chain complex of
    $X$ and $X_n\to X_{\Lambda^n_0}$ (resp.\ $X_n\to X_{\Lambda^n_n}$) is the
    Segal map.
  \item For a non-empty subset $J\subseteq[n]$ we write
    $\overline{X}_{J}\coloneqq\overline{X}_{|J|-1}$.
    \begin{enumerate}
    \item There is a direct sum decomposition
      \begin{equation}
        \label{eq:0DK-sum-ab}
        X_{\Lambda_0^n}\cong\bigoplus_{0\in J\subsetneq[n]}\overline{X}_J,
      \end{equation}
      with the following property: For each $i\neq 0$, the action of the
      $i$-th face map on $X_n\cong\overline{X}_n\oplus X_{\Lambda_0}^n$ corresponds
      to the projection onto those direct summands $\overline{X}_J$ such that
      $J$ does not contain $i$.
    \item There is a direct sum decomposition
      \begin{equation}
        \label{eq:nDK-sum-ab}
        X_{\Lambda_n^n}\cong\bigoplus_{n\in J\subsetneq[n]}\overline{X}_J,
      \end{equation}
      with the following property: For each $i\neq n$, the action of the
      $i$-th face map on $X_n\cong\overline{X}_n\oplus X_{\Lambda_n}^n$ corresponds
      to the projection onto those direct summands $\overline{X}_J$ such that
      $J$ does not contain $i$.\qedhere
    \end{enumerate}
  \end{enumerate}
\end{proposition}

\begin{remark}
  Let $\A$ be an abelian category, $X$ a simplicial object in $\A$ and
  $(\overline{X},\partial)$ its normalised Moore chain complex. Consider the
  commutative diagram
  \[
    \begin{tikzcd}
      \bigoplus\limits_{[n]\thra[k]}\overline{X}_k\dar{d_i}&X_n\rar{\cong}\lar[swap]{\cong}\dar{d_i}&\bigoplus\limits_{0\in
        J\subseteq[n]}\overline{X}_J\dar{d_i}\\
      \bigoplus\limits_{[n-1]\thra[k]}\overline{X}_k&X_{n-1}\rar{\cong}\lar[swap]{\cong}&\bigoplus\limits_{0\in
        J\subseteq [n]\setminus\{i\}}\overline X_J
    \end{tikzcd}
  \]
  where the right hand square is given in
  \th\ref{prop:cofibre_sequence:abelian-cats} and the left hand square
  corresponds to the standard description of the Dold--Kan nerve, see for
  example Section III.2 in \cite{GJ99}. Although there is a canonical bijection
  between surjective monotone maps $[n]\thra[k]$ and subsets of $[n]$ which
  contain $0$ of cardinality $k+1$ given by
  \[
    (f\colon[n]\thra[k])\mapsto\setP{\min f^{-1}(i)}{i\in[k]},
  \]
  the horizontal composites in the above diagram are complicated to describe; in
  particular, even after identifying the respective labelling sets via the above
  bijection these composites do not simply correspond to permuting summands. The
  advantage of the direct sum decomposition \eqref{eq:0DK-sum-ab} (resp.\
  \eqref{eq:nDK-sum-ab}) is that all but the $0$-th (resp.\ $n$-th) face maps
  correspond to canonical projections onto some direct summands. This is
  well suited for our purposes, since we are mostly interested on the relation
  between the simplicial object and various membrane objects defined by
  simplicial subsets of $\horn n0$ (resp.\ $\horn nn$). Of course, the
  complexity of the simplicial object has not disappeared: the $0$-th (resp.\
  $n$-th) face map and all of the degeneracy maps have non-trivial descriptions
  with respect to this alternative direct sum decomposition.
\end{remark}

\begin{proof}[Proof of \th\ref{prop:cofibre_sequence:abelian-cats}]
  In view of \th\ref{rmk:d0dn-abelian} it is enough to consider the case of
  $0$-th horns.
  By definition, the Segal map $X_n\to X_{\Lambda_0^n}$ fits into a commutative
  diagram
  \[
    \begin{tikzcd}[row sep=small, column sep=large,ampersand replacement=\&]
      0\rar\&\overline{X}_n\rar\&X_n\dar\rar{[d_1\,\cdots\, d_n]^{\top}}\&\bigoplus_{i=1}^n X_{n-1}\dar[equals]\\
      \&0\rar\&X_{\Lambda_0^n}\rar[swap]{[p_1\,\cdots\,p_n]^{\top}}\&\bigoplus_{i=1}^n
      X_{n-1}
    \end{tikzcd}
  \]
  with exact rows. Following the proof of Lemma I.3.4 in \cite{GJ99}, we define
  a map
  \[
    f=f^{(1)}\colon X_{\Lambda_0^n}\lra X_n
  \]
  recursively by setting $f^{(n)}\coloneqq s_{n-1}p_n$ and
  \[
    f^{(i)}=f^{(i+1)}-s_{i-1}d_if^{(i+1)}+s_{i-1}p_i
  \]
  for $1\leq i\leq n$. It is straightforward to verify that $d_if=p_i$ for
  $i\neq 0$ (using the simplicial identities and the fact that for $i<j$ we
  have $d_ip_j=d_{j-1}p_i$). It follows that $f$ is a right inverse to the Segal map $X_n\to
  X_{\Lambda_0^n}$.

  The desired direct sum decomposition of $X_{\Lambda_0^n}$ follows by induction
  on $n$. Indeed, for $n=1$ there is a split short exact sequence
  \[
    \begin{tikzcd}
      0\rar&\overline{X}_1\rar&X_1\rar{d_1}\rar&
      X_0\rar\ar[bend right=50, dashed]{l}[swap]{s_0}&0
    \end{tikzcd}
  \]
  and $X_0=X_{\Lambda_0^1}$; the claims are clear in this case.
  Let $n\geq1$ and suppose that the desired direct sum decomposition has been
  established for all $1\leq k\leq n$. It follows that
  \[
    X_{\Lambda_0^{n+1}}=\lim_{0\in J\subsetneq[n+1]}X_J\cong\lim_{0\in
      J\subsetneq[n+1]}\bigoplus_{0\in I\subseteq
      J}\overline{X}_I\cong\bigoplus_{0\in K\subsetneq [n+1]}\overline{X}_K
  \]
  since, by induction, the limit is taken over the projections onto the
  appropriate direct summands. In particular the maps $p_i\colon
  X_{\Lambda_0^{n+1}}\to X_{n}$, $i\neq0$ can be taken as the projections onto
  the direct sum of those $\overline{X}_J$ such that $J$ does not contain $i$.
  This finishes the proof.
\end{proof}

\begin{remark}
  Following the general spirit of this article, we give a description of the
  splitting map $f\colon \Xcover{\horn n0}\to X_n$
  in terms of cubical diagrams.\footnote{This description is related
  to the projector $\pi\colon X_n\to X_n$ onto the complement of
  $X_{\horn n0}$ defined in Section 2 of \cite{Dyc17}.}

  Given $i<j$, we define two embeddings $z^j_i,u^j_i\colon [1]^{i-1}\to [1]^{j-1}$
  via
  \begin{eqnarray*}
    z^j_i(w_1,\dots,w_{i-1})&\coloneqq&(w_1,\dots,w_{i-1},0,\dots,0)\\
    u^j_i(w_1,\dots,w_{i-1})&\coloneqq&(w_1,\dots,w_{i-1},0,1,\dots,1)
  \end{eqnarray*}
  Fix $n\geq 1$ and define a cube $ q^n\colon [1]^{n-1}\to \Delta(n,n-1)$ by
  \begin{equation*}
  v\mapsto  q^n_v\colon
  \begin{cases}
    0\mapsto 0\\
    i\mapsto i-1+v_i, & \text{for } 1\leq i\leq n-1\\
    n\mapsto n-1\\
  \end{cases}
\end{equation*}
  For every $1\leq i\leq n$ we can then define a map $\iota_i\colon X_{n-1}\to X_n$
  by totalising the action of the $(i-1)$-cube $ q^n\circ z^n_i$; explicitly:
  \begin{equation}
    \label{eq:sumsplittingiotas}
    \iota_i\coloneqq \sum_{w\in [1]^{i-1}}(-1)^{|w|} ( q^n_{z^n_i(w)})^*.
  \end{equation}
  Then we have
  \begin{equation*}
    f = \sum_{i=1}^n(-1)^{i-1}\iota_i p_i.
  \end{equation*}
  The essential relation $d_jf = p_j$ is explained by the following observations:
  \begin{itemize}
  \item For all $i>j$, composing the cube
    $ q\colon[1]^{i-1}\to\Delta(n,n-1)$ with the coface map $d^j\colon
    [n-1]\to\n$ yields a cube $ q\colon[1]^{i-1}\to\Delta(n-1,n-1)$ with two
    identical faces along the $j$-th direction. Therefore $d_j\iota_i$ (and
    hence $d_j\iota_ip_i$) vanishes.
  \item For all $i<j$ we have a commutative diagram
    \begin{equation*}
      \begin{tikzcd}[column sep=large, row sep=small]
        {[1]^{j-1}}\rar{z^n_j} &{[1]^{n-1}}\rar{ q^n}&
        \Delta(n,n-1)\rar{{ - \circ d^j}}&\Delta(n-1,n-1)\\
        {[1]^{i-1}}\uar{u^j_i}\ar[rd, "z^n_i"']\rar{z^{n-1}_i}&{[1]^{n-2}}\rar{ q^{n-1}}&
        \Delta(n-1,n-2)\ar[ru, "{d^{i}\circ - }"']\ar[rd, "{d^{j-1}\circ - }"]\\
        &{[1]^{n-1}}\rar{ q^n} &
        \Delta(n,n-1)\rar{{ - \circ d^j}}&\Delta(n-1,n-1)
      \end{tikzcd}
    \end{equation*}
    so that the relation $d_ip_j=d_{j-1}p_i$ implies that in the sum for $d_jf$
    the summands of
    $d_j\iota_ip_i$ cancel with those summands of $d_j\iota_jp_j$ which are
    indexed by an element of the form $u^j_i(w)\in [1]^{j-1}$ for some $w\in
    [1]^{i-1}$.
  \end{itemize}
   The only summand in the sum for $d_jf$ which does not cancel is the one
   corresponding to $(1,\dots,1)\in [1]^{j-1}$;
   this summand is equal to $d_js_{j-1}p_j=p_j$.
\end{remark}

\begin{proof}[Proof of
  \th\ref{thm:outer-Kan-DK-abelian}]
  Let $\A$ be an abelian category and $m$ a natural number. By definition, a
  simplicial object $X$ in $\A$ is an outer $m$-Kan complex if, for each $n>m$,
  the Segal maps
  \begin{equation*}
    X_n\lra X_{\Lambda_0^n}\qquad\text{and}\qquad X_n\lra X_{\Lambda_n^n}
  \end{equation*}
  are isomorphisms in $\A$. By \th\ref{prop:cofibre_sequence:abelian-cats} the
  latter condition is equivalent to the statement that for each $n>m$ the object
  $\overline{X}_n$ is a zero object of $\A$, where $(\overline{X},\partial)$ is
  the normalised Moore chain complex of $X$. In other words, the simplicial
  object $X$ in $\A$ is an outer $m$-Kan complex if and only if the normalised
  Moore chain complex of $X$ is $m$-truncated, which is what we needed to prove.
\end{proof}

\subsubsection{Classification of $2m$-Segal objects in abelian categories}

The goal of this section is to prove the following characterization of
$2m$-Segal objects in abelian categories.

\begin{theorem}\th\label{thm:abelianSegalch}
  Let $X$ be a simplicial object in an abelian category $\A$ and $m\geq 1$. The
  following are equivalent:
  \begin{enumerate}
  \item\label{inthm:abch2mSegal} The simplicial object $X$ is $2m$-Segal.
  \item\label{inthm:abchoutmKan} The simplicial object $X$ is an outer $m$-Kan
    complex.
  \item\label{inthm:abchmtrunc} The connective chain complex
    $\overline{X}\coloneqq C(X)$ is $m$-truncated.\qedhere
  \end{enumerate}
\end{theorem}
\begin{proof}
  The implications \ref{inthm:abchmtrunc} $\lRa$ \ref{inthm:abchoutmKan} $\lRa$
  \ref{inthm:ch2mSegal} are established in \th\ref{thm:outer-Kan-DK-abelian} and
  \th\ref{prop:outer_horns}. It remains to prove the implication
  $\ref{inthm:abch2mSegal}\lRa\ref{inthm:abchmtrunc}$. Recall from Proposition
  2.10 in \cite{Pog17} that every $2m$-Segal object is $l$-Segal for every
  $l\geq 2m$. In particular, if $X$ is $2m$-Segal then it is also $2k$-Segal for
  all $k\geq m$.

  Suppose that there exists $k>m$ such that $\overline{X}_k$ is non-zero. Let
  $I=\{0,3,\dots,3k\}$; we view $I$ as a subset of $[3k]$. The smallest even
  subset of $[3k]$ that contains $I$ has cardinality $2k+1$. Therefore $I$ is
  not contained in the lower $2k$-Segal cover $\T_{3k,2k}^-$. Now, the object
  $\overline{X}_I$ is a direct summand of
  \[
    X_{3k}\cong \bigoplus_{0\in J\subseteq[3k]}\overline{X}_{J}.
  \]
  From the explicit description of the face maps of $X\cong N(C(X))$ given in
  \th\ref{prop:cofibre_sequence:abelian-cats}, we conclude immediately that
  $\overline{X}_I$ is sent to zero by the Segal maps $X_{3k}\to
  \Xcover{\lowerScovering{3k}{2k}}$. A similar argument shows that $I$ is not
  contained in the upper $2k$-Segal cover $\T_{3k,2k}^+$ and therefore
  $\overline{X}_I$ is sent to zero by the Segal map
  $X_{3k}\to\Xcover{\upperScovering{3k}{2k}}$. In particular these Segal maps
  cannot be equivalences, contradicting the fact that $X$ is $2k$-Segal.
\end{proof}

\subsection{Simplicial objects in stable $\infty$-categories}

\subsubsection{Lurie's $\infty$-categorical Dold--Kan correspondence}

Let $\C$ be a stable $\infty$-category. To extend the classical Dold--Kan
correspondence to the $\infty$-categorical context, one first needs to adapt
the notions of simplicial objects and connective chain complexes. The first
notion is straightforward, as functors $\Delta^{\op}\to\C$ (or
more precisely $\N(\Delta)^{\op}\to\C$) already capture the correct notion; see
Definition 6.1.2.2 in \cite{Lur09}.
In contrast, the naive notion of a connective
chain complex in $\C$, namely a connective chain complex
\begin{equation*}
  \cdots\xrightarrow{d}X_n\xrightarrow{d}\cdots\xrightarrow{d}X_1\xrightarrow{d}X_0\to0\to\cdots
\end{equation*}
in the homotopy category $\hh(\C)$, requires further refinement, see Section
1.2.2 in \cite{Lur17}. This is due to the fact that the space of identifications
$d^2\cong 0$ need not be contractible.

\begin{definition}
  \th\label{def:filtered_object} Let $\C$ be a stable $\infty$-category. A
  \emph{filtered object in $\C$} is a functor $\NNN(\NN)\to\C$, where we view
  $\NN$ as a poset with respect to the usual order.
\end{definition}

\begin{remark}
  \th\label{rmk:filtered_object-complex} Let $\C$ be a stable $\infty$-category.
  Recall that the homotopy category $\hh(\C)$ has the structure of a
  triangulated category whence, in particular, $\hh(\C)$ is an additive
  category, see Section 1.1.2 in \cite{Lur17}. A filtered object $B$ in $\C$ can
  be visualised as a diagram
  \begin{equation*}
    B\colon B_0\xrightarrow{f_1}B_1\xrightarrow{f_2}\cdots\xrightarrow{f_n}B_n\rightarrow\cdots.
  \end{equation*}
  Note that $B$ induces a connective chain complex $\overline{B}$ in $\hh(\C)$
  where $\overline{B_0}\coloneqq B_0$ and
  \begin{equation*}
    \overline{B}_n\coloneqq \Sigma^{-n}\cofib(f_n)
  \end{equation*} for
  $n\geq 1$, see Remark 1.2.2.3 in \cite{Lur17}. The data encoded by
  the filtered object $B=B^{(0)}$ can be visualised as a diagram
  \begin{center}
    \includegraphics{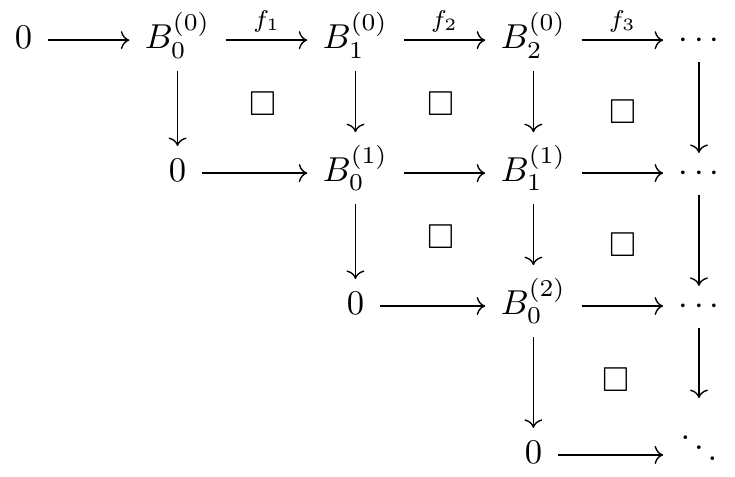}
  \end{center}
  in which all marked squares are biCartesian.
\end{remark}

The introduction of filtered objects in stable $\infty$-categories is justified
by the following $\infty$-categorical version of the Dold--Kan correspondence,
see Theorem 1.2.4.1 in \cite{Lur17}.

\begin{theorem}[Lurie's $\infty$-categorical Dold--Kan correspondence]
  \th\label{thm:Luries_DK} Let $\C$ be a stable $\infty$-category. There is a
  canonical equivalence
  \[
    \Fun(\Delta^{\op},\C)\xra{\simeq}\Fun(\NN,\C)
  \]
  of stable $\infty$-categories.
\end{theorem}

For our purposes it is necessary to recall a particular description of the
equivalence between $\Fun(\Delta^{\op},\C)$ and $\Fun(\NN,\C)$ which is implicit
in Lurie's proof of \th\ref{thm:Luries_DK}. Recall that, for a given integer
$n\geq0$, the full subcategory of $\Delta$ spanned by the finite ordinals
$\setP{[m]}{0\leq m\leq n}$ is denoted by $\Delta_{\leq n}$.

\begin{notation}
  We denote by $\J_n$ the subcategory of the overcategory $(\Delta_{\leq
    n})_{/[n]}$ spanned by the \emph{injective} monotone maps $[m]\to[n]$. The
  composite of canonical functors
  \begin{equation*}
    \NNN(\J_{n})\lra\NNN(\Delta_{\leq n})_{/[n]}\lra\NNN(\Delta_{\leq n})\lra\NNN(\Delta)
  \end{equation*}
  yields a functor $j\colon\NNN(\J_n)\to\NNN(\Delta)$. Note that there is a
  fully faithful functor $i\colon\J_n\hookrightarrow\J_{n+1}$ induced by the
  canonical inclusion $[n]\subset[n+1]$. In the sequel we identify $\J_n$ with
  the poset of \emph{non-empty} subsets of $[n]$. In particular,
  $\J_n^{\triangleleft}$ can be identified with the poset of \emph{all} subsets
  of $[n]$ and is therefore isomorphic to the $(n+1)$-cube $I^{n+1}$.
\end{notation}

\begin{construction}
  \th\label{const:BX} Let $\C$ be a stable $\infty$-category and $X$ a
  simplicial object in $\C$. We associate to $X$ a filtered object $B$ in $\C$
  as follows. Let $n\geq0$ be an integer. We define $B_n$ to be a colimit of the
  punctured $(n+1)$-cube $X|_{\J_n^{\op}}\colon\NNN(\J_n)^{\op}\to\C$ obtained
  from $X$ by restriction along the composite of canonical functors
  \begin{equation*}
      \NNN(\J_n)^{\op}\lra\NNN(\Delta_{\leq n})^{\op}\lra\NNN(\Delta)^{\op}.
  \end{equation*}
  The required maps $f_{n+1}\colon B_n\to B_{n+1}$ are induced by the canonical
  fully faithful functors $i\colon\J_n\hookrightarrow\J_{n+1}$. The fact that
  $B$ is equivalent to the filtered object corresponding to $X$ under the
  equivalence in \th\ref{thm:Luries_DK} is a consequence of Lemma 1.2.4.17 in
  \cite{Lur17} which states that the functor
  $j\colon\NNN(\J_{n})\to\NNN(\Delta_{\leq n})$ is right cofinal.
\end{construction}

\begin{example}
  Let $\C$ be a stable $\infty$-category and $X$ a simplicial object in $\C$.
  According to \th\ref{const:BX}, the terms $B_0$, $B_1$, and $B_2$ of the
  filtered object $B$ in $\C$ associated to $X$ are determined by the
  biCartesian cubes
  \begin{equation*}
    \begin{tikzcd}
      X_{\set{0}}\rar{\sim}&B_0
    \end{tikzcd}\qquad
    \begin{tikzcd}
      X_{\set{0,1}}\rar\dar\popb&X_{\set{1}}\dar\\
      X_{\set{0}}\rar[swap]{f_1}&B_1
    \end{tikzcd}\qquad
    \begin{tikzcd}[column sep=tiny, row sep=tiny]
      X_{\set{0,1,2}}\ar{rr}\ar{dd}\ar{dr}\mycube&&X_{\set{1,2}}\ar{dd}\ar{dr}\\
      & X_{\set{0,2}}\ar{rr}\ar{dd}&&X_{\set{2}}\ar{dd}\\
      X_{\set{0,1}}\ar{rr}\ar{dr}&&X_{\set{1}}\ar{dr}\\
      &X_{\set{0}}\ar{rr}&&B_2
    \end{tikzcd}
  \end{equation*}
  The required maps $f_{n+1}\colon B_n\to B_{n+1}$ are induced by the universal
  property of these diagrams in the obvious way. For example, the map $f_1\colon
  B_0\to B_1$ is indicated above while the map $f_2\colon B_1\to B_2$ is induced
  by the $2$-cube
  \begin{equation*}
    \begin{tikzcd}
      X_{\set{0,1}}\rar\dar&X_{\set{1}}\dar\\
      X_{\set{0}}\rar&B_2
    \end{tikzcd}
  \end{equation*}
\end{example}

\begin{remark}
  \th\label{rmk:d0dn} Let $\C$ be a stable $\infty$-category and $X$ a
  simplicial object in $\C$. Replacing the canonical inclusions
  $[n]\subset[n+1]$, which are given by the co-face maps
  $d^n\colon[n]\hookrightarrow[n+1]$, by the co-face maps $d^0\colon
  [n]\hookrightarrow[n+1]$ in \th\ref{const:BX} gives an equivalent way to
  construct the filtered object $B$ in $\C$ associated to $X$ by the equivalence
  in \th\ref{thm:Luries_DK}, see \th\ref{rmk:d0dn} for comparison.
\end{remark}

\subsubsection{Outer $m$-Kan complexes in stable $\infty$-categories}

\begin{definition}
  \th\label{def:k-truncated_filtered_object} Let $\C$ be a stable
  $\infty$-category and $m$ a natural number. A filtered object $B$ in $\C$ is
  \emph{(strictly) $m$-truncated} if for each $n>m$ the map $f_n\colon
  B_{n-1}\to B_n$ is an equivalence.
\end{definition}

\begin{remark}
  \th\label{rmk:truncated_filtered_vs_cpx} Let $\C$ be a stable
  $\infty$-category and $m$ a natural number. A filtered object $B$ in $\C$ is
  $m$-truncated if and only if the associated connective chain complex
  $\overline{B}$ in the homotopy category $\hh(\C)$ is $m$-truncated in the
  sense of \th\ref{def:k-truncated_complex}, see
  \th\ref{rmk:filtered_object-complex}.
\end{remark}

Our aim is to prove the following theorem, which is an $\infty$-categorical
version of \th\ref{thm:outer-Kan-DK-abelian}.

\begin{theorem}
  \th\label{thm:k-groupoids-k-truncted_complex:infty-cats} Let $\C$ be a stable
  $\infty$-category and $m$ a natural number. Lurie's $\infty$-categorical Dold--Kan correspondence
  \[
    \Fun(\Delta^{\op},\C)\lra\Fun(\NN,\C)
  \]
  restricts to an
  equivalence of $\infty$-categories between the full subcategory of
  $\Fun(\Delta^{\op},\C)$ spanned by the outer $m$-Kan complexes and the full
  subcategory of $\Fun(\NN,\C)$ spanned by the $m$-truncated filtered objects.
\end{theorem}

In fact, \th\ref{thm:k-groupoids-k-truncted_complex:infty-cats} is a consequence
of the following general statement which is an analogue of
\th\ref{prop:cofibre_sequence:abelian-cats} for stable $\infty$-categories.

\begin{proposition}
  \th\label{prop:fibre_sequence:infty-cats} Let $\C$ be a stable
  $\infty$-category, $X$ a simplicial object in $\C$, and $B=B_X$ the filtered
  object in $\C$ associated to $X$ under the equivalence in
  \th\ref{thm:Luries_DK}.
  \begin{enumerate}
  \item\label{it:fibre_sequence} For each $n\geq 1$ the Segal maps
    $X_n\to\Xcover{\horn n0}$ and $X_n\to\Xcover{\horn nn}$ fit into fibre sequences
    \begin{equation*}
      \begin{tikzcd}
        \overline{X}_n\rar\dar\popb&X_n\dar\\
        0\rar&X_{\Lambda_0^n}
      \end{tikzcd}
      \qquad\text{and}\qquad
      \begin{tikzcd}
        \overline{X}_n\rar\dar\popb&X_n\dar\\
        0\rar&X_{\Lambda_n^n}
      \end{tikzcd}
    \end{equation*}
    in $\C$, where $\overline{X}_n\coloneqq
    \Sigma^{-n}(\cofib({B_{n-1}\xrightarrow{f_n}B_n}))$ is the $n$-th term of
    the connective chain complex $\overline{X}$ in $\hh(\C)$ induced by $B$.
  \item\label{it:dk_detailed} For a non-empty subset $J\subseteq[n]$ we write
    $\overline{X}_{J}\coloneqq\overline{X}_{|J|-1}$.
    \begin{enumerate}
    \item There is a direct sum decomposition
      \begin{equation}
        \label{eq:0DK-sum-st}
        X_{\Lambda_0^n}\simeq\bigoplus_{0\in J\subsetneq[n]}\overline{X}_J,
      \end{equation}
      with the following property: For each $i\neq 0$, the action of the
      $i$-th face map on $X_n\simeq\overline{X}_n\oplus X_{\Lambda_0}^n$ corresponds
      to the projection onto those direct summands $\overline{X}_J$ such that
      $J$ does not contain $i$.
    \item There is a direct sum decomposition
      \begin{equation}
        \label{eq:nDK-sum-st}
        X_{\Lambda_n^n}\simeq\bigoplus_{n\in J\subsetneq[n]}\overline{X}_J,
      \end{equation}
      with the following property: For each $i\neq n$, the action of the
      $i$-th face map on $X_n\simeq\overline{X}_n\oplus X_{\Lambda_n}^n$ corresponds
      to the projection onto those direct summands $\overline{X}_J$ such that
      $J$ does not contain $i$.\qedhere
    \end{enumerate}
  \end{enumerate}
\end{proposition}

Before giving a proof of \th\ref{prop:fibre_sequence:infty-cats} we use it to
prove \th\ref{thm:k-groupoids-k-truncted_complex:infty-cats}.

\begin{proof}[Proof of
  \th\ref{thm:k-groupoids-k-truncted_complex:infty-cats}]
  Let $\C$ be a stable $\infty$-category and $m$ a natural number. A simplicial
  object $X$ of $\C$ is an outer $m$-Kan complex if, by definition, for each
  $n>m$ the Segal maps
  \begin{equation*}
    X_n\lra X_{\Lambda_0^n}\qquad\text{and}\qquad X_n\lra X_{\Lambda_n^n}
  \end{equation*}
  are equivalences in $\C$. By \th\ref{prop:fibre_sequence:infty-cats} the
  latter condition is equivalent to the statement that for each $n>m$ the object
  $\overline{X}_n=\Sigma^{-n}\cofib(B_{n-1}\to B_n)$ is a zero object in $\C$,
  where $B$ is the filtered object of $\C$ associated to $X$ under the
  equivalence in \th\ref{thm:Luries_DK}. In other words, the simplicial object
  $X$ is an outer $m$-Kan complex if and only if for each $n>m$ the map
  $B_{n-1}\to B_n$ is an equivalence in $\C$ if and only if the filtered object
  $B$ of $\C$ is $m$-truncated, see \th\ref{rmk:truncated_filtered_vs_cpx}.
\end{proof}

We now give a proof of statement \ref{it:fibre_sequence} in
\th\ref{prop:fibre_sequence:infty-cats}.

\begin{proof}[Proof of statement \ref{it:fibre_sequence} in
  \th\ref{prop:fibre_sequence:infty-cats}] Let $X$ be a simplicial object in
  $\C$ and $B$ the filtered object in $\C$ corresponding to $X$ under the
  equivalence in \th\ref{thm:Luries_DK}. Fix an integer $n\geq1$. Consider the
  following auxiliary diagrams.
  \begin{itemize}
  \item Let $\widetilde{X}$ be an $(n+1)$-cube obtained from $X|_{\J_n^{\op}}$ by
    left Kan extension to $(\J_n^{\triangleleft})^{\op}$. Thus $\widetilde{X}$ is
    a coCartesian $(n+1)$-cube in $\C$ and $\widetilde{X}_{\emptyset}\simeq B_n$.
  \item Let $\widetilde{X}|_{0\in J}$ be the $n$-cube obtained from
    $\widetilde{X}$ by restricting to the subsets of $[n]$ which contain $0$.
    Note that $\widetilde{X}|_{0\in J}=X|_{0\in J}$.
  \item Similarly, let $\widetilde{X}|_{0\notin J}$ be the $n$-cube obtained
    from $\widetilde{X}$ by restricting to the subsets of $[n]$ which do not
    contain $0$.
  \end{itemize}
  The $n$-cubes $\widetilde{X}|_{0\in J}$ and
  $\widetilde{X}|_{0\notin J}$ are opposite facets of the $(n+1)$-cube
  $\widetilde{X}$. Since $\widetilde{X}$ is coCartesian,
  \th\ref{coro:tcofib-tfib} implies the existence of an equivalence
  \begin{equation*}
    \tfib(\widetilde{X}|_{0\in J})\simeq \Sigma^{-n}\tcofib(\widetilde{X}|_{0\notin J}).
  \end{equation*}
  Moreover, by definition
  \begin{equation*}
    X_{\Lambda_0^n}=\lim_{0\in J\subsetneq[n]}X_ J.
  \end{equation*}
  Therefore
  \begin{equation*}
    \tfib(\widetilde{X}|_{0\in J})=\fib(X_n\to X_{\Lambda_0^n}).
  \end{equation*}
  Furthermore, in view of \th\ref{rmk:d0dn} we can identify the map $f_n\colon
  B_{n-1}\to B_n$ with the canonical map
  \begin{equation*}
    \colim_{0\notin J\neq\emptyset}\widetilde{X}_ J\to \widetilde{X}_{\emptyset}.
  \end{equation*}
  In particular
  \begin{equation*}
    \tcofib(\widetilde{X}|_{0\notin J})=\cofib(B_{n-1}\xrightarrow{f_n}B_n).
  \end{equation*}
  Piecing together the above equivalences, we conclude that there is an
  equivalence
  \begin{equation*}
    \overline{X}_n\coloneqq \Sigma^{-n}\cofib(B_{n-1}\xrightarrow{f_n}B_n)\simeq\fib(X_n\to X_{\Lambda_0^n}),
  \end{equation*}
  which is what we needed to show.
\end{proof}

We delay the proof of \ref{it:dk_detailed} in
\th\ref{prop:fibre_sequence:infty-cats} to establish some notation.

\begin{notation}
  Let $\C$ be a stable $\infty$-category and $X$ a simplicial object in $\C$.
  Let $\n\in\Delta$.
 We introduce the following auxiliary posets.
  \begin{itemize}
  \item We denote by $\QQo{\horn n0}$ the poset of
    all proper subsets $0\in I\subsetneq\n$ that contain $0$.
  \item Let $\QQ{\horn n0}$ be the right cone over $\QQo{\horn n0}$; we denote the cone point
    by $\horn n0$ so that we can write $I<\horn n0$ for all $I\in\QQo{\horn n0}$.
  \item We denote by $\QQb{\horn n0}$ the poset obtained from $\QQ{\horn n0}$ by adjoining new
    elements $\overline {I'}$ for every element $I'\in\QQo{\horn n0}$; we declare
    $\overline {I'}<I'$ (hence also $\overline I'<\horn n0$ and  $\overline I'<I$ for
    every $I'\subseteq I\in\QQo{\horn n0}$).
  \item Let $x\in \QQ{\horn n0}$ (this includes the case $x=\horn n0$ and $x=I$
    for $0\in I\subsetneq \n$). We denote by
    $\QQu{x}\subset\QQb{\horn n0}$ the subposet defined by
    \[\QQu{x}\coloneqq \{x\}\cup \setP{\overline {I'}}{I'<x}.\]
    Note that $\QQu{x}$ is a cone with vertex $x$ on the discrete (po)set
    $\setP{\overline {I'}}{0\in I'\subsetneq \n}$ (for $x=\horn n0$)
    or $\setP{\overline{I'}}{0\in I'\subseteq I}$ (for $0\in
    x=I\subsetneq \n$).\qedhere
\end{itemize}
\end{notation}

\begin{proof}[Proof of statement \ref{it:dk_detailed} in
  \th\ref{prop:fibre_sequence:infty-cats}.]
  Let $X$ be a simplicial object in a stable $\infty$-category $\C$. We argue by induction on $n$. In the base case $n=0$ there is nothing to show.
  For the rest of the proof, fix a value for $n\geq 1$ and assume the result is
  proven for all smaller values.

  The simplicial object $X$
  induces a diagram $X\colon \QQoop{\horn n 0}\to \C$ via the
  obvious functor $\QQo{\horn n 0}\to \Delta$. By right Kan extension we
  extend this to a diagram $\widetilde X\colon \QQop{\horn n0}\to\C$ so that
  \begin{enumerate}[label=(\roman*)]
  \item the diagram $\widetilde X\colon\QQop{\horn n0}\to\C$ is a limit
    cone;\label{inproof:hornlimitdiagram}
  \end{enumerate}
  the value of $\widetilde X$ at the cone point of $\QQ{\horn n 0}$ is
  by definition the membrane space $\Xcover{\horn n 0}$.
  We can extend $\widetilde X$ further to a diagram $\widetilde X\colon \QQbop\n\to\C$ by putting
  $\widetilde X(\overline{I})\coloneqq \overline X_{I}$
  for each $0\in I\subsetneq \n$ and by declaring the map
  \[X_I\simeq \bigoplus_{0\in I'\subseteq I}\overline X_{I'}\lra
    \overline X_I\]
  corresponding to $\overline I<I$ to be the projection onto
  the ``top summand'' of the direct sum decomposition which exists by
  induction. By construction,
  \begin{enumerate}[resume,label=(\roman*)]
  \item for each $0\in I\subsetneq\n$ the restriction of $\widetilde {X}$ to the cone
    $\QQu{I}$ is a product cone in $\C$. \label{inproof:smallproductcones}
  \end{enumerate}
  By the pointwise formula (and using the fact that $\QQ{\horn n0}$ is cofinal
  in $\QQb{\horn n0}$),
  \ref{inproof:hornlimitdiagram} and \ref{inproof:smallproductcones} imply that
  \begin{enumerate}[resume,label=(\roman*)]
  \item \label{inproof:Kanextension} $\widetilde X\colon\QQbop{\horn n0}\to \C$ is a right
    Kan extension of its restriction to the discrete poset
    $\setP{\overline I}{0\in I\subsetneq \n}$.
  \end{enumerate}
  which implies that
  \begin{enumerate}[resume,label=(\roman*)]
  \item the restriction of $\widetilde{X}$ to $\QQu{\horn n0}$ is a product cone
    in $\C$.\label{inproof:bigproductcone}
  \end{enumerate}
  We conclude that $\restr{\widetilde X}{\QQu{\horn n0}}$ exhibits a direct sum decomposition
  \[\Xcover{\horn n0}\xra{\simeq} \bigoplus_{0\in I\subsetneq \n}\overline
    X_I.\]
  Furthermore, we can identify the structure maps $p_i\colon\Xcover{\horn n0}
  \to X_{n-1}$ (which are induced by the relation $(\n\setminus \{i\})<\horn n0$
  in $\QQ{\horn n0}$) with the projection onto those summands $\overline X_I$
  where $i\notin I$.
  Since products agree with homotopy products (see Example 1.2.13.1 in \cite{Lur09}), statements
  \ref{inproof:smallproductcones} and \ref{inproof:bigproductcone} remain true
  in the homotopy category. We have the implication
  \ref{inproof:bigproductcone}\&\ref{inproof:smallproductcones} $\lLRa$
  \ref{inproof:Kanextension} $\lRa$ \ref{inproof:hornlimitdiagram} also in the
  homotopy category; hence we conclude that $\restr{\widetilde X}{\QQ{\horn n0}}$ is still a limit cone
  in the homotopy category $\C$.

  We define a map $f\colon\Xcover{\horn n0}\to X_n$ using the formulas
  from the proof of \th\ref{prop:cofibre_sequence:abelian-cats}.
  The same calculation (which uses only the simplicial identities and the defining
  properties of the projections $p_i\colon\Xcover{\horn n0}\to X_{n-1}$) shows
  that for each $i=1,\dots,n$, the composition
  \[\Xcover{\horn n0} \xra{f} X_n\lra\Xcover{\horn n0}\xra{p_i} X_{n-1}\]
  is equal in the homotopy category to the structure map
  $p_i\colon\Xcover{\horn n0}\to X_{n-1}$ of the cone $\restr{\widetilde X}{\QQ{\horn n0}}$.
  Since we have established this cone to be a limit cone in the homotopy
  category, it follows that the composition
  \[\Xcover{\horn n0} \xra{f} X_n\lra\Xcover{\horn n0}\]
  is equal to the identity in $\hh{C}$. We conclude that the fibre sequences of
  statement \ref{it:fibre_sequence} in \th\ref{prop:fibre_sequence:infty-cats}
  is split and exhibits $X_n$ as the direct sum of $\overline X_n$ and
  $\Xcover{\horn n0}$. The result follows.
\end{proof}

\subsubsection{Classification of $2m$-Segal objects in stable
  $\infty$-categories}

We are ready to establish the following characterisation of $2m$-Segal objects
in stable $\infty$-categories, in analogy to \th\ref{thm:abelianSegalch} which
deals with simplicial objects in abelian categories.

\begin{theorem}\th\label{thm:stableSegalch}
  Let $X$ be a simplicial object in a stable $\infty$-category $\C$ and $m\geq
  1$. The following statements are equivalent:
  \begin{enumerate}
  \item\label{inthm:ch2mSegal} The simplicial object $X$ is $2m$-Segal.
  \item\label{inthm:choutmKan} The simplicial object $X$ is an outer $m$-Kan
    complex.
  \item\label{inthm:chmtrunc} Thee filtered object $B=B_X$ in $\C$ associated to
    $X$ under the equivalence in \th\ref{thm:Luries_DK} is $m$-truncated.\qedhere
  \end{enumerate}
\end{theorem}
\begin{proof}
  The implications \ref{inthm:chmtrunc} $\lRa$ \ref{inthm:choutmKan} $\lRa$
  \ref{inthm:ch2mSegal} follow from
  \th\ref{thm:k-groupoids-k-truncted_complex:infty-cats} and
  \th\ref{prop:outer_horns}. In view of the explicit description of the
  simplicial object $X$ given in statement \ref{it:dk_detailed} in
  \th\ref{prop:fibre_sequence:infty-cats}, the implication \ref{inthm:ch2mSegal}
  $\lRa$ \ref{inthm:chmtrunc} can be shown to hold exactly as in
  \th\ref{thm:abelianSegalch}. Indeed, the proof in \th\ref{thm:abelianSegalch}
  does not use the $0$-th face map of the simplicial object.
\end{proof}

\appendix

\section{$n$-cubes in stable $\infty$-categories}
\label{sec:n-cubes}

The idea of studying (co)Cartesian cubes and the interplay between total and
iterated fibres goes back at least to Goodwillie \cite{Goo92}. A systematic
study of such cubes in the stable context was carried out by Beckert and Groth
in \cite{BG18} in the related framework of stable derivators.

\subsection{coCartesian $n$-cubes in stable $\infty$-categories}

Let $n\geq0$ be an integer. The \emph{$n$-cube} is the poset
\begin{equation}
  \label{eq:n-cube}
  I^n\coloneqq \underbrace{[1]\times\cdots\times[1]}_{n\text{ times}}.
\end{equation}
For $v\in I^n$ we define $|v|\coloneqq v_1+\cdots+v_n$. In particular $I^0$ consists of
a single vertex $\emptyset$ with $|\emptyset|=0$.

\begin{definition}
  \th\label{def:n-cube} Let $\C$ be an $\infty$-category and $n\geq0$ an
  integer. An \emph{$n$-cube in $\C$} is an object of the $\infty$-category
  $\Fun(I^n,\C)$, that is a functor $X\colon\NNN(I^n)\to\C$.
\end{definition}

\begin{remark}
  Let $n\geq0$ be an integer and $J$ a finite set of cardinality $n$. Each
  bijection $j\colon\set{1,\dots,n}\to J$ yields an isomorphism between $I^n$
  and the poset $2^J$ of subsets of $J$ given by associating to $v\in I^n$ the
  subset $j_v$ of $J$ whose characteristic function is $v$. In specific contexts
  it is often more convenient to consider \emph{$J$-cubes}, that is functors
  $\NNN(2^J)\to\C$, instead of $n$-cubes in the sense of \th\ref{def:n-cube}.
\end{remark}

We are interested in the following exactness conditions on $n$-cubes, see for
example Definition 6.1.1.2 in \cite{Lur17}.

\begin{definition}
  \th\label{def:n-cube_exactness_conditions} Let $\C$ be an $\infty$-category
  and $n\geq0$ an integer.
  \begin{enumerate}
  \item An $n$-cube $X$ in $\C$ is \emph{Cartesian} if the canonical map
    \begin{equation}
      X_{0\cdots0}\lra\lim_{0<|v|} X_v
    \end{equation}
    is an equivalence in $\C$.
  \item An $n$-cube $X$ in $\C$ is \emph{coCartesian} if the canonical map
    \begin{equation}
      \colim_{|v|<n} X_v\lra X_{1\cdots1}
    \end{equation}
    is an equivalence in $\C$.
  \item An $n$-cube $X$ in $\C$ is \emph{biCartesian} if it is both Cartesian
    and coCartesian.\qedhere
  \end{enumerate}
\end{definition}

We are chiefly interested in the properties of $n$-cubes in \emph{stable}
$\infty$-categories. The following result, Proposition 1.2.4.13 in \cite{Lur17},
justifies our focus on coCartesian $n$-cubes in what follows.

\begin{proposition}
  \th\label{prop:n-cube_biCartesian} Let $\C$ be a stable $\infty$-category and
  $n\geq0$ an integer. An $n$-cube $X$ in $\C$ is Cartesian if and only if it is
  coCartesian if and only if it is biCartesian.
\end{proposition}

Let us illustrate \th\ref{prop:n-cube_biCartesian} in low dimensions.

\begin{example}
  Let $\C$ be a stable $\infty$-category, $n\geq0$ an integer, and $X$ an
  $n$-cube in $\C$.
  \begin{enumerate}
  \item If $n=0$, then $X$ can be identified with an object $X_\emptyset$ of
    $\C$. In this case $X$ is coCartesian if and only if $X_\emptyset$ a zero
    object of $\C$.
  \item If $n=1$, then $X$ can be identified with a map $X_0\to X_1$ in $\C$. In
    this case $X$ is coCartesian if and only if the map $X_0\to X_1$ is an
    equivalence in $\C$.
  \item If $n=2$, then $X$ classifies a (coherent) commutative square
    \begin{equation*}
      \begin{tikzcd}
        X_{00}\rar\dar&X_{01}\dar\\
        X_{10}\rar&X_{11}
      \end{tikzcd}
    \end{equation*}
    in $\C$. In this case $X$ is coCartesian if and only if the above diagram is
    a biCartesian square.\qedhere
  \end{enumerate}
\end{example}

\begin{notation}
  Let $\C$ be a stable $\infty$-category and suppose given a $2$-cube in $\C$ of
  the form
  \begin{equation*}
    \begin{tikzcd}
      X_{00}\rar{f}\dar&X_{01}\dar\\
      X_{10}\rar{g}&X_{11}
    \end{tikzcd}
  \end{equation*}
  Functoriality of the cofibre induces a map $\cofib(f)\to\cofib(g)$ which can
  be subsequently identified with a $1$-cube in $\C$. More generally, the
  isomorphism $I^{n+1}\cong [1]\times I^n$ allows us to identify an $(n+1)$-cube
  $X$ in $\C$ with a morphism $X|_{\set{0}\times I^n}\to X|_{\set{1}\times I^n}$
  in the stable $\infty$-category $\Fun(I^n,\C)$. In particular $X$ has an
  associated cofibre
  \begin{equation}
    \label{eq:cofibre-X}
    \cofib(X)\coloneqq \cofib\left(X|_{\set{0}\times I^n}\to X|_{\set{1}\times I^n}\right)
  \end{equation}
  which itself is an object of $\Fun(I^n,\C)$, that is an $n$-cube in $\C$.
\end{notation}

The following statement is a special case of Lemma 1.2.4.15 in \cite{Lur17}. It
provides an inductive characterisation of coCartesian $n$-cubes in stable
$\infty$-categories.

\begin{proposition}
  \th\label{prop:k-cube_cofibre-criterion} Let $\C$ be a stable
  $\infty$-category and $n\geq0$ an integer. Let $X$ be an $(n+1)$-cube in $\C$
  which we identify with a morphism $X|_{\set{0}\times I^n}\to X|_{\set{1}\times
    I^n}$ in the stable $\infty$-category $\Fun(I^n,\C)$. Then, $X$ is a
  coCartesian $(n+1)$-cube if and only if its cofibre $\cofib(X)$ is a
  coCartesian $n$-cube.
\end{proposition}
\begin{proof}
  Since stable $\infty$-categories admit all finite colimits (see Proposition
  1.1.3.4 in \cite{Lur17}), the claim follows by applying Lemma 1.2.4.15 in
  \cite{Lur17} in the case $K=I^n$.
\end{proof}

We illustrate \th\ref{prop:k-cube_cofibre-criterion} with a simple but important
example.

\begin{example}
  \th\label{ex:n-suspension} Let $\C$ be a stable $\infty$-category and $x$ an
  object of $\C$. By definition, the suspension of $x$ is characterised by the
  existence of a coCartesian square
  \begin{equation*}
    \begin{tikzcd}
      x\rar\dar\popb&0\dar\\
      0\rar&\Sigma(x)
    \end{tikzcd}
  \end{equation*}
  Consider now a coCartesian $3$-cube in $\C$ of the form
  \begin{equation*}
    \begin{tikzcd}[column sep=normal, row sep=small]
      x\ar{rr}\drar\ar{dd}\mycube&&0\drar\ar{dd}\\
      &0\ar[crossing over]{rr}&&0\ar{dd}\\
      0\ar{rr}\drar&&0\drar\\
      &0\ar{rr}\ar[crossing over, leftarrow]{uu}&&y
    \end{tikzcd}
  \end{equation*}
  According to \th\ref{prop:k-cube_cofibre-criterion}, taking point-wise
  cofibres of the above $3$-cube yields a coCartesian square
  \begin{equation*}
    \begin{tikzcd}
      \Sigma(x)\rar\dar\popb&0\dar\\
      0\rar&y
    \end{tikzcd}
  \end{equation*}
  We conclude that there is an equivalence $y\simeq\Sigma^2(x)$ in $\C$. More
  generally, if $n\geq1$ and $X$ is a coCartesian $(n+1)$-cube in $\C$ such that
  for each $v\in I^{n+1}$ with $0<|v|<n+1$ the object $X_v$ is a zero object of
  $\C$, then there is an equivalence $X_{1\cdots1}\simeq\Sigma^n(X_{0\cdots0})$
  in $\C$.
\end{example}

As an immediate consequence of \th\ref{prop:k-cube_cofibre-criterion} we obtain
an inductive criterion to verify whether an $n$-cube in a stable
$\infty$-category is coCartesian.

\begin{corollary}
  \th\label{coro:n-cofib} Let $\C$ be a stable $\infty$-category and $n\geq1$ an
  integer. An $n$-cube $X$ in $\C$ is coCartesian if and only if
  \begin{equation}
    \cofib^n(X)\coloneqq \underbrace{\cofib(\cofib(\cdots\cofib}_{n\text{ times}}(X)))
  \end{equation}
  is a zero object of $\C$.
\end{corollary}
\begin{proof}
  The claim follows by iterated application of
  \th\ref{prop:k-cube_cofibre-criterion} since a $0$-cube in $\C$, that is an
  object of $\C$, is coCartesian if and only if it is a zero object of $\C$.
\end{proof}

\begin{remark}
  Let $\C$ be a stable $\infty$-category, $n\geq1$ an integer, and $X$ an
  $n$-cube in $\C$. The object $\cofib^n(X)$ appearing in \th\ref{coro:n-cofib}
  is sometimes called the `iterated cofibre of $X$', see for example
  \cite{BG18}.
\end{remark}

We illustrate \th\ref{coro:n-cofib} with a simple observation.

\begin{corollary}
  \th\label{coro:faces} Let $\C$ be a stable $\infty$-category and $n\geq1$ an
  integer. Let $X$ be an $(n+1)$-cube in $\C$. If two out of the three cubical
  diagrams $X$, $X|_{\set{0}\times I^n}$, and $X|_{\set{1}\times I^n}$ are
  coCartesian, then so is the third.
\end{corollary}
\begin{proof}
  The claim follows by applying \th\ref{coro:n-cofib} to the cofibre sequence
  \begin{equation*}
    \cofib^n(X|_{\set{0}\times I^n})\lra\cofib^n(X|_{\set{1}\times I^n})\lra\cofib^{n+1}(X),
  \end{equation*}
  of objects in $\C$, keeping in mind the elementary fact that if two out of the
  three objects in a cofibre sequence in a stable $\infty$-category are zero
  objects, then so is the third.
\end{proof}

\begin{corollary}
  \th\label{coro:coCartesian_stable_subcategory} Let $\C$ be a stable
  $\infty$-category and $n\geq1$ an integer. The subcategory
  $\Fun^{\Ex}(I^n,\C)$ of the stable $\infty$-category $\Fun(I^n,\C)$ spanned by
  the coCartesian $n$-cubes in $\C$ is closed under fibres and cofibres. In
  particular $\Fun^{\Ex}(I^n,\C)$ is a stable $\infty$-category.
\end{corollary}
\begin{proof}
  Recall that an $n$-cube in $\C$ is coCartesian if and only if it is
  biCartesian, see \th\ref{prop:n-cube_biCartesian}. Let $X\colon [1]\times
  I^n\to\C$ be an $(n+1)$-cube which we identify with a morphism in
  $\Fun(I^n,\C)$. By \th\ref{coro:faces}, if the $n$-cubes $X|_{\set{0}\times
    I^n}$ and $X|_{\set{1}\times I^n}$ are coCartesian, then $X$ is also
  coCartesian. Finally, \th\ref{prop:k-cube_cofibre-criterion} and its dual show
  that both the cofibre (resp.\ the fibre) of $X$, taken in the stable
  $\infty$-category $\Fun(I^n,\C)$, is coCartesian (resp.\ Cartesian). The claim
  follows.
\end{proof}

As a further application of \th\ref{prop:k-cube_cofibre-criterion} we sketch a
proof of the `pasting lemma' for coCartesian $n$-cubes in stable
$\infty$-categories.

\begin{corollary}
  \th\label{coro:pasting_lemma_cubes} Let $\C$ be a stable $\infty$-category and
  $n\geq1$ an integer. Let $X\colon I^n\times[2]\to\C$ be a diagram in $\C$.
  Suppose that the $(n+1)$-cube $X|_{I^n\times\set{0,1}}$ is coCartesian. Then,
  the $(n+1)$-cube $X|_{I^n\times\set{1,2}}$ is coCartesian if and only if the
  $(n+1)$-cube $X|_{I^n\times\set{0,2}}$ is coCartesian.
\end{corollary}
\begin{proof}
  The case $n=1$ is an $\infty$-categorical version of the classical `pasting
  lemma' for coCartesian squares and is proven in Lemma 4.4.2.1 in \cite{Lur09}.
  The general case can be reduced to the case $n=1$ by iterated application of
  \th\ref{prop:k-cube_cofibre-criterion}. We leave the details to the reader.
\end{proof}

\begin{example}
  Let $\C$ be a stable $\infty$-category. A diagram $X\colon I^2\times[2]\to\C$
  can be visualised as two $3$-cubes
  \begin{equation*}
    \begin{tikzcd}[column sep=tiny, row sep=tiny]
      X_{000}\ar{rr}\drar\ar{dd}&&X_{001}\ar{rr}\drar\ar{dd}&&X_{002}\drar\ar{dd}\\
      &X_{010}\ar[crossing over]{rr}&&X_{011}\ar[crossing over]{rr}&&X_{012}\\
      X_{100}\ar{rr}\drar&&X_{101}\ar{rr}\drar&&X_{102}\drar\\
      &X_{110}\ar{rr}\ar[crossing
      over,leftarrow]{uu}&&X_{111}\ar{rr}\ar[crossing
      over,leftarrow]{uu}&&X_{112}\ar[leftarrow]{uu}
    \end{tikzcd}
  \end{equation*}
  glued at a common $2$-cube. Provided that the leftmost $3$-cube is
  coCartesian, \th\ref{coro:pasting_lemma_cubes} states that the rightmost
  $3$-cube is coCartesian if and only if the composite $3$-cube
  \begin{equation*}
    \begin{tikzcd}[column sep=tiny, row sep=tiny]
      X_{000}\ar{rr}\drar\ar{dd}&&X_{002}\ar{dd}\drar\\
      &X_{010}\ar[crossing over]{rr}&&X_{012}\\
      X_{100}\ar{rr}\drar&&X_{102}\drar\\
      &X_{110}\ar{rr}\ar[crossing over,leftarrow]{uu}&&X_{112}\ar[crossing
      over,leftarrow]{uu}
    \end{tikzcd}
  \end{equation*}
  is coCartesian.
\end{example}

\subsection{The total cofibre of an $n$-cube in a stable $\infty$-category}

Recall the inductive criterion to verify whether an $n$-cube in a stable
$\infty$-category is coCartesian provided by \th\ref{coro:n-cofib}. In view of
\th\ref{def:n-cube_exactness_conditions} and the aforementioned criterion, the
following definition is rather natural.

\begin{definition}
  \th\label{def:total_cofibre} Let $\C$ be a stable $\infty$-category, $n\geq0$
  an integer, and $X$ an $n$-cube in $\C$.
  \begin{enumerate}
  \item The \emph{total cofibre of $X$}, denoted by $\tcofib(X)$, is defined as
    the cofibre of the canonical map $\colim_{|v|<n}X_v\to X_{1\cdots1}$.
  \item Dually, the \emph{total fibre of $X$}, denoted by $\tfib(X)$, is defined
    as the fibre of the canonical map $X_{0\cdots0}\to\lim_{0<|v|}X_v$.
  \end{enumerate}
  Note that if $n=0$, then there are equivalences $\tcofib(X)\simeq X_\emptyset$
  and $\tfib(X)\simeq X_\emptyset$.
\end{definition}

\begin{remark}
  The total cofibre of an $n$-cube is investigated in \cite{BG18} in the related
  framework of (stable) derivators, albeit with a slightly different focus.
\end{remark}

The total cofibre of an $n$-cube in a stable $\infty$-category behaves much like
the cofibre of a morphism. We begin to construct this analogy with an elementary
observation.

\begin{lemma}
  \th\label{lemma:t-cofib} Let $\C$ be a stable $\infty$-category and $n\geq0$
  an integer. An $n$-cube $X$ in $\C$ is coCartesian if and only if $\tcofib(X)$
  is a zero object of $\C$.
\end{lemma}
\begin{proof}
  The claim follows immediately from \th\ref{def:n-cube_exactness_conditions}
  and the definition of the total cofibre of $X$, keeping in mind that a
  morphism in a stable $\infty$-category is an equivalence if and only if its
  cofibre is a zero object of $\C$.
\end{proof}

Let $\C$ be a stable $\infty$-category. By definition, the cofibre of a morphism
$f\colon x\to y$ in $\C$ is characterised by the existence of a coCartesian
square of the form
\begin{equation*}
  \begin{tikzcd}
    x\rar{f}\dar\popb&y\dar\\
    0\rar&\cofib(f)
  \end{tikzcd}
\end{equation*}
The total cofibre of an $n$-cube in a stable $\infty$-category is characterised
by an analogous universal property. To prove this we need the following
auxiliary lemma which is a special case of Corollary 4.2.3.10 in \cite{Lur09}.

\begin{lemma}
  \th\label{lemma:decomposition} Let $\P$ be a poset and $\setP{\P_I\subset
    \P}{I\in\J}$ a collection of subposets of $\P$ such that
  $\P=\bigcup_{I\in\J} \P_I$. We view $\J$ as a poset with respect to the
  partial order induced by $\P$, that is $I\leq J$ if $\P_I\subseteq\P_J$.
  Suppose that for each finite chain $\sigma=\set{\sigma_0\leq
    \sigma_1\leq\cdots\leq \sigma_m}$ in $\P$ the poset
  \begin{equation*}
    \J_\sigma\coloneqq \setP{I\in\J}{\sigma\subseteq\P_I}
  \end{equation*}
  has a contractible nerve. Let $\C$ be an $\infty$-category which admits
  colimits of shape $\P_I$ for each $I\in\J$. Then, for each functor
  $p\colon\NNN(\P)\to\C$, colimits of $p$ can be identified with colimits of a
  diagram $q\colon\NNN(\J)\to\C$ characterised by the following properties:
  \begin{enumerate}
  \item The object $q_I$ of $\C$ is a colimit of the diagram
    $p|_{\P_I}\colon\P_I\to\C$.
  \item For each $I,J\in\J$ such that $I\leq J$ the map $q_I\to q_J$ is induced
    by the inclusion $\P_I\subseteq\P_J$ via the universal property of $q_I$.\qedhere
  \end{enumerate}

\end{lemma}
\begin{proof}
  The claim follows by applying Corollary 4.2.3.10 in \cite{Lur09} in the case
  $K\coloneqq \NNN(\P)$ to the functor $F\colon\J\to 2^\J$ given by $I\mapsto\P_I$, see
  also Remark 4.2.3.9 in \cite{Lur09}.
\end{proof}

\begin{proposition}
  \th\label{prop:tcofib_characterisation} Let $\C$ be a stable
  $\infty$-category, $n$ a positive integer, and $X$ an $n$-cube in $\C$. Then,
  there exists an $(n+1)$-cube $\widetilde{X}\colon [1]\times I^n\to\C$ with the
  following properties:
  \begin{enumerate}
  \item\label{prop:tcofib_characterisation:it:X} The $n$-cube
    $\widetilde{X}|_{\set{0}\times I^n}$ agrees with $X$.
  \item\label{prop:tcofib_characterisation:it:zeroes} For each $v\in
    \set{1}\times I^n$ with $|v|<n+1$ the object $\widetilde{X}_v$ is a zero
    object of $\C$.
  \item\label{prop:tcofib_characterisation:it:coCartesian} The $(n+1)$-cube
    $\widetilde{X}$ is coCartesian.
  \end{enumerate}
  Moreover, for every $(n+1)$-cube $Y$ in $\C$ satisfying properties
  \ref{prop:tcofib_characterisation:it:X}--\ref{prop:tcofib_characterisation:it:coCartesian}
  there is an equivalence $Y_{1,1\cdots1}\simeq\tcofib(X)$ in $\C$.
\end{proposition}
\begin{proof}
  The existence of an $(n+1)$-cube $\widetilde{X}$ in $\C$ satisfying properties
  \ref{prop:tcofib_characterisation:it:X}--\ref{prop:tcofib_characterisation:it:coCartesian}
  above is immediate from the existence of a zero object and of finite colimits
  in $\C$ (taking into account the obvious fact that there are no morphisms in
  the poset
  $I^{n+1}$ from objects in $\set{1}\times I^n$ to objects in $\set{0}\times
  I^n$).

  We prove the second claim. Let $Y$ be an $(n+1)$-cube in $\C$ satisfying
  properties
  \ref{prop:tcofib_characterisation:it:X}--\ref{prop:tcofib_characterisation:it:coCartesian}.
  Observe that, since the $(n+1)$-cube $Y$ is coCartesian, the object
  $Y_{1,1\cdots1}$ is a colimit of the diagram $Y|_{|v|<n+1}$. We shall use
  \th\ref{lemma:decomposition} to exhibit $Y_{1,1\cdots1}$ as the total cofibre
  of $X$.

  Let $\J$ be the poset of proper subsets of $[1]$ and
  $\P\coloneqq I^{n+1}\setminus\set{(1,\cdots,1)}$. We define an order preserving map
  from $\J$ to the poset of subsets of $\P$ by
  \begin{align*}
    \P_\emptyset&\coloneqq \setP{v\in \set{0}\times I^n}{|v|<n},\\
    \P_{\set{0}}&\coloneqq \set{0}\times I^n,\text{ and}\\
    \P_{\set{1}}&\coloneqq \P_\emptyset\cup\setP{v\in \set{1}\times I^n}{|v|<n+1}.
  \end{align*}
  and note that $\P=\P_{\emptyset}\cup\P_{\set{0}}\cup\P_{\set{1}}$. We claim
  that the collection $\set{\P_{\emptyset},\P_{\set{0}},\P_{\set{1}}}$ satisfies
  the hypothesis of \th\ref{lemma:decomposition}. Indeed, the only subposet of
  $\J$ whose nerve is not contractible is $\J'\coloneqq \set{\set{0},\set{1}}$. But it
  is easy to verify that if a chain $\sigma$ in $\P$ is contained both in
  $\P_{\set{0}}$ and $\P_{\set{1}}$, then $\sigma$ must be contained in
  $\P_{\emptyset}$. We conclude that if $\J'\subseteq\J_\sigma$, then
  $\J_\sigma=\J$ in this case.

  The conclusion of \th\ref{lemma:decomposition} implies that $Y_{1,1\cdots1}$
  is a colimit of the diagram
  \begin{equation*}
    \begin{tikzcd}
      \colim\limits_{v\in \P_\emptyset} Y_v\rar\dar&\colim\limits_{v\in \P_{\set{0}}} Y_v\\
      \colim\limits_{v\in \P_{\set{1}}} Y_v
    \end{tikzcd}
  \end{equation*}
  where the maps are induced by the inclusions $\P_\emptyset\subset\P_{\set{0}}$
  and $\P_\emptyset\subset\P_{\set{0}}$ via the universal property of
  $\colim_{v\in\P_\emptyset}Y_v$. Moreover, by construction $\colim_{v\in
    \P_\emptyset} Y_v=\colim_{|v|<n}X_v$ while
  \begin{equation*}
    \colim_{v\in\P_{\set{0}}}Y_v=Y_{(0,1\cdots1)}=X_{1\cdots1}
  \end{equation*}
  given that $(0,1,\dots,1)$ is the maximal element in $\P_{\set{0}}$. Finally,
  utilising the fact that $Y$ satisfies property
  \ref{prop:tcofib_characterisation:it:X} it is straightforward to verify that
  the colimit of the diagram $Y|_{\P_{\set{1}}}$ is a zero object in $\C$. In
  other words, $Y_{1,1\cdots1}$ is part of a coCartesian square
  \begin{equation*}
    \begin{tikzcd}
      \colim\limits_{|v|<n} X_v\rar\dar\po&X_{1\cdots1}\dar\\
      0\rar&Y_{1,1\cdots1}
    \end{tikzcd}
  \end{equation*}
  in $\C$. This shows that $Y_{1,1\cdots1}$ can be identified with the total
  cofibre of $X$, which is what we needed to prove.
\end{proof}

\begin{example}
  Let $\C$ be a stable $\infty$-category and $X$ a $2$-cube in $\C$.
  \th\ref{prop:tcofib_characterisation} implies that there exists a coCartesian
  $3$-cube of the form
  \begin{equation*}
    \begin{tikzcd}[column sep=small, row sep=tiny,ampersand replacement=\&]
      X_{00}\ar{rr}\ar{dd}\drar\mycube\&\&X_{01}\ar{dd}\drar\\
      \&X_{10}\ar[crossing over]{rr}\&\&X_{11}\ar{dd}\\
      0\ar{rr}\drar\&\&0\drar\\
      \&0\ar{rr}\ar[crossing over, leftarrow]{uu}\&\&\tcofib(X)
    \end{tikzcd}
  \end{equation*}
\end{example}

The total cofibre of an $n$-cube in a stable $\infty$-category admits the
following alternative characterisation.

\begin{proposition}
  \th\label{prop:tcofib-cofibn} Let $\C$ be a stable $\infty$-category, $n\geq1$
  an integer, and $X$ an $n$-cube in $\C$. Then, there exists an equivalence
  $\tcofib(X)\simeq\cofib^n(X)$. In particular, $X$ induces a cofibre sequence
  \begin{equation}
    \label{eq:tcofib_cofibre_sequence}
    \tcofib(X|_{\set{0}\times I^n})\lra\tcofib(X|_{\set{1}\times I^n})\lra\tcofib(X)
  \end{equation}
  in $\C$.
\end{proposition}
\begin{proof}
  By \th\ref{prop:tcofib_characterisation} there exists a coCartesian
  $(n+1)$-cube $\widetilde{X}$ such that $\widetilde{X}|_{\set{0}\times I^n}$
  and for each $v\in\set{1}\times I^n$ with $|v|<n+1$ the object
  $\widetilde{X}_v$ is a zero object of $\C$. Moreover, there is an equivalence
  $\widetilde{X}_{1,1\cdots1}\simeq\tcofib(X)$ in $\C$. Also, by definition there is a
  cofibre sequence
  \begin{equation}
    \label{eq:n-cofib_in_proof}
    \cofib^n(\widetilde{X}|_{\set{0}\times I^n})\lra\cofib^n(\widetilde{X}|_{\set{1}\times I^n})\lra\cofib^{n+1}(\widetilde{X}).
  \end{equation}
  Since the $(n+1)$-cube $\widetilde{X}$ is coCartesian, the object
  $\cofib^{n+1}(\widetilde{X})$ is a zero object of $\C$, see
  \th\ref{coro:n-cofib}. Consequently, given that $\C$ is stable, the map
  \begin{equation*}
    \cofib^n(\widetilde{X}|_{\set{0}\times I^n})\lra\cofib^n(\widetilde{X}|_{\set{1}\times I^n})
  \end{equation*}
  is an equivalence in $\C$. The claim follows since
  $\widetilde{X}|_{\set{0}\times I^n}=X$ by construction and since there are
  equivalences
  \begin{equation*}
    \cofib^n(\widetilde{X}|_{\set{1}\times I^n})\lra\widetilde{X}|_{1\cdots1}\simeq\tcofib(X),
  \end{equation*}
  where the leftmost equivalence can be easily established using the fact for
  each $v\in\set{1}\times I^n$ with $|v|<n+1$ the object $\widetilde{X}_v$ is a
  zero object of $\C$. Finally, the existence of the required cofibre sequence
  \ref{eq:tcofib_cofibre_sequence} follows immediately from the above
  equivalence and the cofibre sequence \ref{eq:n-cofib_in_proof}.
\end{proof}

\begin{remark}
  A version of \th\ref{prop:tcofib-cofibn} in the related framework of pointed
  derivators is proven in Theorem 8.25 in \cite{BG18}.
\end{remark}

Let $\C$ be a stable $\infty$-category and suppose given a coCartesian square
\begin{equation*}
  \begin{tikzcd}
    X_{00}\rar{f}\dar\popb&X_{01}\dar\\
    X_{10}\rar{g}&X_{11}
  \end{tikzcd}
\end{equation*}
in $\C$. \th\ref{prop:n-cube_biCartesian} implies that the induced map
$\cofib(f)\to\cofib(g)$ is an equivalence in $\C$. The following result shows
that there is an analogous relationship between the total cofibres of parallel
facets of a coCartesian $(n+1)$-cube in a stable $\infty$-category.

\begin{corollary}
  \th\label{coro:tcofib_equivalence} Let $\C$ be a stable $\infty$-category,
  $n\geq0$ an integer, and $X$ a coCartesian $(n+1)$-cube in $\C$. Then, the
  induced map $\tcofib(X|_{\set{0}\times I^n})\to\tcofib(X|_{\set{1}\times
    I^n})$ is an equivalence in $\C$.
\end{corollary}
\begin{proof}
  The claim follows immediately from \th\ref{lemma:t-cofib} and the cofibre
  sequence \ref{eq:tcofib_cofibre_sequence}.
\end{proof}

Let $\C$ be a stable $\infty$-category and $f\colon x\to y$ a morphism in $\C$.
Recall that there is an equivalence $\cofib(f)\simeq\Sigma(\fib(x))$, witnessed
by the existence of a diagram
\begin{equation*}
  \begin{tikzcd}
    \fib(f)\rar\dar\drar[phantom,"\square"]&x\rar\dar{f}\drar[phantom,"\square"]&0\dar\\
    0\rar&y\rar&\cofib(f)
  \end{tikzcd}
\end{equation*}
in which each square is biCartesian. The following result shows that the total
fibre and the total cofibre of an $n$-cube in $\C$ satisfy an analogous
relationship.

\begin{proposition}
  \th\label{prop:tcofib_tfib-n-suspension} Let $\C$ be a stable
  $\infty$-category, $n\geq1$ an integer and $X$ an $n$-cube in $\C$. Then,
  there is an equivalence $\tcofib(X)\simeq\Sigma^n\tfib(X)$.
\end{proposition}
\begin{proof}
  Utilising \th\ref{prop:tcofib_characterisation} an its dual we construct an
  auxiliary diagram
  \begin{equation*}
    F\colon I^n\times\set{-1,0,1}\lra\C
  \end{equation*}
  with the following properties:
  \begin{itemize}
  \item The $n$-cube $F|_{I^n\times\set{0}}$ agrees with $X$.
  \item The $(n+1)$-cube $F|_{I^n\times\set{0,1}}$ is obtained from $X$ as in
    \th\ref{prop:tcofib_characterisation}.
  \item The $(n+1)$-cube $F|_{I^n\times\set{-1,0}}$ is obtained from $X$ as in
    the dual of \th\ref{prop:tcofib_characterisation}.
  \end{itemize}
  By construction, there are equivalences
  \begin{equation*}
    F_{(1\cdots1,1)}\simeq\tcofib(X)\qquad\text{and}\qquad F_{(0\cdots0,-1)}\simeq\tfib(X)
  \end{equation*}
  in $\C$. Finally, the composite $(n+1)$-cube $F|_{I^n\times\set{-1,1}}$
  exhibits $F_{(1\cdots1,1)}$ as the $n$-fold suspension of $F_{(0\cdots0,-1)}$,
  see \th\ref{ex:n-suspension}. We deduce the existence of an equivalence
  $\tcofib(X)\simeq\Sigma^n\tfib(X)$ in $\C$, which is what we needed to prove.
\end{proof}

We obtain the following consequence of
\th\ref{coro:tcofib_equivalence,prop:tcofib_tfib-n-suspension}.

\begin{corollary}
  \th\label{coro:tcofib-tfib} Let $\C$ be a stable $\infty$-category, $n\geq1$
  an integer, and $X$ a coCartesian $(n+1)$-cube in $\C$. Then, there exists an
  equivalence $\tcofib(X|_{\set{1}\times
    I^n})\simeq\Sigma^n(\tfib(X|_{\set{0}\times I^n}))$ in $\C$.
\end{corollary}
\begin{proof}
  By \th\ref{prop:tcofib_tfib-n-suspension} applied to the $n$-cube
  $X|_{\set{1}\times I^n}$ and the dual of \th\ref{coro:tcofib_equivalence}
  applied to the $n$-cube $X|_{\set{0}\times I^n}$ there are equivalences
  \begin{equation*}
    \tcofib(X|_{\set{1}\times I^n})\simeq\Sigma^n(\tfib(X|_{\set{1}\times I^n}))\simeq\Sigma^n(\tfib(X|_{\set{0}\times I^n})).
  \end{equation*}
  The claim follows.
\end{proof}

\addcontentsline{toc}{section}{References}
\bibliographystyle{alpha}
\bibliography{library}

\end{document}